\documentclass{amsart}

\usepackage{amsmath}
\usepackage{amssymb}
\usepackage{graphicx}
\usepackage{amsmath,amscd}
\usepackage{tikz}
\usepackage{tikz-cd}
\usetikzlibrary{positioning,arrows}
\usepackage[utf8]{inputenc} 
\usepackage[T1]{fontenc}
\usepackage[small,nohug,heads=vee]{diagrams}
\usepackage{enumitem}
\usepackage{fancyhdr}
\usepackage{hyperref}

\numberwithin{equation}{section}
\theoremstyle{plain}

\newtheorem{lem}[equation]{Lemma}

\newtheorem{thm}[equation]{Theorem}
\newtheorem{cor}[equation]{Corollary}
\newtheorem{que}[equation]{Question}
\newtheorem{ob}[equation]{Observation}

\theoremstyle{definition}
\newtheorem{definition}[equation]{Definition}
\newtheorem{remark}[equation]{Remark}
\newtheorem{example}[equation]{Example}

\newcommand*{\inc}{\ensuremath{\mathcal{I}}}
\newcommand{\out}{\operatorname{Out}}
\newcommand{\aut}{\operatorname{Aut}}
\newcommand{\qi}{\operatorname{QI}}
\newcommand{\isom}{\operatorname{Isom}}
\newcommand{\id}{\operatorname{id}}

\begin{document}

\title[QUASI-ISOMETRIC CLASSIFICATION OF RIGHT-ANGLED ARTIN GROUPS I]{QUASI-ISOMETRIC CLASSIFICATION OF RIGHT-ANGLED ARTIN GROUPS I: THE FINITE OUT CASE}
\author{JINGYIN HUANG}
\address{Courant Institute of Mathematical Science\\
New York University, 251 Mercer Street\\
New York, NY, USA.}

\email{jingyin@cims.nyu.edu}
\begin{abstract}
Let $G$ and $G'$ be two right-angled Artin groups. We show they are quasi-isometric if and only if they are isomorphic, under the assumption that the outer automorphism groups $\out(G)$ and $\out(G')$ are finite. If we only assume $\out(G)$ is finite, then $G'$ is quasi-isometric $G$ if and only if $G'$ is isomorphic to a subgroup of finite index in $G$. In this case, we give an algorithm to determine whether $G$ and $G'$ are quasi-isometric by looking at their defining graphs.
\end{abstract}

\maketitle
\tableofcontents
\setcounter{tocdepth}{2}

\section{Introduction}
\subsection{Backgrounds and Summary of Results}
Given a finite simplicial graph $\Gamma$ with vertex set $\{v_{i}\}_{i\in I}$, the right-angled Artin group (RAAG) with defining graph $\Gamma$, denoted by $G(\Gamma)$, is given by the following presentation:
\begin{center}
\{$v_i$, for $i\in{I}\ |\ [v_i,v_j]=1$ if $v_{i}$ and $v_{j}$ are joined by an edge\}
\end{center}
$\{v_{i}\}_{i\in I}$ is called a \textit{standard generating set} for $G(\Gamma)$ (cf. Section \ref{subsec_raag}).

The class of RAAG's enjoys a balance between simplicity and complexity. On one hand, RAAG's have many nice geometric, combinatorial and group theoretic properties (see \cite{charney2007introduction} for a summary); on the other hand, this class inherits the full complexity of the collection of finite simplicial graphs, and even a single RAAG could have very complicated subgroups (see, for example \cite{bestvina1997morse}). 

In recent years, RAAG's have become important models to understand other unknown groups, either by (virtually) embedding the unknown groups into some RAAG's (such a program is outlined in \cite[Section~6]{wise2009research}, see also the references over there), or by finding embedded copies of RAAG's in the unknown groups (\cite{clay2010geometry,koberda2012right,taylor2013right,kim2013anti,baik2014right}). 

In this paper, we study the asymptotic geometry of RAAG's and classify a particular class of RAAG's by their quasi-isometric types. Previously, the quasi-isometric classification of RAAG's has been done for the following two classes. 
\begin{enumerate}
\item \emph{Tree groups} by Behrstock and Neumann. It is shown in \cite{behrstock2008quasi} that for any two trees $\Gamma_{1}$ and $\Gamma_{2}$ with diameter $\ge 3$, $G(\Gamma_{1})$ and $G(\Gamma_{2})$ are quasi-isometric. Higher dimensional analogs of tree groups are studied in \cite{MR2727658}.
\item \emph{Atomic groups} by Bestvina, Kleiner and Sageev. A RAAG is atomic if its defining graph $\Gamma$ is connected and does not contain valence one vertices, cycles of length $<5$ and separating closed stars. It is shown in \cite{MR2421136} that two atomic RAAG's are quasi-isometric if and only if they are isomorphic.
\end{enumerate}
Note that atomic groups are much more \textquotedblleft rigid\textquotedblright\ than tree groups. We define the \textit{dimension} of $G(\Gamma)$ to be the maximal $n$ such that $G(\Gamma)$ contains a $\Bbb Z^{n}$ subgroup, and it coincides with the cohomological dimension of $G(\Gamma)$. All atomic groups are 2-dimensional, hence it is natural to ask what are higher dimensional RAAG's which satisfy similar rigidity properties as atomic RAAG's. This is the starting point of the current paper.

Since we are looking for RAAG's which are rigid, those ones with small quasi-isometry groups would be reasonable candidates. However, even in the atomic case, the quasi-isometry group $\qi(G(\Gamma))$ is huge (see the discussion of quasi-isometry flexibility in \cite[Section 11]{MR2421136}). Then we turn to the outer automorphism group $\out(G(\Gamma))$ for guidance. 

Now we ask whether those RAAG's with \textquotedblleft small\textquotedblright\ outer automorphism groups are also geometrically rigid in an appropriate sense. Actually, \textquotedblleft small\textquotedblright\ outer automorphism groups and (quasi-isometric or commensurability) rigidity results come together in several other cases, for example, higher rank lattices (\cite{mostow1973strong,kleiner1997rigidity,eskin1997quasi,eskin1998quasi}), mapping class groups (\cite{hamenstaedt2005geometry,MR2928983}), $\out(F_{n})$ (\cite{farb2007commensurations}) etc. Our first result is about the quasi-isometric classification for RAAG's with finite outer automorphism group.

\begin{thm}
\label{1.1}
Pick $G(\Gamma_{1})$ and $G(\Gamma_{2})$ such that $\out(G(\Gamma_{i}))$ is finite for $i=1,2$. Then they are quasi-isomeric if and only if they are isomorphic.
\end{thm}

This theorem is proved in Section \ref{sec_qi implies iso}. See Theorem \ref{4.15} for a more detailed version of Theorem \ref{1.1}.

The collection of RAAG's with finite outer automorphism group is a reasonably large class. Recall that there is a 1-1 correspondence between finite simplicial graphs and RAAG's (\cite{droms1987isomorphisms}), thus it makes sense to talk about a random RAAG by considering the Erdős–Rényi model for random graphs. If the parameters of the model are in the right range, then almost all RAAG's have finite outer automorphism group (\cite{charney2012random,MR2966695}).

The class of 2-dimensional RAAG's with finite outer automorphism group is strictly larger than the class of atomic RAAG's, moreover, there are plenty of higher dimensional RAAG's with finite outer automorphism group. 

Whether $\out(G(\Gamma))$ is finite or not can be easily read from $\Gamma$. We defined the \emph{closed star} of a vertex $v$ in $\Gamma$, denoted by $St(v)$, to be the full subgraph (see Section \ref{subsec_notation}) spanned by $v$ and vertices adjacent to $v$. Similarly, $lk(v)$ is defined to be the full subgraph spanned by vertices adjacent to $v$. Note that this definition is slightly different from the usual one.

By results in \cite{servatius1989automorphisms,laurence1995generating}, $\out(G(\Gamma))$ is generated by the following four types of elements (we identify the vertex set of $\Gamma$ with a standard generating set of $G(\Gamma)$): 
\begin{enumerate}
\item Given vertex $v\in\Gamma$, sending $v\to v^{-1}$ and fixing all other generators.
\item Graph automorphisms of $\Gamma$.
\item If $lk(w)\subset St(v)$ for vertices $w,v\in\Gamma$, sending $w\to wv$ and fixing all other generators induces a group automorphism. It is called a \textit{transvection}. When $d(v,w)=1$, it is an \textit{adjacent transvection}, otherwise it is a \textit{non-adjacent transvection}.
\item Suppose $\Gamma\setminus St(v)$ is disconnected. Then one obtains a group automorphism by picking a connected component $C$ and sending $w\to vwv^{-1}$ for each vertex $w\in C$ (all other generators are fixed). It is called a \textit{partial conjugation}.
\end{enumerate}

Elements of type (3) or (4) have infinite order in $\out(G(\Gamma))$ while elements of type (1) or (2) are of finite order. $\out(G(\Gamma))$ is finite if and only if $\Gamma$ does not contain any separating closed star, and there do not exist distinct vertices $v,w\in\Gamma$ such that $lk(w)\subset St(v)$.

\begin{thm}
\label{1.3}
Suppose $\out(G(\Gamma_{1}))$ is finite. Then the following are equivalent:
\begin{enumerate}
\item $G(\Gamma_{2})$ is quasi-isometric to $G(\Gamma_{1})$.
\item $G(\Gamma_{2})$ is isomorphic to a subgroup of finite index in $G(\Gamma_{1})$.
\item $\Gamma_{2}^{e}$ is isomorphic to $\Gamma_{1}^{e}$.
\end{enumerate}
\end{thm}
Here $\Gamma^{e}$ denotes the \textit{extension graph} introduced by Kim and Koberda in \cite{kim2013embedability} (see Definition 2.11). Extension graphs can be viewed as \textquotedblleft curve graphs\textquotedblright\ for RAAG's (\cite{kim2014geometry}). This analog carries on to the aspect of quasi-isometric rigidity. Namely, if $G$ is a mapping class group and $q:G'\to G$ is a quasi-isometry, then it is shown in \cite{MR2928983} that $G'$ naturally acts on the curve graph associated with $G$. This is still true if $G$ is a RAAG with some restriction on its outer automorphism group, for example, $\out(G)$ is finite.

However, in general, there exists a pair of commensurable RAAG's with different extension graphs, see Example \ref{3.29}. There also exists a pair of non-quasi-isometric RAAG's with isomorphic extension graphs, see \cite[Section 5.3]{huang2016quasi}.
 
Motivated by Theorem \ref{1.3} (2), we now look at finite index \textit{RAAG subgroups} (i.e. subgroups which are also RAAG's) of $G(\Gamma_{1})$.

Given a RAAG $G(\Gamma)$ (not necessarily having a finite outer automorphism group) and pick a standard generating set $S$ for $G(\Gamma)$. Let $d_{S}$ be the word metric on $G(\Gamma)$ with respect to $S$. A subset $K\subset G(\Gamma)$ is \textit{$S$-convex} if for any three points $x,y\in K$ and $z\in G(\Gamma)$ such that $d_{S}(x,y)=d_{S}(x,z)+d_{S}(z,y)$, we must have $z\in K$. Every finite $S$-convex subset $K$ naturally gives rise to a finite index RAAG subgroup $G\le G(\Gamma)$ such that $K$ is the fundamental domain of the left action $G\curvearrowright G(\Gamma)$. For example, if $G(\Gamma)=\Bbb Z\oplus\Bbb Z$ and pick $K$ to be a rectangle of size $n$ by $m$, then the corresponding subgroup is of form $n\Bbb Z\oplus m\Bbb Z$. The detailed construction in the more general case is given in Section \ref{subsec_construct finite index}. $G$ is called an \textit{$S$-special} subgroup of $G(\Gamma)$. A subgroup of $G(\Gamma)$ is \textit{special} if it is $S$-special for some standard generating set $S$. A similar construction in the case of right-angled Coxeter groups is in \cite{haglund2008finite}.

Here is an alternating description in terms of the canonical completion introduced by \cite{MR2377497}. Let $S(\Gamma)$ be the Salvetti complex of $G(\Gamma)$ (see Section \ref{subsec_raag}) and let $X(\Gamma)$ be the universal cover. We pick an identification between $G(\Gamma)$ and the 0-skeleton of $X(\Gamma)$. The above subset $K$ gives rise to a convex subcomplex $\bar{K}\subset X(\Gamma)$. Then the corresponding special subgroup is the fundamental group of the canonical completion with respect the local isometry $\bar{K}\to S(\Gamma)$.

Our next result says if $\out(G(\Gamma))$ is finite, then this is the only way to obtain finite index RAAG subgroups of $G(\Gamma)$.

\begin{thm}
\label{1.4}
Suppose $\out(G(\Gamma))$ is finite and let $S$ be a standard generating set for $G(\Gamma)$. Then all finite index RAAG subgroups are $S$-special. Moreover, there is a 1-1 correspondence between non-negative finite $S$-convex subsets of $G(\Gamma)$ based at the identity and finite index RAAG subgroups of $G(\Gamma)$. 
\end{thm}

See Theorem \ref{6.15} for a slight reformulation of Theorem \ref{1.4}.

We need to explain two terms: non-negative and based at the identity. For example, take $G=n\Bbb Z\oplus m\Bbb Z$ inside $\Bbb Z\oplus\Bbb Z$, then any $n$ by $m$ rectangle could be the fundamental domain for the action of $G$. We naturally require the rectangle to be in the first quadrant and contain the identity, which would give us a unique choice. Similar things can be done in all RAAG's and this two terms will be defined precisely in Section \ref{sec_geometry of f.i. raag subgroups}.

The most simple example is when $G(\Gamma)=\Bbb Z$, we have a 1-1 correspondence between finite index subgroups of form $n\Bbb Z$ and the intervals of form $[0,n-1]$.

\begin{cor}
\label{1.5}
If $\out(G(\Gamma_{1}))$ is finite, then $G(\Gamma_{2})$ is quasi-isometric to $G(\Gamma_{1})$ if and only if $G(\Gamma_{2})$ is isomorphic to a special subgroup of $G(\Gamma_{1})$.
\end{cor}

It turns out that there is an algorithm to enumerate the defining graphs of all special subgroups of a RAAG, so
\begin{thm}
\label{1.6}
If $\out(G(\Gamma))$ is finite, then $G(\Gamma')$ is quasi-isometric to $G(\Gamma)$ if and only if $\Gamma'$ can be obtained from $\Gamma$ by finitely many GSE's. In particular, there is an algorithm to determine whether $G(\Gamma')$ and $G(\Gamma)$ are quasi-isometric by looking at the graphs $\Gamma$ and $\Gamma'$.
\end{thm}

GSE is a generalized version of star extension in \cite[Example~1.4]{MR2421136}, see also \cite[lemma~50]{kim2013embedability}. It will be defined in Section \ref{sec_geometry of f.i. raag subgroups}.

A question motivated by Theorem \ref{1.3} is the following:

\begin{que}
\label{question}
Let $G(\Gamma)$ be a RAAG such that $\out(G(\Gamma))$ is finite. And let $H$ be a finite generated group quasi-isometric to $G(\Gamma)$. What can we say about $H$?
\end{que}

As a partial answer to this question, we prove the following result in \cite{huang2016groups}.

\begin{thm}
$($\cite{huang2016groups}$)$ Let $G(\Gamma)$ and $H$ be as in Question \ref{question}. Then the induced quasi-action $H\curvearrowright X(\Gamma)$ is quasi-isometrically conjugate to a geometric action $H\curvearrowright X'$. Here $X'$ is a $CAT(0)$ cube complex which is closely related to $X(\Gamma)$.
\end{thm}
\subsection{Comments on the Proof}
\subsubsection{Proof of Theorem \ref{1.1}}
We start with several notations. The Salvetti complex of $G(\Gamma)$ is denoted by $S(\Gamma)$, the universal cover of $S(\Gamma)$ is denoted by $X(\Gamma)$, and flats in $X(\Gamma)$ that cover standard tori in $S(\Gamma)$ are called \textit{standard flats}. See Section \ref{subsec_raag} for precise definition of these terms.
 
Let $q:X(\Gamma)\to X(\Gamma')$ be a quasi-isometry. The proof of Theorem \ref{1.1} follows the scheme of the proof of the main theorem in \cite{MR2421136}. Similar schemes can also be found in \cite{kleiner1997rigidity,MR2928983}. There are three steps in \cite{MR2421136}. First they show $q$ maps top dimensional flats to top dimensional flats up to finite Hausdorff distance. However, the collection of all top dimensional flats is too large to be linked directly to the combinatorics of RAAG's, so the second step is to show quasi-isometries preserve standard flats up to finite Hausdorff distance. The third step is to straighten the quasi-isometry such that it actually maps standard flats to standard flats exactly, not just up to finite Hausdorff distance, and the conclusion follows automatically. 

In our cases, the first step has been done in \cite{quasiflat}, where we show $q$ still preserves top dimensional flats up to finite Hausdorff distance in higher dimensional case. No assumption on the outer automorphism group is needed for this step.

The second step consists of two parts. First we show $q$ preserves certain top dimensional maximal products up to finite Hausdorff distance. Then one wish to pass to standard flats by intersecting these top dimensional objects. However, in the higher dimensional case, a lower dimensional standard flat may not be the intersection of top dimensional objects, and even in the case it is an intersection, one may not be able to read this information directly from the defining graph $\Gamma$. This is quite different from the 2-dimensional situation in \cite{MR2421136} and relies on several new ingredients. 

A necessary condition for $q$ to preserve the standard flats is that every elements in $\out(G(\Gamma))$ should do so, which implies there could not be any transvections in $\out(G(\Gamma))$. This condition is also sufficient.
\begin{thm}
\label{1.7}
Suppose $\out(G(\Gamma))$ is transvection free. Then there exists a positive constant $D=D(L,A,\Gamma)$ such that for any standard flat $F\subset X(\Gamma)$, there exists a standard flat $F'\subset X(\Gamma')$ such that $d_{H}(q(F),F')<D$.
\end{thm}
Here $d_{H}$ denotes the Hausdorff distance.

In step 3, we introduce an auxiliary simplicial complex $\mathcal{P}(\Gamma)$, which serves as a link between the asymptotic geometry of $X(\Gamma)$ and the combinatorial structure of $X(\Gamma)$. More precisely, on one hand, $\mathcal{P}(\Gamma)$ can be viewed as a simplified Tits boundary for $X(\Gamma)$, on the other hand, one can read certain information about the wall space structure of $X(\Gamma)$ from $\mathcal{P}(\Gamma)$. This complex turns out to coincide with the extension graph introduced in \cite{kim2013embedability}, where it was motivated from the viewpoint of mapping class group.

Denote the Tits boundary of $X(\Gamma)$ by $\partial_{T}(X(\Gamma))$, and let $T(\Gamma)\subset\partial_{T}(X(\Gamma))$ be the union of Tits boundaries of standard flats in $X(\Gamma)$. Then $T(\Gamma)$ has a natural simplicial structure. However, $T(\Gamma)$ contains redundant information, this can be seen in the similar situation where the link of the base point of $S(\Gamma)$ looks more complicated than $\Gamma$, but they essentially contain the same information. 

This redundancy can be resolved by replacing the spheres in $T(X)$ that arise from standard flats by simplexes of the same dimension. This gives rise to a well defined simplicial complex $\mathcal{P}(\Gamma)$, since for any standard flats $F_{1}$ and $F_{2}$ with $\partial_{T}F_{1}\cap\partial_{T}F_{2}\neq\emptyset$, there exists a standard flat $F$ such that $\partial_{T}F=\partial_{T}F_{1}\cap\partial_{T}F_{2}$. See Section \ref{subsec_a boundary map} for more properties of $\mathcal{P}(\Gamma)$.

By Theorem \ref{1.7}, if both $\out(G(\Gamma))$ and $\out(G(\Gamma'))$ are transvection free, then $q$ induces a boundary map $\partial q:\mathcal{P}(\Gamma)\to\mathcal{P}(\Gamma')$, which is a simplicial isomorphism. Next we want to consider the converse and reconstruct a map $X(\Gamma)\to X(\Gamma')$ from the boundary map $\partial q$ in the following sense. Pick vertex $p\in X(\Gamma)$, let $\{F_{i}\}_{i=1}^{n}$ be the collection of maximal standard flats containing $p$. By Theorem \ref{1.7}, for each $i$, there exists a unique maximal standard flat $F'_{i}\subset X(\Gamma')$ such that $d_{H}(q(F_{i}),F'_{i})<\infty$. One may wish to map $p$, which turns out to be the intersection of $F_i$'s, to the intersection of all $F'_i$'s. However, in general $\cap_{i=1}^{n}F'_{i}$ may be empty, or contain more than one point, hence our map may not be well-defined.

It turns out that if we also rule out partial conjugations in $\out(G(\Gamma))$, then $\cap_{i=1}^{n}F'_{i}$ is exactly a point. This give rises to a well-defined map $\bar{q}:X(\Gamma)^{(0)}\to X(\Gamma')^{(0)}$ which maps vertices in a standard flat to vertices in standard flat. If $\out(G(\Gamma'))$ is also finite, then we can define an inverse map of $\bar{q}$ and this is enough to deduce Theorem \ref{1.1}.

\subsubsection{Proof of Theorem \ref{1.3}}
If only $\out(G(\Gamma))$ is assumed to be finite, we can still recover the fact that $\partial q$ is a simplicial isomorphism (this is non-trivial, since Theorem \ref{1.7} does not say for any standard flat $F'\subset X(\Gamma')$, we can find a standard flat $F\subset X(\Gamma)$ such that $d_{H}(q(F),F')<\infty$). Hence we can define $\bar{q}$ as before. However the inverse of $\bar{q}$ does not exist in general.

The next step is to trying to extend $\bar{q}$ to a cubical map (Definition \ref{cubical}) from $X(\Gamma)$ to $X(\Gamma')$. There are obvious obstructions: though $\bar{q}$ maps vertices in a standard geodesic to vertices in a standard geodesic, $\bar{q}$ may not preserve the order of these vertices. A typical example is given in the following picture, where one can permute the green level and the red level in a tree, then the order of vertices in the black line is not preserved.
\begin{center}
\includegraphics[scale=0.5]{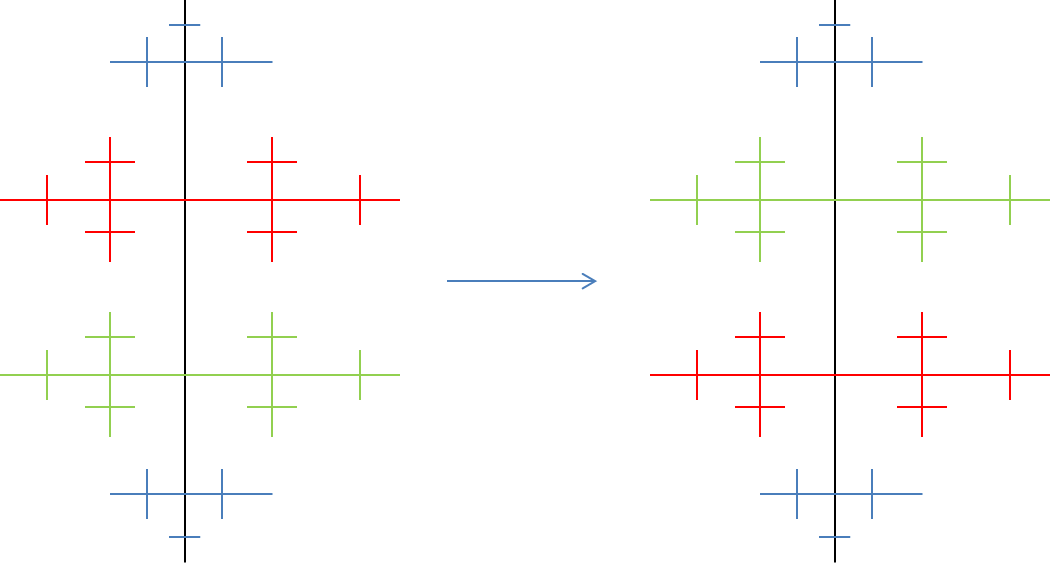}
\end{center}
A remedy is to \textquotedblleft flip backwards\textquotedblright. Namely we will pre-compose $\bar{q}$ with a sequence of permutations of \textquotedblleft levels\textquotedblright\ such that the resulting map restricted to each standard geodesic respects the order. Then we can extend $\bar{q}$ to a cubical map. This argument relies on the understanding of quasi-isometric flexibility, namely how much room we have to perform these flips. One formulation of this aspect is the following.
\begin{thm}
\label{1.8}
If $\out(G(\Gamma))$ is finite, then $\aut(\mathcal{P}(\Gamma))\cong \isom(G(\Gamma),d_{r})$.
\end{thm}
Here $d_{r}$ denotes the syllable metric, see \cite[Section 5.2]{kim2014geometry}.
 
Theorem \ref{1.3} - Theorem \ref{1.6} will rely on the cubical map $\bar{q}$. In particular, $\bar{q}^{-1}(x)$ ($x\in X(\Gamma')$ is a vertex) is a compact convex subcomplex and this is how we obtain the $S$-convex subset in Theorem \ref{1.4}. 

\subsection{Organization of the Paper}
Section \ref{sec_prelim} contains basic notations used in this paper and some background material about $CAT(0)$ cube complexes and RAAG's. In particular, Section \ref{subsec_coarse intersection} collects several technical lemmas about $CAT(0)$ cube complex. One can skip Section \ref{subsec_coarse intersection} on first reading and come back when needed.

In Section \ref{sec_stable subgraph} we prove Theorem \ref{1.7}. Section \ref{subsec_standard subcomplexes and tree products} is about the stability of top dimensional maximal product subcomplexes under quasi-isometries and Section \ref{subsec_stability of std flats} deals with lower dimensional standard flats. In Section \ref{sec_qi implies iso} we prove Theorem \ref{1.1}. We will construct the extension complex from our viewpoint in Section \ref{subsec_a boundary map} and explain how is this object related to Tits boundary, flat space and contact graph. In Section \ref{subsec_reconstruct}, we describe how to reconstruct the quasi-isometry.

Section \ref{subsec_preservation of extesion complex} and Section \ref{subsec_coherent ordering and labelling} are devoted to proving Theorem \ref{1.3}. We prove Theorem \ref{1.4}, Corollary \ref{1.5} and Theorem \ref{1.6} in Section \ref{sec_geometry of f.i. raag subgroups}.

\subsection{Acknowledgement}
I would like to thank Bruce Kleiner and Robert Young, for reading parts of the paper and give valuable comments. I would thank Jason Behrstock, Ruth Charney, and Walter Neumann for suggestion interesting questions, which leads to Section \ref{subsec_GSE} of this paper. I would thank Pallavi Dani, Mark Hagen, Thomas Koberda, Do-gyeom Kim and Sang-hyun Kim for related discussions. I also thank the referee for carefully reading the paper and providing helpful comments.

\section{Preliminaries}
\label{sec_prelim}
\subsection{Notation and Conventions}
\label{subsec_notation}
All graphs in this paper are simple.

The flag complex of a graph $\Gamma$ is denoted by $F(\Gamma)$, i.e. $F(\Gamma)$ is a flag complex such that its 1-skeleton is $\Gamma$.

A subcomplex $K'$ in a combinatorial polyhedral complex $K$ is \textit{full} if $K'$ contains all the subcomplexes of $K$ which have the same vertex set as $K'$. If $K$ is 1-dimensional, then we also call $K'$ a \textit{full subgraph}.

We use \textquotedblleft$\ast$\textquotedblright\ to denote the join of two simplicial complexes and \textquotedblleft$\circ$\textquotedblright\ to denote the join of two graphs. Let $K$ be a simplicial complex or a graph. By viewing the 1-skeleton as a metric graph with edge length $=1$, we obtain a metric defined on the 0-skeleton, which we denote by $d$. Let $N\subset K$ be a subcomplex. We define the \textit{orthogonal complement} of $N$, denoted by $N^{\perp}$, to be the set $\{w\in K^{(0)}\mid d(w,v)=1$ for any vertex $v\in N\}$; define the \textit{link} of $N$, denoted by $lk(N)$, to be the full subcomplex spanned by $N^{\perp}$; and define the \textit{closed star} of $N$, denoted by $St(N)$, to be the full subcomplex spanned by $N\cup lk(N)$. Suppose $L$ is a subcomplex such that $N\subset L\subset K$. We denote the closed star of $N$ in $L$ by $St(v,L)$. If $L$ is a full subcomplex, then $St(N,L)=St(N)\cap L$. We can define $lk(N,L)$ in a similar way. Let $M\subset K$ be an arbitrary subset. We denote the collection of vertices inside $M$ by $v(M)$.

We use $\id$ to denote the identity element of a group, and use $\textmd{Id}$ to denote the identity map from a space to itself.

Let $(X,d)$ be a metric space. The open ball of radius $r$ centred at $p$ in $X$ will be denoted by $B(p,r)$. Given subsets $A,B\subset X$, the open $r$-neighborhood of a subset $A$ is denoted by $N_{r}(A)$. The diameter of $A$ is denoted by $diam(A)$. The Hausdorff distance between $A$ and $B$ is denoted by $d_{H}(A,B)$. We will also use the following adapted notation of coarse set theory introduced in \cite{mosher2004quasi}.
\begin{table}[ht] 
\centering 
\begin{tabular}{c c} 
\hline
Symbol & Meaning\\ 
\hline 
$A\subset_{r}B$ & $A\subset N_{r}(B)$\\ 
$A\subset_{\infty}B$  & $\exists r>0$ such that $A\subset N_{r}(B)$\\ 
$A\overset{r}{=}B$ & $d_{H}(A,B)\le r$ \\ 
$A\overset{\infty}{=}B$ & $d_{H}(A,B)<\infty$ \\
$A\cap_{r} B$ & $N_{r}(A)\cap N_{r}(B)$ \\
\hline 
\end{tabular} 
\end{table}

\subsection{$CAT(0)$ space and $CAT(0)$ cube complex}
The standard reference for $CAT(0)$ spaces is \cite{MR1744486}. 

Let $(X,d)$ be a $CAT(0)$ space. Pick $x,y\in X$, we denote the unique geodesic segment joining $x$ and $y$ by $\overline{xy}$. For $y,z\in X\setminus \{x\}$, denote the comparison angle between $\overline{xy}$ and $\overline{xz}$ at $x$ by $\overline{\angle}_{x}(y,z)$ and the Alexandrov angle by $\angle_{x}(y,z)$. 

The boundary of $X$, denoted by $\partial X$, is the collection of asymptotic classes of geodesic rays. $\partial X$ has an angular metric, which is defined by 
\begin{equation*}
\angle(\xi_{1},\xi_{2})=\lim_{t,t'\to\infty}\overline{\angle}_{p}(l_{1}(t),l_{2}(t'))
\end{equation*}
where $l_{1}$ and $l_{2}$ are unit speed geodesic rays emanating from a base point $p$ such that $l_{i}(\infty)=\xi_{i}$ for $i=1,2$. This metric does not depend on the choice of $p$, and the length metric associated to the angular metric, denoted by $d_{T}$, is called the Tits metric. The Tits boundary, denoted by $\partial_{T}X$, is the $CAT(1)$ space $(\partial X,d_{T})$ (see Chapter II.8 and II.9 of \cite{MR1744486}).

Given two metric spaces $(X_{1},d_{1})$ and $(X_{2},d_{2})$, denote the \textit{Cartesian product} of $X_{1}$ and $X_{2}$ by $X_{1}\times X_{2}$, i.e. $d=\sqrt{d^{2}_{1}+d^{2}_{2}}$ on $X_{1}\times X_{2}$. If $X_{1}$ and $X_{2}$ are $CAT(0)$, then so is $X_{1}\times X_{2}$.

An \textit{$n$-flat} in a $CAT(0)$ space $X$ is the image of an isometric embedding $\Bbb E^{n}\to X$. Note that any flat is convex in $X$.

Pick a convex subset $C\subset X$, then $C$ is also $CAT(0)$. We use $\pi_{C}$ to denote the nearest point projection from $X$ to $C$. $\pi_{C}$ is well-defined and is 1-Lipschitz. Moreover, pick $x\in X\setminus C$, then $\angle_{\pi_{C}(x)}(x,y)\ge\pi/2$ for any $y\in C$ such that $y\neq \pi_{C}(x)$ (see Proposition II.2.4 of \cite{MR1744486}).

If $C'\subset X$ is another convex set, then $C'$ is \textit{parallel} to $C$ if $d(\cdot,C)|_{C'}$ and $d(\cdot,C')|_{C}$ are constant functions. In this case, there is a natural isomorphism between $C\times[0,d(C,C')]$ and the convex hull of $C$ and $C'$. We define the \textit{parallel set} of $C$, denoted by $P_{C}$, to be the union of all convex subsets of $X$ parallel to $C$. If $C$ has geodesic extension property, or more generally, $C$ is boundary-minimal (Section 3.C of \cite{caprace2009isometry}), then $P_{C}$ is a convex subset in $X$. Moreover, $P_{C}$ admits a canonical splitting $P_{C}=C\times C^{\perp}$, where $C^{\perp}$ is also a $CAT(0)$ space.

Now we turn to $CAT(0)$ cube complexes. All cube complexes in this paper are assumed to be finite dimensional.

A cube complex $X$ is obtained by gluing a collection of unit Euclidean cubes isometrically along their faces, see II.7.32 of \cite{MR1744486} for a precise definition. Then the cube complex has a natural piecewise Euclidean metric. This metric is complete and geodesic since $X$ is finite dimensional (I.7.19 of \cite{MR1744486}) and is non-positively curved if the link of each vertex is a flag complex (\cite{MR919829}). If in addition $X$ is simply connected, then this metric is $CAT(0)$ and $X$ is said to be a \textit{$CAT(0)$ cube complex}. We can put a different metric on the 1-skeleton $X^{(1)}$ by considering it as a metric graph with all edge lengths 1. This is called the \textit{$\ell^{1}$ metric}. We use $d$ for the $CAT(0)$ metric on $X$ and $d_{\ell^1}$ for the $\ell^1$ metric on $X^{(1)}$. The natural injection $(X^{(1)},d_{\ell^1})\hookrightarrow (X,d)$ is a quasi-isometry (I.7.31 of \cite{MR1744486} or Lemma 2.2 of \cite{caprace2011rank}). In this paper, we will mainly use the $CAT(0)$ metric unless otherwise specified. Also any notions which depend on the metric, like geodesic, convex subset, convex hull etc, will be understood automatically with respect to the $CAT(0)$ metric unless otherwise specified.

\begin{definition}
	\label{cubical}
	(cf. \cite[Section 2.1]{caprace2011rank}) A cellular map between $CAT(0)$ cube complexes is \textit{cubical} if its restriction $\sigma\to\tau$ between cubes factors as $\sigma\to\eta\to\tau$, where the first map $\sigma\to\eta$ is a natural projection onto a face of $\sigma$ and the second map $\eta\to\tau$ is an isometry. 
\end{definition}

A \textit{geodesic segment, geodesic ray} or \textit{geodesic} in $X$ is an isometric embedding of $[a,b]$, $[0,\infty)$ or $\Bbb R$ into $X$ with respect to the $CAT(0)$ metric. A \textit{combinatorial geodesic segment, combinatorial geodesic ray} or \textit{combinatorial geodesic} is a $\ell^{1}$-isometric embedding of $[a,b]$, $[0,\infty)$ or $\Bbb R$ into $X^{(1)}$ such that its image is a subcomplex.

Let $X$ be a $CAT(0)$ cube complex and let $Y\subset X$ be a subcomplex. Then the following are equivalent (see \cite{haglund2008finite}):
\begin{enumerate}
\item $Y$ is convex with respect to the $CAT(0)$ metric.
\item $Y$ is a full subcomplex and $Y^{(1)}\subset X^{(1)}$ is convex with respect to the $\ell^{1}$ metric.
\item $Lk(p,Y)$ (the link of $p$ in $Y$) is a full subcomplex of $Lk(P,X)$ for every vertex $p\in Y$.
\end{enumerate}

The collection of convex subcomplexes in a $CAT(0)$ cube complex enjoys the following version of Helly's property (\cite{gerasimov1998fixed}):

\begin{lem}
\label{2.1}
Let $X$ be as above and $\{C_{i}\}_{i=1}^{k}$ be a collection of convex subcomplexes. If $C_{i}\cap C_{j}\neq\emptyset$ for any $1\le i\neq j\le k$, then $\cap_{i=1}^{k}C_{i}\neq\emptyset$.
\end{lem}

\begin{lem}
\label{2.2}
Let $X_{1}$ and $X_{2}$ be two $CAT(0)$ cube complexes and let $K\subset X_{1}\times X_{2}$ be a convex subcomplex. Then $K$ admits a splitting $K=K_{1}\times K_{2}$ where $K_{i}$ is a convex subcomplex of $X_{i}$ for $i=1,2$.
\end{lem}
The lemma is clear when $X_{1}\cong[0,1]$, and the general case follows from this special case. 

Now we come to the notion of hyperplane, which is the cubical analog of \textquotedblleft track\textquotedblright\  introduced in \cite{dunwoody1985accessibility}. A \textit{hyperplane} $h$ in a cube complex $X$ is a subset such that 
\begin{enumerate}
\item $h$ is connected.
\item For each cube $C\subset X$, $h\cap C$ is either empty or a union of mid-cubes of $C$. 
\item $h$ is minimal, i.e. if there exists $h'\subset h$ satisfying (1) and (2), then $h=h'$.
\end{enumerate}
Recall that a \textit{mid-cube} of $C=[0,1]^{n}$ is a subset of form $f_{i}^{-1}(1/2)$, where $f_{i}$ is one of the coordinate functions.

If $X$ is a $CAT(0)$ cube complex, then the following are true (see \cite{MR1347406}):
\begin{enumerate}
\item Each hyperplane is embedded, i.e. $h\cap C$ is either empty or a mid-cube of $C$ (in more general cube complexes, it is possible that $h\cap C$ contains two or more mid-cubes of $C$);
\item $h$ is a convex subset in $X$ and $h$ with the induced cell structure from $X$ is also a $CAT(0)$ cube complex;
\item $X\setminus h$ has exactly two connected components, they are called \textit{halfspaces}. The closure of a halfspace is called \textit{closed halfspace}, which is also convex in $X$ with respect to the $CAT(0)$ metric.
\item Let $N_{h}$ be the smallest subcomplex of $X$ that contains $h$. Then $N_{h}$ is a convex subcomplex of $X$ and there is a natural isometry $i:N_{h}\to h\times [0,1]$ such that $i(h)=h\times\{1/2\}$. $N_{h}$ is called the \textit{carrier} of $h$.
\item For every edge $e\subset X$, there exists a unique hyperplane $h_{e}$ which intersects $e$ in its midpoint. In this case, we say $h_{e}$ is the hyperplane dual to $e$ and $e$ is an edge dual to the hyperplane $h_{e}$.
\item Lemma \ref{2.1} is also true for a collection of hyperplanes.
\end{enumerate}

Now it is easy to see an edge path $\omega\subset X$ is a combinatorial geodesic segment if and only if there do not exist two different edges of $\omega$ such that they are dual to the same hyperplane. Moreover, for two vertices $v,w\in X$, their $\ell^{1}$ distance is exactly the number of hyperplanes that separate $v$ from $w$.

Pick an edge $e\subset X$ and let $\pi_{e}:X\to e\cong [0,1]$ be the $CAT(0)$ projection. Then
\begin{enumerate}
\item The hyperplane dual to $e$ is exactly $\pi^{-1}_{e}(1/2)$.
\item $\pi_{e}^{-1}(t)$ is convex in $X$ for any $0\le t\le 1$, moreover, if $0<t<t'<1$, then $\pi_{e}^{-1}(t)$ and $\pi_{e}^{-1}(t')$ are parallel.
\item Let $N_{h_e}$ be the carrier of the hyperplane dual to $e$. Then $N_{h_e}$ is the closure of $\pi_{e}^{-1}(0,1)$. Alternatively, we can describe $N_{h_e}$ as the parallel set of $e$.
\end{enumerate}

\subsection{Coarse intersections of convex subcomplexes} 
\label{subsec_coarse intersection}
\begin{lem}[Lemma 2.10 of \cite{quasiflat}]
\label{2.3}
Let $X$ be a $CAT(0)$ cube complex of dimension $n$ and let $C_{1}$, $C_{2}$ be convex subcomplexes. Put $\Delta=d(C_{1},C_{2})$. Let $Y_{1}=\{y\in C_{1}\mid d(y,C_{2})=\Delta\}$ and $Y_{2}=\{y\in C_{2}\mid d(y,C_{1})=\Delta\}$. Then:
\begin{enumerate}
\item $Y_{1}$ and $Y_{2}$ are not empty.
\item $Y_{1}$ and $Y_{2}$ are convex; $\pi_{C_{1}}$ maps $Y_{2}$ isometrically onto $Y_{1}$ and $\pi_{C_{2}}$ maps $Y_{1}$ isometrically onto $Y_{2}$; the $CAT(0)$ convex hull of $Y_{1}\cup Y_{2}$ is isometric to $Y_{1}\times [0,\Delta]$ (since we are taking the $CAT(0)$ convex hull, it does not has to be a subcomplex).
\item $Y_{1}$ and $Y_{2}$ are subcomplexes, and $\pi_{C_{2}}|_{Y_{1}}$ is a cubical isomorphism from $Y_1$ to $Y_2$ with its inverse given by $\pi_{C_{1}}|_{Y_{2}}$.
\item For any $\epsilon>0$, there exists $A=A(\Delta,n,\epsilon)$ such that if $p_{1}\in C_{1}$, $p_{2}\in C_{2}$ and $d(p_{1},Y_{1})\geq \epsilon>0$, $d(p_{2},Y_{2})\geq \epsilon>0$, then 
\begin{equation}
\label{2.4}
d(p_{1}, C_{2})\geq \Delta + Ad(p_{1},Y_{1});\ \ \  d(p_{2}, C_{1})\geq \Delta + Ad(p_{2},Y_{2})
\end{equation}
\end{enumerate}
\end{lem}

\begin{remark}
\label{2.5}
Equation (\ref{2.4}) implies for any $r>0$, $(C_{1}\cap_{r}C_{2})\subset_{r'}Y_{i}$ ($i=1,2$), where $r'=\min(1,(2r-\Delta)/A)+r$ and $A=A(\Delta,n,1)$. Moreover, $\partial_{T}C_{1}\cap\partial_{T}C_{2}=\partial_{T}Y_{1}=\partial_{T}Y_{2}$.
\end{remark}

The remark implies $Y_{1}\overset{\infty}{=}Y_{2}\overset{\infty}{=}(C_{1}\cap_{r}C_{2})$ for $r$ large enough. We use $\inc(C_{1},C_{2})=(Y_{1},Y_{2})$ to describe this situation, where $\mathcal{I}$ stands for the word \textquotedblleft intersect\textquotedblright. The next lemma gives a combinatorial description of $Y_{1}$ and $Y_{2}$.

\begin{lem}
\label{2.6}
Let $X$, $C_{1}$, $C_{2}$, $Y_{1}$ and $Y_{2}$ be as above. Pick an edge $e$ in $Y_{1}$ $($or $Y_{2})$ and let $h$ be the hyperplane dual to $e$. Then $h\cap C_{i}\neq\emptyset$ for $i=1,2$. Conversely, if a hyperplane $h'$ satisfies $h'\cap C_{i}\neq\emptyset$ for $i=1,2$, then $h'$ is the dual hyperplane of some edge $e'$ in $Y_{1}$ $($or $Y_{2})$. Moreover, $\inc(h'\cap C_{1},h'\cap C_{2})=(h'\cap Y_{1},h'\cap Y_{2})$.
\end{lem}

\begin{proof}
The first part of the lemma follows from Lemma \ref{2.3}. Let $\inc(h'\cap C_{1},h'\cap C_{2})=(Y'_{1},Y'_{2})$. Pick $x\in Y'_{1}$ and let $x'=\pi_{h'\cap C_{2}}(x)\in Y'_{2}$. Then $\pi_{h'\cap C_{1}}(x')=x$. Let $N_{h'}=h'\times[0,1]$ be the carrier of $h'$. Then $(h'\cap C_{i})\times(\frac{1}{2}-\epsilon,\frac{1}{2}+\epsilon)=C_{i}\cap(h'\times(\frac{1}{2}-\epsilon,\frac{1}{2}+\epsilon))$ for $i=1,2$ and $\epsilon<\frac{1}{2}$. Thus for any $y\in C_{2}$, $\angle_{x'}(x,y)\ge\pi/2$, which implies $x'=\pi_{C_{2}}(x)$. Similarly, $x=\pi_{C_{1}}(x')=\pi_{C_{1}}\circ\pi_{C_{2}}(x)$, hence $x\in Y_{1}$ and $Y'_{1}\subset Y_{1}$. By the same argument, $Y'_{2}\subset Y_{2}$, thus $Y'_{i}=Y_{i}\cap h'$ for $i=1,2$ and the lemma follows.
\end{proof}

Lemma \ref{2.3}, Remark \ref{2.5} and Lemma \ref{2.6} can also be applied to $CAT(0)$ rectangle complexes of finite type, whose cells are of form $\Pi_{i=1}^{n}[0,a_{i}]$. \textquotedblleft Finite type\textquotedblright\ means there are only finitely many isometry types of rectangle cells in the rectangle complex.

\begin{lem}
\label{2.7}
Let $X,C_{1},C_{2},Y_{1}$ and $Y_{2}$ be as above. If $h$ is a hyperplane separating $C_{1}$ from $C_{2}$, then there exists a convex set $Y\subset h$ such that $Y$ is parallel to $Y_{1}$ (or $Y_{2}$).
\end{lem}

\begin{proof}
Let $\Delta=d(C_{1},C_{2})$ and let $M=Y_{1}\times [0,\Delta]$ be the convex hull of $Y_{1}$ and $Y_{2}$. We want to prove $M\cap h=Y_{1}\times \{t\}\subset Y_{1}\times [0,\Delta]$ for some $t\in [0,\Delta]$. It suffices to show for any edge $e\subset Y_{1}$, $(e\times [0,\Delta])\cap h=e\times \{t\}$ for some $t$. 

Pick point $x\in e$ and let $\{x\}\times\{t\}$ be the point in $M=Y_{1}\times [0,\Delta]$. Since $e\times\{t\}$ is parallel to $e$, $e\times\{t\}$ sits inside a cube and $e\times\{t\}$ is parallel to an edge of this cube. Thus either $e\times\{t\}\subset h$ or $e\times\{t\}$ is parallel to some edge dual to $h$, but the second case implies that $h$ is dual to $e$ and $h\cap Y_{1}\neq\emptyset$, which is impossible, so $e\times \{t\}\subset(e\times [0,\Delta])\cap h$. Now we are done since $(\{x\}\times [0,\Delta])\cap h$ is exactly one point for each $x\in e$. 
\end{proof}

\subsection{Right-angled Artin groups} 
\label{subsec_raag}
Pick a finite simplicial graph $\Gamma$. Let $G(\Gamma)$ be a RAAG. A generating set $S\subset G(\Gamma)$ is called a \emph{standard generating set} if all relators in the associated group presentation are commutators. Each standard generating set $S$ determines a graph $\Gamma_S$ whose vertices are elements in $S$ and two vertices are adjacent if the corresponding group elements commute. It follows from \cite{droms1987isomorphisms} that the isomorphism type of $\Gamma_S$ does not depend on the choice of the standard generating set $S$, in particular, $\Gamma_S$ is isomorphic to $\Gamma$.

Let $S$ be a standard generating set for $G(\Gamma)$. We label the vertices of $\Gamma$ by elements in $S$. $G(\Gamma)$ has a nice Eilenberg-MacLane space $S(\Gamma)$, called the Salvetti complex (see \cite{MR1368655,charney2007introduction}). $S(\Gamma)$ is a non-positively curved cube complex. The 2-skeleton of $S(\Gamma)$ is the usual presentation complex of $G(\Gamma)$. If the presentation complex contains a copy of 2-skeleton of a 3-torus, then we attach a 3-cell to obtain a 3-torus. We can build $S(\Gamma)$ inductively in this manner, and this process will stop after finitely many steps. The closure of each $k$-cell in $S(\Gamma)$ is a $k$-torus. Torus of this kind is called a \textit{standard torus}. There is a 1-1 correspondence between the $k$-cells (or standard tori of dimension $k$) in $S(\Gamma)$ and $k$-cliques (complete subgraph of $k$ vertices) in $\Gamma$, thus $\dim(S(\Gamma))=dim(F(\Gamma))+1$. We define the \textit{dimension} of $G(\Gamma)$ to be the dimension of $S(\Gamma)$.

Denote the universal cover of $S(\Gamma)$ by $X(\Gamma)$, which is a $CAT(0)$ cube complex. Our previous labeling of vertices of $\Gamma$ induces a labeling of the standard circles of $S(\Gamma)$, which lifts to a labeling of edges of $X(\Gamma)$. We choose an orientation for each standard circle of $S(\Gamma)$ and this would give us a directed labeling of the edges in $X(\Gamma)$. If we pick a base point $v\in X(\Gamma)$ ($v$ is a vertex), then there is a 1-1 correspondence between words in $G(\Gamma)$ and edge paths in $X(\Gamma)$ which start at $v$. 

Each full subgraph $\Gamma'\subset\Gamma$ gives rise to a subgroup $G(\Gamma')\hookrightarrow G(\Gamma)$. Subgroup of this kind is called a \textit{$S$-standard subgroup} and a left coset of an $S$-standard subgroup is called a \textit{$S$-standard coset} (we will omit $S$ when the generating set is clear). There is also an embedding $S(\Gamma')\hookrightarrow S(\Gamma)$ which is locally isometric. Let $p:X(\Gamma)\to S(\Gamma)$ be the covering map. Then each connect component of $p^{-1}(S(\Gamma'))$ is a convex subcomplex isometric to $X(\Gamma')$. We will call these components \textit{standard subcomplexes with defining graph $\Gamma'$}. A \textit{standard k-flat} is a standard complex which covers a standard $k$-torus in $S(\Gamma)$. When $k=1$, we also call it a \textit{standard geodesic}.

We pick an identification of the Cayley graph of $G(\Gamma)$ with the 1-skeleton of $X(\Gamma)$, hence $G(\Gamma)$ is identified with the vertices of $X(\Gamma)$. Let $v\in X(\Gamma)$ be the base vertex which corresponds to the identity in the Cayley graph of $G(\Gamma)$. Then for any $h\in G(\Gamma)$, the convex hull of $\{hgv\}_{g\in G(\Gamma')}$ is a standard subcomplex associated with $\Gamma'$. Thus there is a 1-1 correspondence between standard subcomplexes with defining graph $\Gamma'$ in $X(\Gamma)$ and left cosets of $G(\Gamma')$ in $G(\Gamma)$. 

Note that for every edge $e\in X(\Gamma)$, there is a vertex in $\Gamma$ which shares the same label as $e$, and we denote this vertex by $V_{e}$. If $K\subset X(\Gamma)$ is a subcomplex, we define $V_{K}$ to be $\{V_{e}\mid e$ is an edge in $K\}$ and $\Gamma_{K}$ to be the full subgraph spanned by $V_{K}$. $\Gamma_K$ is called the \textit{support} of $K$. In particular, if $K$ is a standard subcomplex, then the \textit{defining graph} of $K$ is $\Gamma_{K}$. 

Every finite simplicial graph $\Gamma$ admits a canonical join decomposition $\Gamma=\Gamma_{1}\circ\Gamma_{2}\circ\cdots\circ\Gamma_{k}$, where $\Gamma_{1}$ is the maximal clique join factor and for $2\le i\le k$, $\Gamma_{i}$ does not allow any non-trivial join decomposition and is not a point. $\Gamma$ is \textit{irreducible} if this join decomposition is trivial. This decomposition induces a product decomposition $X(\Gamma)=\Bbb E^{n}\times\Pi_{i=2}^{k}X(\Gamma_{i})$, which is called the \textit{De Rahm decomposition} of $X(\Gamma)$. This is consistent with the canonical product decomposition of $CAT(0)$ cube complex discussed in Section 2.5 of \cite{caprace2011rank}. 
 
We turn to the asymptotic geometry of RAAG's. A right-angled Artin group $G(\Gamma)$ is one-ended if and only if $\Gamma$ is connected. Moreover, the $n$-connectivity at infinity of $G(\Gamma)$ can be read off from $\Gamma$, see \cite{brady2001connectivity}. In order to classify all RAAG's up to quasi-isometry, it suffices to consider those one-ended RAAG's. This follows from the main results in \cite{papasoglu2002quasi}. Moreover, Lemma 3.2 of \cite{papasoglu2002quasi} implies the following:

\begin{lem}
\label{2.8}
If $q:X(\Gamma)\to X(\Gamma')$ is an $(L,A)$-quasi-isometry, then there exists $D=D(L,A)>0$ such that for any connected component $\Gamma_{1}\subset\Gamma$ such that $\Gamma_{1}$ is not a point and any standard subcomplex $K_{1}\subset X(\Gamma)$ with defining graph $\Gamma_{1}$, there is a unique connected component $\Gamma'_{1}\subset\Gamma'$ and a unique standard subcomplex $K'_{1}\subset X(\Gamma')$ with defining graph $\Gamma'_{1}$ such that $d_{H}(q(K_{1}),K'_{1})<D$.
\end{lem}

It is shown in \cite{behrstock2012divergence} and \cite{abrams2013pushing} that $G(\Gamma)$ has linear divergence if and only if $\Gamma$ is a join or $\Gamma$ is one point, which implies $\Gamma$ being a join is a quasi-isometric invariant. Moreover, their results together with Theorem B of \cite{kapovich1998quasi} implies that the De Rahm decomposition is stable under quasi-isometry:

\begin{thm}
\label{2.9}
Given $X=X(\Gamma)$ and $X'=X(\Gamma')$, let $X=\Bbb R^{n}\times\prod_{i=1}^{k}X(\Gamma_{i})$ and let $X'=\Bbb R^{n'}\times\prod_{j=1}^{k'}X(\Gamma'_{j})$ be the corresponding De Rahm decomposition. If $\phi:X\to X'$ is an $(L,A)$ quasi-isometry, then $n=n'$, $k=k'$ and there exist constants $L'=L'(L,A)$, $A'=A'(L,A)$, $D=D(L,A)$ such that after re-indexing the factors in $X'$, we have $(L',A')$ quasi-isometry $\phi_{i}: X(\Gamma_{i})\to X(\Gamma'_{i})$ so that $d(p'\circ\phi, \prod_{i=1}^{k}\phi_{i}\circ p)<D$, where $p: X\to \prod_{i=1}^{k}X(\Gamma_{i})$ and $p':X'\to \prod_{i=1}^{k}X(\Gamma'_{i})$ are the projections.
\end{thm}

Thus in order to study the quasi-isometric classification of RAAG's, it suffices to study those RAAG's which are one-ended and irreducible, but this will rely on finer quasi-isometric invariant of RAAG's. 

Recall that in the case of Gromov hyperbolic spaces, quasi-isometries map geodesics to geodesics up to finite Hausdorff distance, hence induces a well-defined boundary map. The analog of this fact for 2-dimensional RAAG's has been established in \cite{bestvina2008quasiflats}, i.e. quasi-isometries map 2-flats to 2-flats up to finite Hausdorff distance. The following is a higher dimensional generalization of Theorem 3.10 of \cite{bestvina2008quasiflats}.

\begin{thm}[Theorem 5.20 of \cite{quasiflat}]
\label{2.10}
If $\phi: X(\Gamma_{1})\to X(\Gamma_{2})$ is an $(L, A)$-quasi-isometry, then $dim(X(\Gamma_{1}))=dim(X(\Gamma_{2}))$. And there is a constant $D=D(L, A)$ such that for any top-dimensional flat $F_{1}\subset X(\Gamma_{1})$, there is a unique flat $F_{2}\subset X(\Gamma_{2})$ such that $d_{H}(\phi(F_{1}),F_{2})< D.$
\end{thm}

For each right-angled Artin group $G(\Gamma)$, there is a simplicial graph $\Gamma^{e}$, called the \textit{extension graph}, which is introduced in \cite{kim2013embedability}. Extension graphs can be viewed as \textquotedblleft curve graphs\textquotedblright\ for RAAG's (\cite{kim2014geometry}).

\begin{definition}[Definition 1 of \cite{kim2013embedability}]
\label{2.11}
The vertex set of $\Gamma^{e}$ consists of words in $G(\Gamma)$ that are conjugate to elements in $S$ (recall that $S$ is a standard generating set for $G(\Gamma)$), and two vertices are adjacent in $\Gamma^{e}$ if and only if the corresponding words commute in $G(\Gamma)$.
\end{definition}

The flag complex of the extension graph is called the \textit{extension complex}.

Since the curve graph captures the combinatorial pattern of how Dehn twist flats intersect in a mapping class group, it plays an important role in the quasi-isometric rigidity of mapping class group (\cite{hamenstaedt2005geometry,MR2928983}). Similarly, we will see in Section \ref{sec_qi implies iso} that the extension graph captures the combinatorial pattern of the coarse intersection of certain collection of flats in a RAAG, and it is a quasi-isometric invariant for certain classes of RAAG's.

\section{Stable subgraph}
\label{sec_stable subgraph}
We study the behavior of certain standard subcomplexes under quasi-isometries in this section.

\subsection{Coarse intersection of standard subcomplexes and flats}
\label{subsec_standard subcomplexes and tree products}
\begin{lem}
\label{3.1}
Let $\Gamma$ be a finite simplicial graph and let $K_{1}$, $K_{2}$ be two standard subcomplexes of $X(\Gamma)$. If $(Y_{1},Y_{2})=\inc (K_{1},K_{2})$, then $Y_{1}$ and $Y_{2}$ are also standard subcomplexes.
\end{lem}

\begin{proof}
The lemma is clear if $K_{1}\cap K_{2}\neq\emptyset$. Now we assume $d(K_{1},K_{2})=c>0$. Pick a vertex $v_{1}\in K_{1}$, by Lemma \ref{2.3}, there exists vertex $v_{2}\in K_{2}$ such that $d(v_{1},v_{2})=c$. Let $l:[0,c]\to X(\Gamma)$ be the unit speed geodesic from $v_{1}$ to $v_{2}$. We can find sequence of cubes $\{B_{i}\}_{i=1}^{N}$ and $0=t_{0}<t_{1}<...<t_{N-1}<t_{N}=c$ such that each $B_{i}$ contains $\{l(t)\mid t_{i-1}<t<t_{i}\}$ as interior points.

Let $V_{l}=\cup_{i=1}^{N}V_{B_{i}}$ (recall that $V_{B_i}$ is the collection of the labels of edges in $B_i$, cf. Section \ref{subsec_raag}) and let $V_{i}=V_{K_{i}}$ for $i=1,2$. Put $V'=V_{1}\cap V_{2}\cap V_{l}^{\perp}$ ($V_{l}^{\perp}$ denotes the orthogonal complement of $V_1$, see Section \ref{subsec_notation}) and let $\Gamma'$ be the full subgraph spanned by $V'$. Let $Y'_{1}$ be the standard subcomplex that has defining graph $\Gamma'$ and contains $v_1$ (if $V'$ is empty, then $Y'_{1}=v_{1}$). We claim $Y_{1}=Y'_{1}$.

Pick an edge $e\subset K_{1}$ such that $v_{1}\in e$ and $V_{e}\in V'$. Let $h$ be the hyperplane dual to $e$ and $N_{h}\cong h\times[0,1]$ be the carrier of $h$. Since $d(V_{e},w)=1$ for any $w\in V_{l}$, we can assume $l\subset h\times\{1\}\subset N_{h}$. By our definition of $V'$, there is an edge $e'\in K_{2}$ with $v_{2}\in e'$ and $h$ dual to $e'$, thus $e$ and $e'$ cobound an isometrically embedded flat rectangle (one side of the rectangle is $l$), which implies $e\subset Y_{1}$. Let $l'$ be the side of the rectangle opposite to $l$. We can define $V_{l'}$ similarly as we define $V_{l}$, then $V_{l'}=V_{l}$. Now let $\omega$ be any edge path starting at $v_{1}$ such that $V_{e'}\in V'$ for any edge $e'\subset \omega$. Then it follows from the above argument and induction on the combinatorial length of $\omega$ that $\omega\subset Y_{1}$, thus $Y'_{1}\subset Y_{1}$. 

For the other direction, since $Y_{1}$ is a convex subcomplex by Lemma \ref{2.3}, it suffices to prove every vertex of $Y_{1}$ belongs to $Y'_{1}$. By the induction argument as above, we only need to show for edge $e_{1}$ with $v_{1}\in e_{1}$, if $e_{1}\subset Y_{1}$, then $e_{1}\subset Y'_{1}$. Lemma \ref{2.3} implies that there exists edge $e_{2}\subset Y_{2}$ such that $e_{1}$ and $e_{2}$ cobound an isometrically embedded flat rectangle (one side of the rectangle is $l$). So $l$ is in the carrier of the hyperplane dual to $e_{1}$. It follows that $V_{e_{1}}\in V'$ and $e_{1}\subset Y'_{1}$.
\end{proof}

\begin{cor}
\label{3.2}
Let $K_{1},K_{2},Y_{1}$ and $Y_{2}$ be as above. Let $h$ be a hyperplane separating $K_{1}$ and $K_{2}$ and let $e$ be an edge dual to $h$. Then $V_{e}\in V^{\perp}_{Y_{1}}=V^{\perp}_{Y_{2}}$. In particular, pick vertex $v\in\Gamma$, then $v\in V_{Y_{1}}$ if and only if
\begin{enumerate}
\item $v\in V_{K_{1}}\cap V_{K_{2}}$.
\item For any hyperplane $h'$ separating $K_{1}$ from $K_{2}$ and any edge $e'$ dual to $h'$, $d(v,V_{e'})=1$.
\end{enumerate}
\end{cor}

\begin{proof}
Let $l$ and $V_l$ be the same as the proof of Lemma \ref{3.1}. Let $V'_{l}$ be a collection of vertices of $\Gamma$ such that $v\in V'_{l}$ if and only if $v=V_{e'}$ for some edge $e'\subset X(\Gamma)$ satisfying (2). It suffices to prove $V'_{l}=V_{l}$.

$V'_{l}\subset V_{l}$ is clear since if a hyperplanes $h$ separates $K_{1}$ from $K_{2}$, then $l$ intersects $h$ transversally at one point. To see $V_{l}\subset V'_{l}$, it suffices to show $h\cap K_{i}=\emptyset$ for $i=1,2$, where $h$ is a hyperplane that intersects $l$ transversally. Let $x=l\cap h$. Suppose $h\cap K_{1}\neq\emptyset$ and let $x'=\pi_{h\cap K_{1}}(x)$. Now consider the triangle $\Delta(v_{1},x,x')$ (recall that $v_{1}=l(0)$), we have $\angle_{v_{1}}(x,x')\ge\pi/2$ (since $\pi_{K_{1}}(x)=v_{1}$), $\angle_{x'}(v_{1},x)\ge\pi/2$ (see the proof of Lemma \ref{2.6}) and $\angle_{x}(v_{1},x')>0$, which is a contradiction, so $h\cap K_{1}=\emptyset$, similarly $h\cap K_{2}=\emptyset$.
\end{proof}

\begin{remark}
\label{3.3}
Recall that a standard coset of $G(\Gamma)$ is a left coset of a standard subgroup of $G(\Gamma)$. Lemma \ref{3.1} implies that for each pair of standard cosets of $G(\Gamma)$, we can associated another standard coset, which captures the coarse intersection of the pair. Moreover, we can also define a notion of distance between two standard cosets, which takes value on $G(\Gamma)$.
\end{remark}

\begin{lem}
\label{3.4}
Let $K\subset X(\Gamma)$ be a convex subcomplex and let $\Gamma'=lk(\Gamma_{K})$ ($\Gamma_{K}$ is the support of $K$, see Section \ref{subsec_raag}; $lk(\Gamma_{K})$ is the full subgraph spanned by $V_{K}^{\perp}$, see Section \ref{subsec_notation}). Then the parallel set $P_{K}$ of $K$ is a convex subcomplex and canonically splits as $K\times X(\Gamma')$.
\end{lem}

Note that we do not require $K$ to satisfy geodesic extension property.
\begin{proof}
Pick a vertex $v\in K$. Let $\Gamma''=\Gamma_{K}$ and let $P_{1}$ be the unique standard subcomplex that passes through $v$ and has defining graph $\Gamma'\circ\Gamma''$ (recall that $\circ$ denotes the graph join). Then $K\subset P_{1}$. Let $P'$ be the natural copy of $K\times X(\Gamma')$ inside $P_{1}$. It is clear that  $P'\subset P_{K}$.

Let $K'$ be a convex subset parallel to $K$ and let $\phi:K\to K'$ be the isometry induced by $CAT(0)$ projection onto $K'$. Pick vertex $v\in K$ and let $l$ be the geodesic segment connecting $v$ and $\phi(v)$. We define $V_{l}$ as in the proof of Lemma \ref{3.1} (note that $\phi(v)$ is not necessarily a vertex). Let $e$ be any edge such that $v\in e\subset K$. Then there is a flat rectangle with $e$, $\phi(e)$ and $l$ as its three sides. Thus $l$ is contained in the carrier of the hyperplane dual to $e$ and $V_{l}\subset V_{e}^{\perp}$. Note that if $l'$ is the side opposite to $l$, then $V_{l'}=V_{l}$. For any given edge $e'\subset K$, we can find an edge path $\omega\subset K$ such that $e$ is the first and $e'$ is the last edge in $\omega$. By induction on the combinatorial length of $w$ and the same argument as above, we can show $V_{l}\subset V_{e'}^{\perp}$, thus $V_{l}\subset V_{K}^{\perp}$ and $K'\subset P'$. It follows that $P_{K}\subset P'$, so $P_{K}=P'$.
\end{proof}

\begin{remark}
\label{3.5}
The following is a generalization of the above lemma for general $CAT(0)$ cube complexes. Let $X$ be a $CAT(0)$ cube complex. A convex set $K\subset X$ is \textit{regular} if for any $x\in K$, $\Sigma_{x}K$ (the space of direction of $x$ in $K$, see Chapter II.3 of \cite{MR1744486}) satisfies:
\begin{enumerate}
\item $\Sigma_{x}K$ is a subcomplex of $\Sigma_{x}X$ with respect to the canonical all-right spherical complex structure on $\Sigma_{x}X$.
\item There exists $r>0$ such that $B(x,r)\cap K$ is isometric to the $r$-ball centred at the cone point in the Euclidean cone over $\Sigma_{x}K$.
\end{enumerate}
If $K\subset X$ is a regular convex subset, the $P_{K}$ is convex and admits a splitting $P_{K}\cong K\times N$ where $N$ has an induced cubical structure from $X$ ($N$ is $CAT(0)$).
\end{remark}

\begin{lem}
\label{3.6}
Let $q:X(\Gamma_{1})\to X(\Gamma_{2})$ be an $(L,A)$-quasi-isometry and let $F\subset X(\Gamma_{1})$ be a subcomplex isometric to $\Bbb E^{k}$. Suppose $n=dim(X(\Gamma_{1}))=dim(X(\Gamma_{2}))$. If there exist $R_{1}>0$, $R_{2}>0$ and top dimensional flats $F_{1}$, $F_{2}$ such that $F\overset{R_{2}}{=}F_{1}\cap_{R_{1}}F_{2}$ and $F\overset{\infty}{=}F_{1}\cap_{R}F_{2}$ for any $R\ge R_{1}$, then 
\begin{enumerate}
\item There exist a constant $D=D(L,A,R_{1},R_{2},n)$ and a subcomplex $F'\subset X(\Gamma_{2})$ isometric to $\Bbb E^{k}$ such that $q(F)\overset{D}{=}F'$.
\item There exists a constant $D'=D'(L,A)$ such that $q(P_{F})\overset{D'}{=}P_{F'}$.
\end{enumerate}
\end{lem}

\begin{proof}
By Theorem \ref{2.10}, there exist top dimensional flats $F'_{1}\subset X(\Gamma_{2})$ and $F'_{2}\subset X(\Gamma_{2})$ such that $q(F_{i})\overset{D_{1}}{=}F'_{1}$ for $D_{1}=D_{1}(L,A)$ and $i=1,2$. Thus there exists $R'=R'(L,A,R_{1},R_{2})$ and $R_{3}=R_{3}(L,A,R_{1},R_{2})>R_{1}$ such that $q(F_{1}\cap_{R_{1}}F_{2})\subset F'_{1}\cap_{R'}F'_{2}\subset q(F_{1}\cap_{R_{3}}F_{2})$, this and Remark \ref{2.5} imply
\begin{equation}
\label{3.7}
q(F_{1}\cap_{R_{1}}F_{2})\overset{D_{2}}{=}F'_{1}\cap_{R'}F'_{2}
\end{equation}
for $D_{2}=D_{2}(n,d(F_{1},F_{2}))=D_{2}(L,A,R_{1},R_{2},n)$.

Let $(Y_{1},Y_{2})=\inc (F'_{1},F'_{2})$. Then there exists $D_{3}=D_{3}(L,A,R_{1},R_{2},n)$ such that
\begin{equation}
\label{3.8}
Y_{1}\overset{D_{3}}{=}F'_{1}\cap_{R'}F'_{2}.
\end{equation}
From (\ref{3.7}) and (\ref{3.8}), we have
\begin{equation}
\label{3.9}
q(F)\overset{D_{4}}{=}Y_{1}
\end{equation}
for $D_{4}=D_{4}(L,A,R_{1},R_{2},n)$. By Lemma \ref{2.3}, $Y_{1}$ is a convex subcomplex of $F'_{1}$. This together with (\ref{3.9}) imply $Y_1=F'\times\prod_{i=1}^{k'}I_{i}$ where $F'$ is isometric to $\Bbb E^{k}$ and $\{I_{i}\}_{i=1}^{k'}$ are finite intervals. Moreover, by (\ref{3.9}), $diam(\prod_{i=1}^{k'}I_{i})$ must be bounded in terms of $D_{4},L$ and $A$, thus (1) follows.

Let $\{F_{\lambda}\}_{\lambda\in\Lambda}$ be the collection of top dimensional flats in $X(\Gamma_{1})$ which are contained in the parallel set $P_{F}$ of $F$. Lemma \ref{3.4} implies 
\begin{equation}
\label{3.10}
d_{H}(\cup_{\lambda\in\Lambda}F_{\lambda},P_{F})\leq 1.
\end{equation} 
For $\lambda\in\Lambda$, there exists $R_{\lambda}>0$ such that $F\subset_{R_{\lambda}} F_{\lambda}$. Let $F'_{\lambda}$ be the top dimensional flat in $X(\Gamma_{2})$ such that $q(F_{\lambda})\overset{D_{1}}{=}F'_{\lambda}$. Then by (1), there exists $R'_{\lambda}>0$ such that $F'\subset_{R'_{\lambda}}(F'_{\lambda})$. This and Lemma \ref{2.3} imply $F'_{\lambda}\subset P_{F'}$ for any $\lambda\in\Lambda$. Thus by (\ref{3.10}), there exists $D'=D'(L,A)$ such that $q(P_{F})\subset_{D'} P_{F'}$. And (2) follows by running the same argument for the quasi-isometry inverse of $q$.
\end{proof}

A \textit{tree product} is a convex subcomplex $K\subset X(\Gamma)$ such that $K$ splits as a product of trees, i.e. there exists a cubical isomorphism $K\cong\Pi_{i=1}^{n}T_{i}$ where $T_{i}$'s are trees. A \textit{standard tree product} is a tree product which is also a standard subcomplex. 

One can check that $K$ is a standard tree product if and only if the defining graph $\Gamma_{K}$ of $K$ has a join decomposition $\Gamma_{K}=\Gamma_{1}\circ\Gamma_{2}\circ\cdots\circ\Gamma_{n}$ where each $\Gamma_{i}$ is discrete. Thus one can choose the above $T_{i}$'s to be standard subcomplexes of $K$. Note that every standard flat is a standard tree product, and every subcomplex isometric to $\Bbb E^{k}$ is a tree product. 

\begin{lem}
\label{3.11}
Suppose $\dim(X(\Gamma))=n$. Let $q:X(\Gamma)\to X(\Gamma')$ be a quasi-isometry. Let $K=\prod_{i=1}^{n}T_{i}$ be a top dimensional tree product with its tree factors. Then there exists a standard tree product $K'$ in $X(\Gamma')$ such that $q(K)\subset_{\infty} K'$.
\end{lem}

The proof essentially follows \cite[Theorem 4.2]{MR2421136}.

\begin{proof}
For $1\le i\le n$, let $V_{i}=V_{T_{i}}\in\Gamma$ be the collection of labels of edges in $T_{i}$. The case where all $V_{i}$'s are consist of one point follows from Theorem \ref{2.10}. If each $V_{i}$ contains at least two points, then by Lemma \ref{3.6}, for any geodesic $l\subset T_{i}$, there exists a subcomplex $l'\subset X(\Gamma')$ isometric to $\Bbb R$ such that $q(l)\overset{\infty}{=}l'$. Since $l'$ is unique up to parallelism, the collection of labels of edges in $l'$ does not depend on the choice of $l'$ and will be denoted by $V_{q(l)}$. For $1\le i\le n$, define $V'_{i}=\cup_{l\subset T_{i}}V_{q(l)}$ where $l$ varies among all geodesics in $T_{i}$. 

We claim $V'_{i}\subset (V'_{j})^{\perp}$ for $i\neq j$. To see this, pick geodesic $l_{i}\in T_{i}$ and let $F=\prod_{i=1}^{n}l_{i}$. Then there exist top dimensional flat $F'$ and geodesic lines $\{l'_{i}\}_{i=1}^{n}$ ($l'_{i}$ is a subcomplex) in $X(\Gamma')$ such that $q(F)\overset{\infty}{=}F'$ and $q(l_{i})\overset{\infty}{=}l'_{i}$. Since $l'_{i}\subset_{\infty}F'$, by Lemma \ref{2.3}, we can assume $l'_{i}$ is a subcomplex of $F'$. Pick $i\neq j$. Since $l_i$ and $l_j$ are orthogonal, they have infinite Hausdorff distance. Thus $l'_i$ and $l'_j$ have infinite Hausdorff distance. By our assumption, $l'_i$ and $l'_j$ are isometric to $\mathbb E^1$, and they are convex subcomplexes of $F'\cong \mathbb E^n$. Thus either $l'_i$ and $l'_j$ are parallel, or they are orthogonal. The former is impossible since $l'_i$ and $l'_j$ have infinite Hausdorff distance. Thus $\{l'_{i}\}_{i=1}^{n}$ is a mutually orthogonal collection. 

Let $\Gamma'_{1}=V'_{1}\circ V'_{2}\circ\cdots\circ V'_{n}\subset\Gamma'$. Then each $V'_{i}$ has to be a discrete full subgraph by our dimension assumption. Let $\{F_{\lambda}\}_{\lambda\in\Lambda}$ be the collection of top dimensional flats in $K$ and let $F'_{\lambda}$ be the unique flat such that $q(F_{\lambda})\overset{\infty}{=}F'_{\lambda}$. Note that for arbitrary $F_{\lambda_{1}}$ and $F_{\lambda_{2}}$, there exists a finite chain in $\{F_{\lambda}\}_{\lambda\in\Lambda}$ which starts with $F_{\lambda_{1}}$ and ends with $F_{\lambda_{2}}$ such that the intersection of adjacent elements in the chain contains a top dimensional orthant. Thus the collection $\{F'_{\lambda}\}_{\lambda\in\Lambda}$ also have this property. Then $\cup_{\lambda\in\Lambda}F'_{\lambda}$ is contained in a standard subcomplex of $X(\Gamma')$ with defining graph $\Gamma'_{1}$.

It remains to deal with the case where there exist $i\neq j$ such that $|V_{i}|=1$ and $|V_{j}|\ge 2$. We suppose $|V_i|=1$ for $1\le i\le m$ and $|V_{i}|\ge 2$ for $i>m$. By applying Lemma \ref{3.6} with $F=\prod_{i=1}^{m}T_{i}$, we can reduce to lower dimensional case and the lemma follows by induction on dimension.
\end{proof}

\begin{cor}
\label{3.12}
Let $q:X(\Gamma)\to X(\Gamma')$ be a quasi-isometry and let $K$ be a top dimensional maximal standard tree product, i.e. $K$ is not properly contained in another tree product. Then there exists a standard tree product $K'\subset X(\Gamma')$ such that $q(K)\overset{\infty}{=}K'$.
\end{cor}

\subsection{Standard flats in transvection free RAAG's}
\label{subsec_stability of std flats}
Up to now, we have only dealt with top dimensional standard subcomplexes. The next goal is to study those standard subcomplexes which are not necessarily top dimensional. In particular, we are interested in whether quasi-isometries will preserve standard flats up to finite Hausdorff distance. The answer turns out to be related to the outer automorphism group of $G(\Gamma)$.

One direction is obvious, namely, if for any quasi-isometry $q:X(\Gamma)\to X(\Gamma')$, $q$ maps any standard flat in $X(\Gamma)$ to a standard flat in $X(\Gamma')$ up to finite Hausdorff distance, then $\out(G(\Gamma))$ must be transvection free (i.e. $\out(G(\Gamma))$ does not contain any transvections). The converse is also true. Now we set up several necessary tools to prove the converse.

In this section, $\Gamma$ will always be a finite simplicial graph.

\begin{definition}
\label{3.13}
A subgraph $\Gamma_{1}\subset\Gamma$ is \textit{stable} in $\Gamma$ if 
\begin{enumerate}
	\item $\Gamma_{1}$ is a full subgraph.
	\item Let $K\subset X(\Gamma)$ be a standard subcomplex such that $\Gamma_K=\Gamma_1$ and let $\Gamma'$ be a finite simplicial graph so that, for some $L,A$, there is an $(L,A)$-quasi-isometry $q:X(\Gamma)\to X(\Gamma')$. Then there exists $D=D(L,A,\Gamma_1,\Gamma)>0$ and a standard subcomplex $K'\subset X(\Gamma')$ such that $d_{H}(q(K),K')<D$.
\end{enumerate}
For simplicity, we will also say the pair $(\Gamma_{1},\Gamma)$ is \textit{stable} in this case. A standard subcomplex $K\subset X(\Gamma)$ is \textit{stable} if it arises from a stable subgraph of $\Gamma$.
\end{definition}

We claim the defining graph $\Gamma_{K'}$ of $K'$ is stable in $\Gamma'$. To see this, pick any graph $\Gamma''$ so that there is an $(L,A)$-quasi-isometry $q':X(\Gamma')\to X(\Gamma'')$, and pick any standard subcomplex $K'_1\subset X(\Gamma')$ with defining graph $\Gamma_{K'}$. Note that there is an isometry $i:X(\Gamma')\to X(\Gamma')$ such that $i(K')=K'_1$. Since the map $q'\circ i\circ q$ is a quasi-isometry from $X(\Gamma)$ to $X(\Gamma'')$, then by the stability of $\Gamma_1$, $q'\circ i\circ q(K)$ is Hausdorff close to a standard subcomplex in $X(\Gamma'')$, hence the same is true for $q'(K'_1)$. It follows from this claim that one can obtain quasi-isometric invariants by identifying certain classes of stable subgraphs. 

It is immediate from the definition that for finite simplicial graphs $\Gamma_{1}\subset\Gamma_{2}\subset\Gamma_{3}$, if $(\Gamma_{1},\Gamma_{2})$ is stable and $(\Gamma_{2},\Gamma_{3})$ is stable, then $(\Gamma_{1},\Gamma_{3})$ is stable. However, it is not necessarily true that if $(\Gamma_{1},\Gamma_{3})$ and $(\Gamma_{2},\Gamma_{3})$ are stable, then $(\Gamma_{1},\Gamma_{2})$ is stable. In the sequel, we will investigate several other properties of stable subgraph. The following lemma is an easy consequence of Lemma \ref{3.1} and Remark \ref{2.5}:

\begin{lem}
\label{3.14}
Suppose $\Gamma_{1}$ and $\Gamma_{2}$ are stable in $\Gamma$. Then $\Gamma_{1}\cap\Gamma_{2}$ is also stable in $\Gamma$.
\end{lem}

The following result follows from Lemma \ref{2.8}.

\begin{lem}
\label{3.15}
If $\Gamma_{1}$ is stable in $\Gamma$, then every connected component of $\Gamma_{1}$ that contains more than one point is also stable in $\Gamma$.
\end{lem}

\begin{lem}
\label{3.16}
Suppose $\Gamma_{1}$ is stable in $\Gamma$. Let $V$ be the vertex set of $\Gamma_1$ and let $\Gamma_{2}$ be the full subgraph spanned by $V$ and the orthogonal complement $V^{\perp}$. Then $\Gamma_{2}$ is also stable in $\Gamma$.
\end{lem}

\begin{proof}
Let $K_{2}\subset X(\Gamma)$ be a standard subcomplex with its defining graph $\Gamma_{K_{2}}=\Gamma_{2}$ and let $K_{1}\subset K_{2}$ be any standard subcomplex satisfying $\Gamma_{K_{1}}=\Gamma_{1}$. Lemma \ref{3.4} implies $K_{2}=P_{K_{1}}=K_{1}\times K_{1}^{\perp}$. For vertex $x\in K_{1}^{\perp}$, denote $K_{x}=K_{1}\times\{x\}$. Let $q: X(\Gamma)\to X(\Gamma')$ be an $(L,A)$-quasi-isometry. Then there exists standard subcomplex $K'_{x}$ such that $d_{H}(q(K_{x}),K'_{x})<D=D(L,A,\Gamma_{1},\Gamma)$. Thus $K'_{x}\overset{\infty}{=}K'_{y}$ for vertices $x,y\in K_{1}^{\perp}$. It follows from Lemma \ref{3.1} that $K'_{x}$ and $K'_{y}$ are parallel. Thus $q(P_{K_{1}})\subset_{R} P_{K'_{x}}$ for $R=D+L+A$. Moreover, $P_{K'_{x}}$ is also a standard subcomplex by Lemma \ref{3.4}. By considering the quasi-isometry inverse and repeating the previous argument, we know $q(P_{K_{1}})\overset{\infty}{=}P_{K'_{x}}$, thus $\Gamma_{2}$ is also stable in $\Gamma$.
\end{proof}

\begin{lem}
\label{3.17}
Suppose $\Gamma_{1}$ is stable in $\Gamma$. Pick vertex $v\notin\Gamma_{1}$, then the full subgraph spanned by $v^{\perp}\cap\Gamma_{1}$ is stable in $\Gamma$.
\end{lem}

\begin{proof}
We use $\Gamma_{2}$ to denote the full subgraph spanned by $v^{\perp}\cap\Gamma_{1}$. Let $K_{2}\subset X(\Gamma)$ be a standard subcomplex such that $\Gamma_{K_{2}}=\Gamma_{2}$ and let $K_{1}\subset X(\Gamma)$ be the unique standard subcomplex such that $\Gamma_{K_{1}}=\Gamma_{1}$ and $K_{2}\subset K_{1}$. Pick vertex $x\in K_{2}$ and let $e\subset X(\Gamma)$ to be the edge such that $V_{e}=v$ and $x\in e$. Suppose $\bar{x}$ is the other end point of $e$. Let $\bar{K}_{i}$ be the standard subcomplex that contains $\bar{x}$ and has defining graph $\Gamma_{i}$ for $i=1,2$. Denote the hyperplane dual to $e$ by $h$. Since $v\notin\Gamma_{1}$, $h\cap K_{1}=\emptyset$ and $h\cap\bar{K}_{1}=\emptyset$, thus $h$ separates $K_{1}$ and $\bar{K}_{1}$ and $d(K_{1},\bar{K}_{1})=1$. It follows from Corollary \ref{3.2} that $\inc (K_{1},\bar{K}_{1})=(K_{2},\bar{K}_{2})$, in particular $K_{2}\overset{D}{=}K_{1}\cap_{R} \bar{K}_{1}$ for $D$ depending on $R$ and the dimension of $X(\Gamma)$. Now the lemma follows since $\Gamma_{1}$ is stable.
\end{proof}

The next result is a direct consequence of Corollary \ref{3.12}.

\begin{lem}
\label{3.18}
If $\Gamma_{1}$ is stable in $\Gamma$, then there exists $\Gamma_{2}$ which is stable in $\Gamma_{1}$ such that 
\begin{enumerate}
\item $\Gamma_{2}$ is a graph join $\bar{\Gamma}_{1}\circ\bar{\Gamma}_{2}\circ\cdots\circ\bar{\Gamma}_{k}$ where $\bar{\Gamma}_{i}$ is discrete for $1\le i\le k$.
\item $k=dim(X(\Gamma_{1}))$.
\end{enumerate}
\end{lem}

\begin{lem}
\label{3.19}
Let $\Gamma$ be a finite simplicial graph such that there do not exist vertices $v\neq w$ of $\Gamma$ such that $v^{\perp}\subset St(w)$. Then every stable subgraph of $\Gamma$ contains a stable vertex. 
\end{lem}

\begin{proof}
Let $\Gamma_{1}$ be a minimal stable subgraph, i.e. it does not properly contain any stable subgraph of $\Gamma$. It suffices to show $\Gamma_{1}$ is a point. 
We argue by contradiction and assume $\Gamma_{1}$ contains more than one point.

First we claim $\Gamma_{1}$ can not be discrete. Suppose the contrary is true. Pick vertices $v,w\in\Gamma_{1}$ and pick vertex $u\in v^{\perp}\setminus St(w)$. By Lemma \ref{3.17}, $u^{\perp}\cap\Gamma_{1}$ is also stable. Note that $v\in u^{\perp}\cap\Gamma_{1}$ and $w\notin u^{\perp}\cap\Gamma_{1}$, which contradicts the minimality of $\Gamma_{1}$.

We claim $\Gamma_{1}$ must be a clique. Since $\Gamma_{1}$ is not discrete, by Lemma \ref{3.18}, we can find a stable subgraph
\begin{equation}
\label{3.20}
\Gamma_{2}=\bar{\Gamma}_{1}\circ\bar{\Gamma}_{2}\circ\cdots\circ\bar{\Gamma}_{m}\subset\Gamma_{1}
\end{equation}
where $\{\bar{\Gamma}_{i}\}_{i=1}^{m}$ are discrete full subgraphs and $m\geq 2$. Then $\Gamma_{2}=\Gamma_{1}$. Suppose some $\bar{\Gamma}_{i}$ contains more than one point, and let $\Gamma_{3}$ be the join of the remaining join factors. Then Theorem \ref{2.9} implies that $\Gamma_{3}$ is stable, contradicting the minimality of $\Gamma_{1}$. Therefore $\Gamma_{1}$ is a clique.

Pick distinct vertices $v_{1},v_{2}\in\Gamma_{1}$. By our assumption, there exists vertex $w\in v_{1}^{\perp}\setminus St(v_{2})$. Since $\Gamma_{1}$ is a clique, $\Gamma_{1}\subset St(v_{2})$, then $w\notin \Gamma_{1}$. Let $\Gamma_{4}$ be the full subgraph spanned by $w^{\perp}\cap \Gamma_{1}$. Then $\Gamma_{4}$ is stable by Lemma \ref{3.17}. Moreover, $\Gamma_{4}\subsetneq\Gamma_{1}$ (since $v_{2}\notin\Gamma_{4}$), which yields a contradiction.
\end{proof}

\begin{lem}
\label{3.22}
Let $\Gamma$ be as in Lemma \ref{3.19} and let $\Gamma_{1}$ be a stable subgraph of $\Gamma$. Then for any vertex $w\in\Gamma_{1}$, there exists a stable vertex $v\in\Gamma_{1}$ such that $d(v,w)\le 1$.
\end{lem}

\begin{proof}
Denote the combinatorial distance in $\Gamma$ and $\Gamma_{1}$ by $d$ and $d_{1}$ respectively. Since $\Gamma_{1}$ is a full subgraph, $d(x,y)=1$ if and only if $d_{1}(x,y)=1$ and $d(x,y)\ge 2$ if and only if $d_{1}(x,y)\ge 2$ for vertices $x,y\in\Gamma_{1}$. If $w$ is isolated in $\Gamma_1$, then we can use the argument in the second paragraph of the proof of Lemma \ref{3.19} to get rid of all vertices in $\Gamma_1$ except $w$, which implies $w$ is a stable vertex. If $w$ is not isolated, we can assume $\Gamma_{1}$ is connected by Lemma \ref{3.15}.

By Lemma \ref{3.19}, there exists a stable vertex $u\in\Gamma_{1}$. If $d_{1}(u,w)\le 1$, then we are done, otherwise let $\omega$ be a geodesic in $\Gamma_{1}$ connecting $u$ and $w$ ($\omega$ might not be a geodesic in $\Gamma$) and let $\{v_{i}\}_{i=0}^{n}$ be the consecutive vertices in $\omega$, here $v_{0}=w$, $v_{n}=u$ and $n=d_{1}(w,u)$.

Since $u$ is stable, by Lemma \ref{3.16}, $St(u)$ is also stable. Note that $d_{1}(v_{n-2},u)=2$, so $d(v_{n-2},u)=2$ and $v_{n-2}\notin St(u)$. Lemma \ref{3.17} implies $v_{n-2}^{\perp}\cap St(u)$ is stable and by Lemma \ref{3.14}, $v_{n-2}^{\perp}\cap St(u)\cap\Gamma_{1}$ is also stable. Note that $v_{n-2}^{\perp}\cap St(u)\cap\Gamma_{1}\neq\emptyset$ since it contains $v_{n-1}$. Lemma \ref{3.19} implies there is a stable vertex $u'\in v_{n-2}^{\perp}\cap St(u)\cap\Gamma_{1}$ and it is easy to see $d_{1}(w,u')=n-1$. Now the lemma follows by induction.
\end{proof}

\begin{lem}
\label{3.23}
Let $\Gamma$ be as in Lemma \ref{3.19}. Then every vertex of $\Gamma$ is stable.
\end{lem}

\begin{proof}
Let $\Gamma_{w}$ be the intersection of all stable subgraphs that contain $w$. By Lemma \ref{3.14}, $\Gamma_{w}$ is the minimal stable subgraph that contains $w$. It suffices to prove $\Gamma_{w}=\{w\}$. We argue by contradiction and denote the vertices in $\Gamma_{w}\setminus\{w\}$ by $\{v_{i}\}_{i=1}^{k}$. The minimality of $\Gamma_{w}$ implies we can not use Lemma \ref{3.17} to get rid of some $v_{i}$ while keep $w$, thus $w^{\perp}\setminus St(v_{i})\subset\{v_{1}\cdots v_{i-1},v_{i+1}\cdots v_{k}\}$ for any $i$, in other words
\begin{equation}
\label{3.24}
w^{\perp}\subset St(v_{i})\cup\{v_{1}\cdots v_{i-1},v_{i+1}\cdots v_{k}\}
\end{equation}
for $1\le i\le k$. Then there does not exist $i$ such that $\Gamma_{w}\subset St(v_{i})$, otherwise we would have $w^{\perp}\subset St(v_{i})$ by (\ref{3.24}).

On the other hand, Lemma \ref{3.22} implies there exists a stable vertex $u\in\Gamma_{w}$ with $d(w,u)=1$. Then $St(u)$ is stable (Lemma \ref{3.16}) and $St(u)\cap\Gamma_{w}$ is stable (Lemma \ref{3.14}). Note that $w\in St(u)\cap\Gamma_{w}$, by the minimality of $\Gamma_{w}$, $\Gamma_{w}\subset St(u)$, which yields a contradiction.
\end{proof}

\begin{lem}
\label{3.25}
Let $\Gamma$ be a finite simplicial graph and pick stable subgraphs $\Gamma_{1},\Gamma_{2}$ of $\Gamma$. Let $\bar{\Gamma}$ be the full subgraph spanned by $V$ and $V^{\perp}$ where $V=V_{\Gamma_{1}}$. If $\Gamma_{2}\subset\bar{\Gamma}$, then the full subgraph spanned by $\Gamma_{1}\cup\Gamma_{2}$ is stable in $\Gamma$. 
\end{lem}

To simplify notation, in the following proof, we will denote $q(K)\approx K'$ where $q,K$ and $K'$ are as in Definition \ref{3.13}. We will also assume without loss of generality that $q(K)\subset K'$. 

\begin{proof}
Let $q:X(\Gamma)\to X(\Gamma')$ be an $(L,A)$-quasi-isometry. Suppose $K_{1}$ and $K$ are standard subcomplexes in $X(\Gamma)$ such that $\Gamma_{K_{1}}=\Gamma_{1}$, $\Gamma_{K}=\bar{\Gamma}$ and $K_{1}\subset K$. Put $K'\approx q(K)$, $K'_{1}\approx q(K_{1})$, $K=K_{1}\times K_{1}^{\perp}$ and $K'=K'_{1}\times K'^{\perp}_{1}$. The proof of Lemma \ref{3.16} implies there exist a quasi-isometry $q':K_{1}^{\perp}\to K'^{\perp}_{1}$ and a constant $D_{1}=D_{1}(L,A,\Gamma_{1},\Gamma)$ such that
\begin{equation}
\label{3.26}
d(q'\circ p_{2}(x),p'_{2}\circ q(x))<D_{1}
\end{equation}
for any $x\in K$ where $p_{2}:K\to K^{\perp}_{1}$ and $p'_{2}:K'\to K'^{\perp}_{1}$ are projections.

Let $\Gamma_{2}=\Gamma_{21}\circ\Gamma_{22}$ where $\Gamma_{21}=\Gamma_{1}\cap\Gamma_{2}$ and let $K_{22},K_{2}$ be standard subcomplexes such that $\Gamma_{K_{22}}=\Gamma_{22}$, $\Gamma_{K_{2}}=\Gamma_{2}$ and $K_{22}\subset K_{2}\subset K$. By (\ref{3.26}), it suffices to prove there exist a standard subcomplex $K'_{22}\subset K'$ and a constant $D=D(L,A,\Gamma_{1},\Gamma_{2},\Gamma)$ such that $d_{H}(p'_{2}\circ q(K_{22}),K'_{22})<D$. Let $K'_{2}\approx q(K_{2})$. Then $K'_{2}\subset K'$ and $p'_{2}(K'_{2})$ is a standard subcomplex. By (\ref{3.26}), $p'_{2}\circ q(K_{22})\overset{\infty}{=}p'_{2}\circ q(K_{2})\overset{\infty}{=}p'_{2}(K'_{2})$, thus we can take $K'_{22}=p'_{2}(K'_{2})$.
\end{proof}

\begin{remark}
\label{3.27}
In general the full subgraph spanned by $\Gamma_{1}\cup\Gamma_{2}$ is not necessarily stable even if $\Gamma_{1}$ and $\Gamma_{2}$ are stable, see Remark \ref{3.36}. 
\end{remark}

The next theorem follows from Lemma \ref{3.23} and Lemma \ref{3.25}.

\begin{thm}
\label{3.28}
Given finite simplicial graph $\Gamma$, the following are equivalent:
\begin{enumerate}
\item $\out(G(\Gamma))$ is transvection free.
\item For any $(L,A)$-quasi-isometry $q: X(\Gamma)\to X(\Gamma')$, there exists positive constant $D=D(L,A,\Gamma)$ such that for any standard flat $F\subset X(\Gamma)$, there exists a standard flat $F'\subset X(\Gamma')$ such that $d_{H}(q(F),F')<D$.
\end{enumerate}
\end{thm}

\subsection{Standard flats in general RAAG's}
At this point, we have the following natural questions:
\begin{enumerate}
\item In Theorem \ref{3.28}, is it true that every standard flat in $X(\Gamma')$ comes from some standard flat in $X(\Gamma)$? A related question could be, is condition (1) in Theorem \ref{3.28} a quasi-isometric invariant for right-angled Artin groups?
\item What can we say about the stable subgraphs of $\Gamma$ if we drop condition (1) in Theorem \ref{3.28}? 
\end{enumerate} 
We will first give a negative answer to question (1) in Example \ref{3.29} below. Then we will prove the Theorem \ref{3.35}, which answers question (2). Section \ref{sec_qi implies iso}, and in particular the proof of Theorem \ref{1.1}, will not depend on this subsection. However, we will need Theorem \ref{3.35} and Lemma \ref{3.32} for Section \ref{sec_qi and special subgroups}.

\begin{thm}
\label{3.35}
Let $\Gamma$ be an arbitrary finite simplicial graph. A clique $\Gamma_{1}\subset\Gamma$ is stable if and only if there do not exist vertices $w\in\Gamma_{1}$ and $v\in\Gamma\setminus\Gamma_{1}$ such that $w^{\perp}\subset St(v)$.
\end{thm}

In other words, the clique $\Gamma_{1}$ is stable if and only if the corresponding $\Bbb Z^{n}$ subgroup of $G(\Gamma_{1})$ is invariant under all transvections.

\begin{example}
\label{3.29}
Let $\Gamma_{1}$ be the graph on the left and let $\Gamma_{2}$ be the one on the right. It is easy to see $\out(G(\Gamma_{1}))$ is transvection free while $\out(G(\Gamma_{2}))$ contains non-trivial transvection ($\Gamma_{2}$ has a dead end at vertex $u$). We claim $G(\Gamma_{1})$ and $G(\Gamma_{2})$ are commensurable, and in particular quasi-isometric. 
\
\begin{center}
\begin{tikzpicture}
\tikzstyle{every node}=[circle, draw, inner sep=0pt, minimum width=5pt]
\node (n1) at (2,0) {v};
\node (n2) at (1,1) {w};
\node (n3) at (0,0) {z};
\node (n4) at (0.5,-1) {};
\node (n5) at (1.5,-1) {};
\node (n7) at (2.5,-1) {};
\node (n8) at (3.5,-1) {};
\node (n9) at (4,0) {};
\node (n10) at (3,1) {k};
\node (n11) at (7,1) {};
\node (n12) at (8,1) {};
\node (n13) at (7,-1) {};
\node (n14) at (8,-1) {};
\node (n15) at (7.5,0) {};
\node (n16) at (8.5,0) {};
\node (n17) at (9,1) {};
\node (n18) at (10,1) {};
\node (n19) at (9.5,0) {};
\node (n20) at (10.5,0) {};
\node (n21) at (9,-1) {};
\node (n22) at (10,-1) {};
\node (n23) at (8.5,-1.5) {u};
\node (n24) at (6.5,0) {};

\foreach \from/\to in {n1/n2,n2/n3,n3/n4,n4/n5,n5/n1,n1/n7,n7/n8,n8/n9,n9/n10,n10/n1,n24/n11,n11/n12,n12/n16,n16/n14,n13/n14,n24/n13,n24/n15,n15/n16,n16/n21,n16/n17,n17/n18,n20/n18,n16/n19,n19/n20,n21/n22,n22/n20,n16/n23}
\draw (\from) -- (\to);
\end{tikzpicture}
\end{center}
\

Let $\Gamma\subset\Gamma_{1}$ be the pentagon on the left side and let $Y$ be the Salvetti complex of $\Gamma$. Suppose $X_{1}=Y\sqcup Y\sqcup \Bbb (S^{1}\times [0,1])/\sim$, here the two boundary circles of the annulus are identified with two standard circles which are in different copies of $Y$. Then $\pi_{1}(X_{1})=G(\Gamma_{1})$. Define homomorphism $h_{1}:G(\Gamma)\to\Bbb Z/2$ by sending $w$ to the non-trivial element in $\Bbb Z/2$ and other generators to the identity element. Let $Y'$ be the 2-sheet cover of $Y$ with respect to $\ker(h_{1})$.

Define homomorphism $h_{2}:G(\Gamma_{1})\to\Bbb Z/2$ by sending $w$ and $k$ to the non-trivial element in $\Bbb Z/2$ and other generators to the identity element. Let $X$ be the 2-sheet cover of $X_{1}$ with respect to $\ker(h_{2})$. Then $X$ is made of two copies of $Y'$ and two annuli with the boundaries of the annuli identified with the $v$-circles in $Y'$ (each $Y'$ has two $v$-circles which cover the $v$-circle in $Y$), see the picture below.

$X$ is homotopy equivalent to a Salvetti complex. To see this, let $S_{w}$ be the circle in $Y'$ which covers the $w$-circle in $Y$ two times and let $S_{z}\vee S_{v}$ be a wedge of the two circles in $Y'$ which covers the wedge of $z$-circle and $v$-circle in $Y$. There is a copy of $S_{w}\times(S_{z}\vee S_{v})$ inside $Y'$. Let $I$ be a segment in $S_{w}$ such that its end points are mapped to the base point of $Y$ under the covering map. We collapse $I\times(S_{z}\vee S_{v})$ to $\{pt\}\times(S_{z}\vee S_{v})$ inside each copy of $Y'$ in $X$, and collapse one of the annuli in $X$ to a circle by killing the interval factor. Denote the resulting space by $X'$. Then $X'$ is homotopy equivalent to $X$ and the un-collapsed annulus in $X$ becomes a torus in $X'$. It is not hard to see $X'$ is a Salvetti complex with defining graph $\Gamma_{2}$.

Any standard geodesic in $X(\Gamma_{2})$ which comes from vertex $u$ is not Hausdorff close to a quasi-isometric image of some standard geodesic in $X(\Gamma_{1})$, since $u$ is not a stable vertex while every vertex in $\Gamma_{1}$ is stable.

\begin{center}
\includegraphics[scale=0.48]{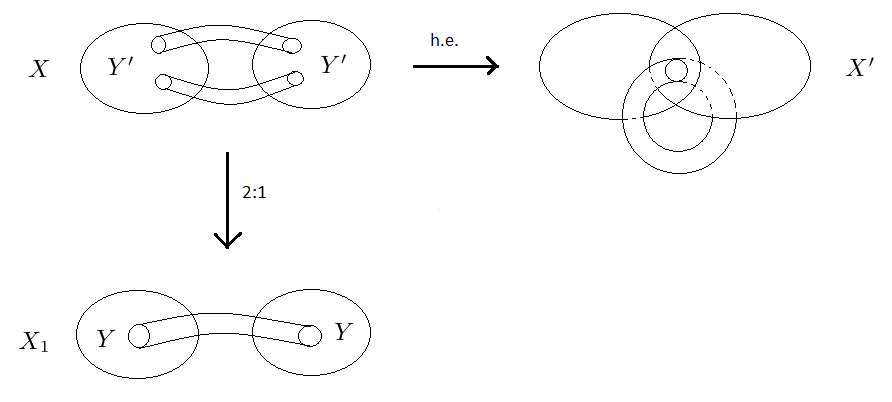}
\end{center}

Here is a generalization of the above example. Suppose $\Gamma$ is a finite simplicial graph such that there exist vertices $v_{1},v_{2}\in\Gamma$ with $d(v_{1},v_{2})=2$ so that they are separated by the intersection of links $lk(v_{1})\cap lk(v_{2})$. Define a homomorphism $h:G(\Gamma)\to\Bbb Z/2$ by sending $v_{1}$ and $v_{2}$ to the non-trivial element in $\mathbb Z/2$ and killing all other generators. Then $\ker(h)$ is also a right-angled Artin group by the same argument as before. To find its defining graph, let $\{C_{i}\}_{i=1}^{n}$ be the components of $\Gamma\setminus(lk(v_{1})\cap lk(v_{2}))$ and suppose $v_{1}\in C_{1}$. Define $\Gamma_{1}=C_{1}\cup (lk(v_{1})\cap lk(v_{2}))$ and $\Gamma_{2}=(\cup_{i=2}^{n}C_{i})\cup (lk(v_{1})\cap lk(v_{2}))$. Then $\Gamma_{1}$ and $\Gamma_{2}$ are full subgraphs of $\Gamma$, moreover, $St(v_{i})\in C_{i}$. For $i=1,2$, let $\Gamma'_{i}$ be the graph obtained by gluing two copies of $\Gamma_{i}$ along $St(v_{i})$ and let $\Gamma'_{3}$ be the join of one point and $lk(v_{1})\cap lk(v_{2})$. Then the defining graph of $\ker(h)$ can be obtained by gluing $\Gamma'_{1}$, $\Gamma'_{2}$ and $\Gamma'_{3}$ along $lk(v_{1})\cap lk(v_{2})$. 

Note that we are taking advantage of separating closed stars while constructing the counterexample. If separating closed stars are not allowed in $\Gamma$, then we have a positive answer to question (1) (see Section \ref{sec_qi and special subgroups}).
\end{example}

In the rest of this subsection, we will prove Theorem \ref{3.35}. $\Gamma$ will be an arbitrary finite simplicial graph in the rest of this subsection. Theorem \ref{3.35} is actually a consequence of the following more general result.
\begin{lem}
	\label{3.32}
	Pick a vertex $w\in\Gamma$ and let $\Gamma_{w}$ be the intersection of all stable subgraphs of $\Gamma$ that contain $w$. Define $W=\{w'\in\Gamma\mid w^{\perp}\subset St(w')\}$. Then $\Gamma_{w}$ is the full subgraph spanned by $W$.
\end{lem}

In other words, $G(\Gamma_{w})\le G(\Gamma)$ is the minimal standard subgroup containing $w$ with the property that $G(\Gamma_{w})$ is invariant under any transvection.

Now we show how to deduce Theorem \ref{3.35} from Lemma \ref{3.32}
\begin{proof}[Proof of Theorem \ref{3.35}]
	The only if part can be proved by contradiction (choose a transvection which does not preserve the subgroup $G(\Gamma_{1})$). For the converse, let $\{v_{i}\}_{i=1}^{n}$ be the vertex set of $\Gamma_1$ and let $\Gamma_{v_{i}}$ be the minimal stable subgraph that contains $v_{i}$ for $1\le i\le n$. By our assumption and Lemma \ref{3.32}, $\Gamma_{v_{i}}\subset \Gamma_{1}$. Thus the full subgraph spanned by $\cup_{i=1}^{n}\Gamma_{v_{i}}$ is stable by Lemma \ref{3.25}, which means $\Gamma_{1}$ is stable.
\end{proof}

It remains to prove Lemma \ref{3.32}. We first set up several auxiliary lemmas.

\begin{lem}
\label{3.30}
Let $v\in\Gamma$ be a vertex which is not isolated. Then at least one of the following is true:
\begin{enumerate}
\item $v$ is contained in a stable discrete subgraph with more than one vertex.
\item $v$ is contained in a stable clique subgraph.
\item There is a stable discrete subgraph with more than one vertex whose vertex set is in $v^{\perp}$.
\item There is a stable clique subgraph whose vertex set is in $v^{\perp}$.
\end{enumerate}
\end{lem}

\begin{proof}
Since $v$ is not isolated, we can assume $\Gamma$ is connected by Lemma \ref{3.15}. By Lemma \ref{3.18}, we can find a stable subgraph
$\Gamma_{1}=\bar{\Gamma}_{1}\circ\bar{\Gamma}_{2}\circ\cdots\circ\bar{\Gamma}_{n}$ where $\{\bar{\Gamma}_{i}\}_{i=1}^{n}$ are discrete full subgraphs and $n=dim(X(\Gamma))$. If $v\in\Gamma_{1}$, then by the third paragraph of the proof of Lemma \ref{3.19}, we know either (1), (2) or (4) is true.

Suppose $d(v,\Gamma_{1})=1$. Let $\Gamma_{2}$ be the full subgraph spanned by $v^{\perp}\cap\Gamma_{1}$. Then $\Gamma_{2}$ is stable by Lemma \ref{3.17}. The proof of Lemma \ref{3.19} implies every stable subgraph of $\Gamma$ contains either a stable discrete subgraph or a stable clique subgraph (this does not depend on the $v^{\perp}\nsubseteq St(w)$ assumption), thus either (3) or (4) is true.

Suppose $d(v,\Gamma_{1})\ge 2$. Pick vertex $u\in\Gamma_{1}$ such that $d(v,u)=d(v,\Gamma_{1})=n$ and let $\omega$ be a geodesic connecting $v$ and $u$. Suppose $\{v_{i}\}_{i=0}^{n}$ are the consecutive vertices in $\omega$ such that $v_{0}=v$ and $v_{n}=u$. Let $\Gamma'$ be the full subgraph spanned by $v_{n-1}^{\perp}\cap\Gamma$ and let $\Gamma''$ be the full subgraph spanned by $V$ and $V^{\perp}$ where $V=V_{\Gamma'}$ (the vertex set of $\Gamma'$). Then $\Gamma'$ is stable by Lemma \ref{3.17} and $\Gamma''$ is stable by Lemma \ref{3.16}. Note that $d(v,x)\ge n$ for any vertex $x\in V$, so $d(v,y)\ge n-1$ for any vertex $y\in V^{\perp}$. Thus $d(v,\Gamma'')\ge n-1$. However, $v_{n-1}\in\Gamma''$. So $d(v,\Gamma'')=n-1$. Now we can induct on $n$ and reduce to the $d(v,\Gamma_{1})=1$ case.
\end{proof}

It is interesting to see that if $\Gamma$ has large diameter, then there are a lot of non-trivial stable subgraphs.

We record the following lemma which is an easy consequence of Theorem \ref{2.9}.

\begin{lem}
\label{3.31}
Suppose $\Gamma=\Gamma_{1}\circ\Gamma_{2}$ where $\Gamma_{1}$ is the maximal clique join factor of $\Gamma$. If $\Gamma'_{2}$ is stable in $\Gamma_{2}$, then $\Gamma_{1}\circ\Gamma'_{2}$ is stable in $\Gamma$.
\end{lem}

Now we are ready to prove Lemma \ref{3.32}.

\begin{proof}[Proof of Lemma \ref{3.32}]
By Lemma \ref{3.14}, $\Gamma_{w}$ is the minimal stable subgraph that contains $w$. If there exists vertex $w'\in W$ such that $w'\notin \Gamma_{w}$, then sending $w\to ww'$ and fixing all other vertices would induce a group automorphism, which gives rise to a quasi-isometry from $X(\Gamma)$ to $X(\Gamma)$. The existence of such quasi-isometry would contradict the stability of $\Gamma_{w}$, thus $W\subset\Gamma_{w}$.

Let $W'$ be the vertex set of $\Gamma_{w}$. It remains to prove $W'\subset W$. Suppose $W\subsetneq W'$ and let $u\in W'\setminus W$. Then $\emptyset\neq w^{\perp}\setminus St(u)$. The minimality of $\Gamma_{w}$ implies we can not use Lemma \ref{3.17} to get rid of $u$ while keep $w$, then $w^{\perp}\setminus St(u)\subset W'\setminus\{u,w\}$. In summary,
\begin{equation}
\label{3.33}
\emptyset\neq w^{\perp}\setminus St(u)\subset W'\setminus\{u,w\}.
\end{equation}
In particular, $w$ is not isolated in $\Gamma_{w}$ and
\begin{equation}
\label{3.34}
\Gamma_{w}\nsubseteq St(u).
\end{equation}

Now we apply Lemma \ref{3.30} to $\Gamma_{w}$ and $w$, and recall that if a subgraph is stable in $\Gamma_{w}$, then it is stable in $\Gamma$. If case (1) in Lemma \ref{3.30} is true, then we will get a contradiction since $w$ is not isolated in $\Gamma_{w}$. If case (2) is true, then $\Gamma_{w}$ sits inside some clique, which is contradictory to (\ref{3.34}).

If case (3) is true, let $\Gamma_{1}\subset\Gamma_{w}$ be the corresponding stable discrete subgraph. Let $V_{1}=V_{\Gamma_{1}}$ and let $V'_{1}=\{u\in\Gamma_{w}\mid d(u,v)=1$ for any $v\in V_{1}\}$. Suppose $\Gamma'_{w}$ is the full subgraph spanned by $V_{1}$ and $V'_{1}$. Then $\Gamma'_{w}$ is stable by Lemma \ref{3.16}, hence $\Gamma'_{w}=\Gamma_{w}$. Let $\Gamma_{w}=\bar{\Gamma}_{1}\circ\bar{\Gamma}_{2}\circ\cdots\circ\bar{\Gamma}_{k}$ be the join decomposition induced by the De Rahm decomposition of $X(\Gamma_{w})$. Then $k\ge 2$ and $u$ does not sit inside the clique factor by (\ref{3.34}).

If there is no clique factor, then each join factor is stable by Theorem \ref{2.9} and $w$ is inside one of the join factors, which contradict the minimality of $\Gamma_{w}$. If the clique factor exists and $w$ sits inside the clique factor, then by Theorem \ref{2.9}, the clique factor is stable and we have the same contradiction as before. If the clique factor exists and $w$ sits outside the clique factor, this reduces to the next case.

If case (4) is true, let $\Gamma_{2}\subset\Gamma_{w}$ be the corresponding stable clique subgraph. We can also assume without loss of generality that $w$ is not contained in a stable clique. Let $V_{2}=V_{\Gamma_{2}}$ and $V'_{2}=\{u\in\Gamma_{w}\mid d(u,v)=1$ for any $v\in V_{2}\}$. Suppose $\Gamma''_{w}$ is the full subgraph spanned by $V_{2}$ and $V'_{2}$. Then $\Gamma''_{w}=\Gamma_{w}$ as before. Let $\Gamma_{w}=\Gamma'_{1}\circ\Gamma'_{2}$ where $\Gamma'_{1}$ corresponds to the Euclidean De Rahm factor of $X(\Gamma_{w})$. Note that $\Gamma'_{2}$ is non-trivial and $w,u\in\Gamma'_{2}$ as in the discussion of case (3). Equation (\ref{3.33}) implies that $w^{\perp}\nsubseteq St(u)$ is still true if we take the orthogonal complement of $w$ and the closed star of $u$ in $\Gamma'_{2}$, in particular, $w$ is not isolated in $\Gamma'_{2}$. Moreover, $dim(X(\Gamma'_{2}))<dim(X(\Gamma_{w}))\le dim(X(\Gamma))$. 

If $dim(X(\Gamma))=2$, then $\Gamma'_{2}$ has to be discrete, which is contradictory to that $w$ is not isolated in $\Gamma'_{2}$. If $dim(X(\Gamma))=n>2$, by induction we can assume the lemma is true for all lower dimensional graphs. Then there exists $\bar{\Gamma}_{w}$ stable in $\Gamma'_{2}$ such that $w\in\bar{\Gamma}_{w}$ and $u\notin\bar{\Gamma}_{w}$. By Lemma \ref{3.31}, $\bar{\Gamma}_{w}\circ\Gamma'_{1}$ is stable in $\Gamma_{w}$, hence in $\Gamma$, which contradicts the minimality of $\Gamma_{w}$.
\end{proof}

\begin{remark}
\label{3.36}
It is nature to ask whether Theorem \ref{3.35} is still true if we do not require $\Gamma_1$ to be a clique. There turns out to be counterexamples. Let $\Gamma$ be the graph as below and $\Gamma_{1}\subset\Gamma$ be the disjoint union of $v$ and $w$. It is easy to check there do not exist $v_{1}\in\Gamma_{1}$ and $v_{2}\in\Gamma\setminus\Gamma_{1}$ such that $v_{1}^{\perp}\subset St(v_{2})$. Note that $St(u)$ separates $\Gamma$, then we get a partial conjugation that sends $v\to v$ and $w\to u^{-1}wu$, which implies $\Gamma_{1}$ is not stable.

\
\begin{center}
\begin{tikzpicture}
\tikzstyle{every node}=[circle, draw, inner sep=0pt, minimum width=8pt]
\node (n1) at (2,0) {$u$};
\node (n2) at (1,1) {};
\node (n3) at (0,0) {$v$};
\node (n4) at (0.5,-1) {};
\node (n5) at (1.5,-1) {};
\node (n6) at (3,1) {};
\node (n7) at (4,0) {$w$};
\node (n8) at (3.5,-1) {};
\node (n9) at (2.5,-1) {};

\foreach \from/\to in {n1/n2,n2/n3,n3/n4,n4/n5,n5/n1,n1/n6,n6/n7,n7/n8,n8/n9,n9/n1}
\draw (\from) -- (\to);
\end{tikzpicture}
\end{center}
\

A more interesting example (but of the same nature) is the following. Let $\Gamma_{1}$ be the graph in the left side as below and $\Gamma_{2}$ be the graph in the right side. Then $G(\Gamma_{1})$ is quasi-isometric to $G(\Gamma_{2})$ by the discussion in Section 11 of \cite{MR2421136}. Let $q:X(\Gamma_{2})\to X(\Gamma_{1})$ be a quasi-isometry and let $K$ be a standard subcomplex in $X(\Gamma_{2})$ such that its defining graph $\Gamma_{K}$ is a pentagon in $\Gamma_{2}$. Suppose $q(K)$ is Hausdorff close to a standard subcomplex $K'$ in $X(\Gamma)$. Then $\Gamma_{K'}$ must be a connected proper subgraph of $\Gamma_{1}$, hence is a tree. But this is impossible by the results in \cite{behrstock2008quasi}.

\
\begin{center}
\begin{tikzpicture}
\tikzstyle{every node}=[circle, draw, inner sep=0pt, minimum width=5pt]
\node (n1) at (2,0) {};
\node (n2) at (1,1) {};
\node (n3) at (0,0) {};
\node (n4) at (0.5,-1) {};
\node (n5) at (1.5,-1) {};
\node (n6) at (5,0) {};
\node (n7) at (5,1) {};
\node (n8) at (4,1) {};
\node (n9) at (4,-1) {};
\node (n10) at (5,-1) {};
\node (n11) at (6,-1) {};
\node (n12) at (6,1) {};

\foreach \from/\to in {n1/n2,n2/n3,n3/n4,n4/n5,n5/n1,n6/n7,n7/n8,n8/n9,n9/n10,n10/n6,n10/n11,n11/n12,n12/n7}
\draw (\from) -- (\to);
\end{tikzpicture}
\end{center}
\

\end{remark}

\section{From quasi-isometries to isomorphisms}
\label{sec_qi implies iso}
\subsection{The extension complexes} 
\label{subsec_a boundary map}
\subsubsection{Extension complexes and standard flats}
Let $q:X(\Gamma)\to X(\Gamma')$ be a quasi-isometry. Usually $q$ does not induce a well-defined boundary map, see \cite{croke2000spaces}. However, Theorem \ref{3.28} implies we still have control on a subset of the Tits boundaries when $\out(G(\Gamma))$ and $\out(G(\Gamma'))$ are transvection free. In this subsection, we will reorganize this piece of information in terms of extension complexes.

Recall that we identify the vertex set of $\Gamma$ with a standard generating set $S$ of $G(\Gamma)$. And we also label the standard circles in the Salvetti complex by elements in $S$. By choosing an orientation for each standard circle, we obtain a directed labeling of edges in $X(\Gamma)$.

Denote the extension complex of $\Gamma$ by $\mathcal{P}(\Gamma)$. We give an alternative definition of $\mathcal{P}(\Gamma)$ here, which is natural for our purposes. The vertices of $\mathcal{P}(\Gamma)$ are in 1-1 correspondence with the parallel classes of standard geodesics in $X(\Gamma)$ (two standard geodesics are in the same parallel class if they are parallel). Two distinct vertices $v_{1},v_{2}\in\mathcal{P}(\Gamma)$ are connected by an edge if and only if we can find standard geodesic $l_{i}$ in the parallel class associated with $v_{i}$ ($i=1,2$) such that $l_{1}$ and $l_{2}$ span a standard 2-flat. The next observation follows from Lemma \ref{3.1} and Lemma \ref{2.3}:
\begin{ob}
\label{4.1}
If $v_{1}\neq v_{2}$, then $v_{1}$ and $v_{2}$ are joined by an edge if and only if there exist $l'_{i}$ in the parallel class associated with $v_{i}$ ($i=1,2$) and $R>0$ such that $l'_{1}\subset N_{R}(P_{l'_{2}})$.
\end{ob}
$\mathcal{P}(\Gamma)$ is defined to be the flag complex of its 1-skeleton.

\begin{lem}
\label{4.2}
$\mathcal{P}(\Gamma)$ is isomorphic to the extension complex of $\Gamma$.
\end{lem}

\begin{proof}
It suffices to show the 1-skeleton of $\mathcal{P}(\Gamma)$ is isomorphic to the extension graph $\Gamma^e$. Pick vertex $v\in\mathcal{P}(\Gamma)$ and let $l$ be a standard geodesic in the parallel class associated with $v$. We identify $l$ with $\Bbb R$ in an orientation-preserving way (the orientation in $l$ is induced by the directed labeling). Recall that $G(\Gamma)\curvearrowright X(\Gamma)$ by deck transformations. Let $\alpha_{v}\in G(\Gamma)$ be the element such that $\alpha_{v}(l)=l$ and $\alpha_{v}(x)=x+1$ for any $x\in l$. It is easy to see $\alpha_{v}$ is conjugate to an element in $S$, thus $\alpha_{v}$ gives rise a vertex $\alpha_{v}\in\Gamma^{e}$ by Definition \ref{2.11}. Note that $\alpha_{v}$ does not depend the choice of $l$ in the parallel class, so we have a well-defined map from the vertex set of $\mathcal{P}(\Gamma)$ to the vertex set of $\Gamma^{e}$. Moreover, if $v_{1}$ and $v_{2}$ are adjacent, then $\alpha_{v_{1}}$ and $\alpha_{v_{2}}$ commute.

Now we define an inverse map. Pick $\alpha=gsg^{-1}\in\Gamma^{e}$ ($s\in S$). Then all standard geodesics which are stabilized by $\alpha$ are in the same parallel class. Let $v_{\alpha}$ be the vertex in $\mathcal{P}(\Gamma)$ associated with this parallel class. We map the vertex $\alpha$ of $\Gamma^e$ to the vertex $v_{\alpha}$. Now we show this map extends to the 1-skeleton. For $i=1,2$, let $\alpha_{i}=g_{i}s_{i}g^{-1}_{i}\in\Gamma^{e}$. By the centralizer theorem of \cite{servatius1989surface}, $\alpha_{1}$ and $\alpha_{2}$ commute if and only if $[s_{1},s_{2}]=1$ and there exists $g\in G(\Gamma)$ such that $\alpha_{i}=gs_{i}g^{-1}$. Thus $v_{\alpha_{1}}$ and $v_{\alpha_{2}}$ are adjacent in $\mathcal{P}(\Gamma)$.
\end{proof}

Since every edge in the standard geodesics of the same parallel class has the same label, the labeling of the edges of $X(\Gamma)$ induces a labeling of the vertices of $\mathcal{P}(\Gamma)$. Moreover, since $G(\Gamma)\curvearrowright X(\Gamma)$ by label-preserving cubical isomorphisms, we obtain an induced action $G(\Gamma)\curvearrowright\mathcal{P}(\Gamma)$ by label-preserving simplicial isomorphisms. Moreover, the unique label-preserving map from the vertices of $\mathcal{P}(\Gamma)$ to the vertices of $F(\Gamma)$ extends to a simplicial map
\begin{equation}
\label{projection}
\pi:\mathcal{P}(\Gamma)\to F(\Gamma).
\end{equation}

Pick arbitrary vertex $p\in X(\Gamma)$, one can obtain a simplicial embedding $i_{p}$ from the flag complex $F(\Gamma)$ of $\Gamma$ to $\mathcal{P}(\Gamma)$ by considering the collection of standard geodesics passing through $p$. We will denote the image of $i_{p}$ by $(F(\Gamma))_{p}$. Note that for each vertex $p\in X(\Gamma)$, $\pi\circ i_p:F(\Gamma)\to F(\Gamma)$ is the identity map. 

Pick $(k-1)$-simplex in $\mathcal{P}(\Gamma)$ with vertex set $\{v_{i}\}_{i=1}^{k}$ and pick standard geodesic $l_{i}$ in the parallel class associated with $v_{i}$ for $1\le i\le k$. Since $P_{l_{i}}\cap P_{l_{j}}\neq\emptyset$ for $1\le i\neq j\le k$, by Lemma \ref{2.1}, $\cap_{i=1}^{k}P_{l_{i}}\neq\emptyset$. By Corollary \ref{3.2} and Lemma \ref{3.4}, there exist standard geodesics $\{l'_{i}\}_{i=1}^{k}$ satisfying
\begin{enumerate}
\item $l'_{i}$ is parallel to $l_{i}$ for each $i$.
\item The convex hull of $\{l'_{i}\}_{i=1}^{k}$ is a standard $k$-flat, denoted by $F_{k}$.
\item $\cap_{i=1}^{k}P_{l_{i}}=P_{F_{k}}$.
\end{enumerate}
Thus we have a 1-1 correspondence between the $(k-1)$-simplexes of $\mathcal{P}(\Gamma)$ and parallel classes of standard $k$-flats in $X(\Gamma)$. In particular, maximal simplexes in $\mathcal{P}(\Gamma)$, namely those simplexes which are not properly contained in some larger simplexes of $\mathcal{P}(\Gamma)$, are in 1-1 correspondence with maximal standard flats in $X(\Gamma)$. For standard flat $F\subset X(\Gamma)$, we denote the simplex in $\mathcal{P}(\Gamma)$ associated with the parallel class containing $F$ by $\Delta(F)$. 
\begin{ob}
\label{4.3}
Let $\Delta_{1}$, $\Delta_{2}$ be two simplexes in $\mathcal{P}(\Gamma)$ such that $\Delta=\Delta_{1}\cap\Delta_{2}\neq\emptyset$. For $i=1,2$, let $F_{i}\subset X(\Gamma)$ be a standard flat such that $\Delta(F_{i})=\Delta_{i}$. Set $(F'_{1},F'_{2})=\inc (F_{1},F_{2})$. Then $\Delta(F'_{1})=\Delta(F'_{2})=\Delta$.
\end{ob}

We define the \textit{reduced Tits boundary}, denoted $\bar{\partial}_{T}(X(\Gamma))$, to be the subset of $\partial_{T}(X(\Gamma))$ which is the union of Tits boundaries of standard flats in $X(\Gamma)$. For standard flat $F\subset X(\Gamma)$, we triangulate $\partial_{T}F$ into all-right spherical simplexes which are the Tits boundaries of orthant subcomplexes in $F$. Pick another standard flat $F'\subset X(\Gamma)$, then $\partial_{T}F\cap\partial_{T}F'$ is a subcomplex in both $\partial_{T}F$ and $\partial_{T}F'$ by Lemma \ref{3.1} and Remark \ref{2.5}. Thus we can endow $\bar{\partial}_{T}(X(\Gamma))$ with the structure of an all-right spherical complex. 

Now we look at the relation between $\bar{\partial}_{T}(X(\Gamma))$ and $\mathcal{P}(\Gamma)$. For each standard flat $F\subset X(\Gamma)$, we can associate $\partial_{T}F$ with $\Delta(F)\subset\mathcal{P}(\Gamma)$. This induces a surjective simplicial map $s:\bar{\partial}_{T}(X(\Gamma))\to\mathcal{P}(\Gamma)$ ($s$ can be defined by induction on dimension). Note that the inverse image of each simplex in $\mathcal{P}(\Gamma)$ under $s$ is a sphere in $\bar{\partial}_{T}(X(\Gamma))$. Then one can construct $\bar{\partial}_{T}(X(\Gamma))$ from $\mathcal{P}(\Gamma)$ as follows. We start with a collection of $\Bbb S^{0}$'s which are in 1-1 correspondence to vertices of $\mathcal{P}(\Gamma)$ and form a join of $n$ copies of $\Bbb S^{0}$'s if and only if the corresponding $n$ vertices in $\mathcal{P}(\Gamma)$ span a $(n-1)$-simplex. In other words, $\bar{\partial}_{T}(X(\Gamma))$ is obtained by applying the spherical complex construction in the sense of \cite[Definition 2.1.22]{brady2007geometry} to $\mathcal{P}(\Gamma)$.

Let $K_{1}\subset X(\Gamma)$ be a standard subcomplex. We define $\bar{\partial}_{T}(K_{1})$ to be the union of Tits boundaries of standard flats in $K_{1}$. Note that $\bar{\partial}_{T}(K_{1})=\bar{\partial}_{T}(X(\Gamma))\cap\partial_{T}K_{1}$. $\bar{\partial}_{T}(K_{1})$ descends to a subcomplex in $\mathcal{P}(\Gamma)$, which will be denoted by $\Delta(K_{1})$. 

\begin{lem}
Let $K_1$ and $K_2$ be two standard subcomplexes of $X(\Gamma)$. Put $(K'_1,K'_2)=\inc(K_1,K_2)$. Then $\Delta(K'_{1})=\Delta(K'_{2})=\Delta(K_{1})\cap\Delta(K_{2})$.
\end{lem}

\begin{proof}
By Remark \ref{2.5}, we have $\partial_{T}K'_{1}=\partial_{T}K'_{2}=\partial_{T}K_{1}\cap\partial_{T}K_{2}$, hence $\bar{\partial}_{T}K'_{1}=\bar{\partial}_{T}K'_{2}=\bar{\partial}_{T}K_{1}\cap\bar{\partial}_{T}K_{2}$ and $\Delta(K'_{1})=\Delta(K'_{2})=\Delta(K_{1})\cap\Delta(K_{2})$.
\end{proof}

Now we study how the extension complexes behave under quasi-isometries. 

\begin{lem}
	\label{4.7}
	Pick $\Gamma_{1}$ and $\Gamma_{2}$ such that $\out(G(\Gamma_{i}))$ is transvection free for $i=1,2$. Then any quasi-isometry $q: X(\Gamma_{1})\to X(\Gamma_{2})$ induces a simplicial isomorphism $q_{\ast}:\mathcal{P}(\Gamma_{1})\to\mathcal{P}(\Gamma_{2})$. If only $\out(G(\Gamma_{1}))$ is assumed to be transvection free, we still have a simplicial embedding $q_{\ast}:\mathcal{P}(\Gamma_{1})\to\mathcal{P}(\Gamma_{2})$.
\end{lem}

\begin{proof}
	We only proof the case when both $\out(G(\Gamma_{1}))$ and $\out(G(\Gamma_{2}))$ are transvection free. The other case is similar. By Theorem \ref{3.28}, every vertex in $\Gamma_{1}$ is stable, thus $q$ sends any parallel class of standard geodesics in $X(\Gamma_{1})$ to another parallel class of standard geodesics in $X(\Gamma_{2})$ up to finite Hausdorff distance. This induces a well-defined map $q_{\ast}$ from the 0-skeleton of $\mathcal{P}(\Gamma_{1})$ to the 0-skeleton of $\mathcal{P}(\Gamma_{2})$. The map $q_{\ast}$ is a bijection by considering the quasi-isometry inverse. Moreover, Observation \ref{4.1} implies two vertices in $\mathcal{P}(\Gamma_{1})$ are adjacent if and only if their images under $q_{\ast}$ are adjacent. So we can extend $q_{\ast}$ to be a graph isomorphism between the 1-skeleton of $\mathcal{P}(\Gamma_{1})$ and the 1-skeleton of $\mathcal{P}(\Gamma_{2})$. Since both $\mathcal{P}(\Gamma_{1})$ and $\mathcal{P}(\Gamma_{2})$ are flag complexes, we can extend the isomorphism to the whole complex.
\end{proof}

\subsubsection{Extension complexes and their relatives}
Now we discuss the relation between $\mathcal{P}(\Gamma)$ with several other objects in the literature. The material in this subsection will not be used later.

We can endow $F(\Gamma)$ with the structure of complex of groups, which gives us an alternative definition of $\mathcal{P}(\Gamma)$. More specifically, $\mathcal{P}(\Gamma)=F(\Gamma)\times G(\Gamma)/\sim$, here $St(v)\times g_{1}$ and $St(v)\times g_{2}$ ($v\in F(\Gamma)$ is a vertex) are identified if and only if there exists an integer $m$ such that $g^{-1}_{1}g_{2}=v^{m}$ (we also view $v$ as one of the generators of $G(\Gamma)$). Hence for $k$-simplex $\Delta^{k}\subset F(\Gamma)$ with vertex set $\{v_{i}\}_{i=1}^{k}$, $St(\Delta^{k})\times g_{1}$ and $St(\Delta^{k})\times g_{2}$ are identified if and only if $g^{-1}_{1}g_{2}$ belongs to the $\Bbb Z^{k}$ subgroup of $G(\Gamma)$ generated by $\{v_{i}\}_{i=1}^{k}$. One can compare this with a similar construction for Coxeter group in \cite{davis1983groups}.

There is another important object associated with a right-angled Artin group, called the modified Deligne complex in \cite{MR1368655} and the \textit{flat space} in \cite{MR2421136}.

\begin{definition}
\label{4.4}
Let $\Bbb P(\Gamma)$ be poset of left cosets of standard abelian subgroups of $G(\Gamma)$ (include the trivial subgroup) such that the partial order is induced by inclusion of sets. Then the \textit{modified Deligne complex} is defined to be the geometric realization of the derived poset of $\Bbb P(\Gamma)$.
\end{definition}

Recall that elements in the \textit{derived poset} of a poset $\Bbb P$ are totally ordered finite chains in $\Bbb P$. It can be viewed as an abstract simplex. 

The extension complex $\mathcal{P}(\Gamma)$ can be viewed as a coarse version of the modified Deligne complex. Let $A,B$ be two subsets of a metric space. We say $A$ and $B$ are \textit{coarsely equivalent} if $A\overset{\infty}{=}B$, and $A$ are \textit{coarsely contained} in $B$ if $A\subset_{\infty}B$. Let $\Bbb P'(\Gamma)$ be the poset whose elements are coarsely equivalent classes of left cosets of non-trivial standard abelian subgroups of $G(\Gamma)$, and the partial order is induced by coarse inclusion of sets.

\begin{ob}
\label{4.5}
The poset $\Bbb P'(\Gamma)$ is an abstract simplicial complex and it is isomorphic to $\mathcal{P}(\Gamma)$.
\end{ob}

Roughly speaking, $\Bbb P(\Gamma)$ captures the combinatorial pattern of how standard flats in $X(\Gamma)$ intersect with each other, and $\mathcal{P}(\Gamma)$ is about how they coarsely intersect with each other, thus $\Bbb P(\Gamma)$ contains more information than $\mathcal{P}(\Gamma)$. However, in certain cases, it is possible to recover information about $\Bbb P(\Gamma)$ from $\mathcal{P}(\Gamma)$, and this enable us to prove quasi-isometric classification/rigidity results for RAAG's.

We can define the poset $\Bbb P'(\Gamma)$ for arbitrary Artin group by considering the collection of coarse equivalent classes of spherical subgroups in an Artin group under coarse inclusion. Then the geometric realization of the derived poset of $\Bbb P'(\Gamma)$ would be a natural candidate for the extension complex of an Artin group. It is interesting to ask how much of the results in \cite{kim2014geometry} can be generalized to this context.

There is also a link between $\mathcal{P}(\Gamma)$ and the structure of hyperplanes in $X(\Gamma)$. Recall that for every $CAT(0)$ cube complex $X$, the \textit{crossing graph} of $X$, denoted by $C(X)$, is a graph whose vertices are in 1-1 correspondence to the hyperplanes in $X$, and two vertices are adjacent if and only if the corresponding hyperplanes intersect. The \textit{contact graph}, introduced in \cite{hagen2014weak} and denoted by $\mathcal{C}(X)$, has the same vertex set as $C(X)$, and two vertices are adjacent if and only if the carriers of the corresponding hyperplanes intersect.

There is a natural surjective simplicial map $p:C(X(\Gamma))\to\Gamma^{e}$ defined as follows. Pick a vertex $v\in C(X(\Gamma))$ and let $h$ be the corresponding hyperplane. Since all standard geodesics which intersect $h$ at one point are in the same parallel class, we define $p(v)$ to be the vertex in $\Gamma^{e}$ associated with this parallel class (see Lemma \ref{4.2}). It is clear that if $v_{1},v_{2}\in C(X(\Gamma))$ are adjacent vertices, then $p(v_{1})$ and $p(v_{2})$ are adjacent, so $p$ extends to a simplicial map. Pick vertex $w\in\Gamma^{e}$, then $p^{-1}(e)$ is the collection of hyperplanes dual to a standard geodesic.

\begin{thm}[\cite{kim2013embedability,hagen2014weak,quasitree}]
\label{4.6}
If $\Gamma$ is connected, then $C(X(\Gamma))$, $\mathcal{C}(X(\Gamma))$ and $\mathcal{P}(\Gamma)$ are quasi-isometric to each other, moreover, they are quasi-isometric to a tree.
\end{thm}

From this viewpoint, $\mathcal{P}(\Gamma)$ captures both the geometric information of $X(\Gamma)$ (the standard flats) and the combinatorial information (the hyperplanes).

\subsection{Reconstruction of quasi-isometries}
\label{subsec_reconstruct}
We show the boundary map $q_{\ast}:\mathcal{P}(\Gamma)\to\mathcal{P}(\Gamma')$ in Lemma \ref{4.7} induces a well defined map from $G(\Gamma)$ to $G(\Gamma')$.
\begin{lem}
\label{4.8}
Let $F_{1}$ and $F_{2}$ be two maximal standard flats in $X(\Gamma)$ and let $\Delta_{1}$, $\Delta_{2}$ be the corresponding maximal simplexes in $\mathcal{P}(\Gamma)$. If $F_{1}$ and $F_{2}$ are separated by a hyperplane $h$, then there exist vertices $v_{i}\in\Delta_{i}$ for $i=1,2$ and $v\in \mathcal{P}(\Gamma)$ such that $v_{1}$ and $v_{2}$ are in different connected component of $\mathcal{P}(\Gamma)\setminus St(v)$.
\end{lem}  

\begin{proof}
Let $e$ be an edge dual to $h$ and let $l$ be the standard geodesic that contains $e$. Set $v=\Delta(l)\in\mathcal{P}(\Gamma)$. By Lemma \ref{3.4}, the parallel set $P_{l}$ of $l$ is isometric to $h\times \Bbb E^{1}$. Thus every standard geodesic parallel to $l$ must have non-trivial intersection with $h$. Since $F_{1}\cap h=\emptyset$, $F_{1}$ can not contain any standard geodesic parallel to $l$, which means $v\notin \Delta_{1}$. Moreover, $\Delta_{1}\nsubseteq St(v)$ since $\Delta_{1}$ is a maximal simplex. Similarly, $\Delta_{2}\nsubseteq St(v)$, thus we can find vertices $v_{i}\in\Delta_{i}\setminus St(v)$ for $i=1,2$. We claim $v_{1}$, $v_{2}$ and $v$ are the vertices we are looking for.

If there is a path $\omega\subset \mathcal{P}(\Gamma)\setminus St(v)$ connecting $v_{1}$ and $v_{2}$, we can assume $\omega$ is consist of a sequence of edges $\{e_{i}\}_{i=1}^{k}$ with $v_{1}\in e_{1}$ and $v_{2}\in e_{k}$. For each $e_{i}$, pick a maximal simplex $\Delta'_{i}$ that contains $e_{i}$ and let $F'_{i}$ be the maximal standard flat such that $\Delta(F'_{i})=\Delta'_{i}$. Then $v\notin\Delta'_{i}$ for $1\le i\le k$, hence $F'_{i}\cap h=\emptyset$. 

Set $\Delta'_{0}=\Delta_{1}$, $\Delta'_{k+1}=\Delta_{2}$, $F'_{0}=F_{1}$ and $F'_{k+1}=F_{2}$. Since $\Delta'_{i}\cap\Delta'_{i+1}$ contains a vertex in $\omega$, we have 
\begin{equation}
\label{4.9}
(\Delta'_{i}\cap\Delta'_{i+1})\setminus St(v)\neq\emptyset
\end{equation}
for $0\le i\le k$. Since $F'_{0}$ and $F'_{k+1}$ are in different sides of $h$, there exists $i_{0}$ such that $h$ separates $F'_{i_{0}}$ and $F'_{i_{0}+1}$. Let $(F''_{i_{0}},F''_{i_{0}+1})=\inc (F'_{i_{0}},F'_{i_{0}+1})$. By Observation \ref{4.3}, $\Delta(F''_{i_{0}})=\Delta(F''_{i_{0}+1})=\Delta'_{i}\cap\Delta'_{i+1}$. However, by Lemma \ref{2.7}, there exists a convex subset of $h$ parallel to $F''_{i_{0}}$, thus $F''_{i_{0}}\subset_{\infty} h\subset P_{l}$. It follows from Observation \ref{4.1} that $\Delta'_{i}\cap\Delta'_{i+1}\subset St(v)$, which contradicts (\ref{4.9}).
\end{proof}

Denote the Cayley graph of $G(\Gamma)$ with respect to the standard generating set $S$ by $C(\Gamma)$. We pick an identification between $C(\Gamma)$ and the $1$-skeleton of $X(\Gamma)$. Thus $G(\Gamma)$ is identified with the vertex set of $X(\Gamma)$. 

\begin{lem}
\label{4.10}
Let $\Gamma_{1}$ be a finite simplicial such that
\begin{enumerate}
\item There is no separating closed star in $F(\Gamma_{1})$.
\item $F(\Gamma_{1})$ is not contained in a union of two closed stars.
\end{enumerate} 
Then any simplicial isomorphism $s:\mathcal{P}(\Gamma_{1})\to\mathcal{P}(\Gamma_{2})$ induces a unique map $s':G(\Gamma_{1})\to G(\Gamma_{2})$ such that for any maximal standard flat $F_{1}\subset X(\Gamma_{1})$, vertices in $F_{1}$ are mapped by $s'$ to vertices lying in a maximal standard flat $F_{2}\subset X(\Gamma_{2})$ with $\Delta(F_{2})=s'(\Delta(F_{1}))$.
\end{lem}

\begin{proof}
Pick vertex $p\in G(\Gamma_{1})$. Let $\{F_{i}\}_{i=1}^{k}$ be the collection of maximal standard flats containing $p$. For $1\le i\le k$, define $\Delta_{i}=\Delta(F_{i})$ and $\Delta'_{i}=s(\Delta_{i})$. Let $F'_{i}\subset X(\Gamma_{2})$ be the maximal standard flat such that $\Delta(F'_{i})=\Delta'_{i}$. Let $K_{p}=(F(\Gamma_1))_{p}=\cup_{i=1}^{k}\Delta_{i}$ (recall that $K_{p}\cong F(\Gamma_{1})$). We claim 
\begin{equation}
\label{4.11}
\cap_{i=1}^{k}F'_{i}\neq\emptyset.
\end{equation}
The lemma will then follows from (\ref{4.11}). To see this, we deduce from condition (2) that $\cap_{i=1}^{k}\Delta_i=\emptyset$. Hence $\cap_{i=1}^{k} F_i=\{p\}$. It follows that $\cap_{i=1}^{k}\Delta'_i=\emptyset$. This together with (\ref{4.11}) imply that $\cap_{i=1}^{k}F'_{i}$ is exactly one point. We define $s'$ by sending $p$ to this point. One readily verifies that $s'$ has the required properties.

It remains to prove (\ref{4.11}).

Suppose (\ref{4.11}) is not true. Then by Lemma \ref{2.1}, there exist $i_{1}$ and $i_{2}$ such that $F'_{i_{1}}\cap F'_{i_{2}}=\emptyset$. Thus $F'_{i_{1}}$ and $F'_{i_{2}}$ are separated by a hyperplane. It follows from Lemma \ref{4.8} that there exist vertices $v'\in\mathcal{P}(\Gamma_{2})$, $v'_{1}\in\Delta'_{i_{1}}$ and $v'_{2}\in\Delta'_{i_{2}}$ such that $v'_{1}$ and $v'_{2}$ are in different connected components of $\mathcal{P}(\Gamma_{2})\setminus St(v')$. Let $v=s^{-1}(v')$, $v_{1}=s^{-1}(v'_{1})$ and $v_{2}=s^{-1}(v'_{2})$. Then $K_{p}\setminus (K_{p}\cap St(v))$ is disconnected (since $v_{1},v_{2}\in K_{p}$ and they are separated by $St(v)$). 

If $v\in K_{p}$, then $K_{p}$ would contain a separating closed star, which yields a contradiction, thus (\ref{4.11}) is true in this case. 

Suppose $v\notin K_{p}$. Pick a standard geodesic $l$ such that $\Delta(l)=v$ and let $\{h_{i}\}_{i=1}^{n}$ be the collection of hyperplanes in $X(\Gamma)$ such that each $h_{i}$ separates $p$ from the parallel set $P_{l}$ of $l$ (note that $p\notin P_{l}$). For $1\le i\le n$, pick an edge $e_{i}$ dual to $h_{i}$ and let $w_{i}$ be the unique vertex in $K_{p}$ that has the same label as $e_{i}$. Let $w_{0}\in K_{p}$ be the unique vertex which has the same label as $v$. We claim 
\begin{equation}
\label{4.12}
St(v)\cap K_{p}=\cap_{i=0}^{n}(St(w_{i})\cap K_{p}).
\end{equation}

For every $u \in K_p$, let $l_u$ be the unique standard geodesic such that $\Delta(l_{u})=u$ and $p\in l_{u}$. 

Pick $u\in St(v)\cap K_{p}$. Observation \ref{4.1} implies $\inc (l_{u},P_{l})=(l_{u},l'_{u})$, where $l'_{u}$ is some standard geodesic in $P_{l}$. Then for $1\le i\le n$, $h_{i}$ separates $l_{u}$ from $P_{l}$, otherwise $h_{i}\cap l_{u}\neq\emptyset$ and Lemma \ref{2.6} implies $h_{i}\cap P_{l}\neq\emptyset$, which is a contradiction. It follows from Corollary \ref{3.2} that $u$ and $w_{i}$ are adjacent for $0\le i\le n$, thus $u\in \cap_{i=0}^{n}(St(w_{i})\cap K_{p})$. Therefore $St(v) \cap K_p \subset \cap_{i=0}^n (St(w_i)\cap K_p)$. 

Pick $u\in \cap_{i=0}^{n}(St(w_{i})\cap K_{p})$. First we show $l_u\cap P_l=\emptyset$. Suppose there is a vertex $z$ in $l_{u}\cap P_{l}$. Since $v$ and $w_{0}$ have the same label and $u \in St(w_{0})$, it follows that the edge in $l_{u}$ which contains $z$ belongs to the parallel set $P_{l}$. Then $l_{u}\subset P_{l}$, contradicting the fact that $p\notin P_{l}$. Therefore $l_u\cap P_l =\emptyset$.

Now we pick an edge path $\omega$ of shortest combinatorial length that travels from $l_{u}$ to $P_{l}$. Let $\{f_{j}\}_{j=1}^{m}$ be the consecutive edges in $\omega$ such that $f_{1}\cap l_{u}\neq\emptyset$. For each $f_j$, let $\bar{h}_{j}$ be the hyperplane dual to $f_{j}$. Then $\bar{h}_{j}$ separates $l_{u}$ from $P_{l}$ (otherwise $\omega$ would not be the shortest edge path), hence separates $p$ from $P_{l}$. This and $u\in \cap_{i=0}^{n}(St(w_{i})\cap K_{p})$ imply that $d(\pi(u),V_{f_{j}})\le 1$ for each $j$, where $\pi$ is the map in (\ref{projection}) and $V_{f_j}$ is the label of the edge $f_j$. It follows that $\omega$ is contained in the parallel set $P_{l_u}$, and hence the intersection $P_{l_u}\cap P_l$ contains some vertex $z$. Again, since $u \in St(w_0)$, and $w_0$ has the same label as $v$, we find that the standard geodesic $l_u' \subset P_{l_u}$ that is parallel to $l_u$ and passes through $z$, is contained in $P_l$. Therefore $u\in St(v)\cap K_{p}$. Then (\ref{4.12}) follows. 

By condition (2) of Lemma \ref{4.10}, we have 
\begin{equation}
\label{proper subset}
(St(w_{0})\cap K_{p})\cup(\cap_{i=1}^{n}(St(w_{i})\cap K_{p}))\subsetneq K_{p}.
\end{equation}
Let $A=K_p\setminus (St(w_{0})\cap K_{p})$ and let $B=K_p\setminus (\cap_{i=1}^{n}(St(w_{i})\cap K_{p}))$. Then (\ref{proper subset}) implies $A\cap B\neq\emptyset$. Thus we have the following Mayer-Vietoris sequence for reduced homology.
\begin{equation*}
\cdots \to\tilde{H}_0(A\cap B)\to \tilde{H}_0(A)\oplus \tilde{H}_0(B)\to \tilde{H}_0(A\cup B)\to 0.
\end{equation*}
Recall that $K_p\setminus(K_p\cap St(v))$ is disconnected, we deduce that $\tilde{H}_0(A\cup B)$ is nontrivial from (\ref{4.12}). Thus $\tilde{H}_0(A)\oplus \tilde{H}_0(B)$ is nontrivial, which implies either $\cap_{i=1}^{n}(St(w_{i})\cap K_{p})$ or $St(w_{0})\cap K_{p}$ would separate $K_{p}$. Thus we can induct on $n$ to deduce that there exists $i_{0}$ such that $St(w_{i_{0}})\cap K_{p}$ separates $K_{p}$. This yields a contradictory to condition (1) of Lemma \ref{4.10}.
\end{proof}

There are counterexamples if we only assume (1) in Lemma \ref{4.10}. For example, let $\Gamma_1$ and $\Gamma_2$ be discrete graphs made of two points. Then $\mathcal{P}(\Gamma_1)$ and $\mathcal{P}(\Gamma_2)$ are discrete sets. Now it is not hard to construct a permutation of a discrete set to itself which does not satisfy the conclusion of Lemma \ref{4.10}. If we go back to the proof of Lemma \ref{4.10}, then the step using Mayer-Vietoris sequence will fail, since we need $A\cap B\neq\emptyset$ in order to use the reduced version of Mayer-Vietoris sequence.

\begin{cor}
\label{4.13}
Suppose $G(\Gamma_{1})$ and $G(\Gamma_{2})$ both satisfy the assumption of Lemma \ref{4.10}. Then they are isomorphic if and only if $\mathcal{P}(\Gamma_{1})$ and $\mathcal{P}(\Gamma_{2})$ are isomorphic as simplicial complexes.
\end{cor}

\begin{proof}
The only if direction follows from the fact that $G(\Gamma_1)$ and $G(\Gamma_2)$ are isomorphic if and only if $\Gamma_1$ and $\Gamma_2$ are isomorphic, see \cite{droms1987isomorphisms}. It remains to prove the if direction. Pick an isomorphism $s:\mathcal{P}(\Gamma_{1})\to\mathcal{P}(\Gamma_{2})$, and let $s':G(\Gamma_1)\to G(\Gamma_2)$ be the map in Lemma \ref{4.10}. Pick vertex $p\in G(\Gamma_1)$ and let $q=s(p)$. We define $(F(\Gamma_1))_p\subset \mathcal{P}(\Gamma_1)$ and $(F(\Gamma_2))_q\subset \mathcal{P}(\Gamma_2)$ as in the first paragraph of the proof of Lemma \ref{4.10}. Then (\ref{4.11}) implies $s((F(\Gamma_1))_p)\subset (F(\Gamma_2))_q$. This induces a graph embedding $\Gamma_1 \hookrightarrow \Gamma_2$. By repeating the previous discussion for $s^{-1}$, we obtain another graph embedding $\Gamma_2 \hookrightarrow \Gamma_1$. Since both $\Gamma_1$ and $\Gamma_2$ are finite simplicial graphs, they are isomorphic. Hence $G(\Gamma_{1})\cong G(\Gamma_2)$.
\end{proof}

\begin{lem}
\label{4.14}
Let $G(\Gamma)$ be a RAAG such that $\out(G(\Gamma))$ is finite and $G(\Gamma)\ncong \mathbb Z$. Then $F(\Gamma)$ satisfies the assumption of Lemma \ref{4.10}.
\end{lem}

\begin{proof}
It is clear that $F(\Gamma)$ should satisfy condition (1) of Lemma \ref{4.10} since no nontrivial partial conjugation is allowed. If $F(\Gamma)$ is contained in a closed star, then $\Gamma$ is a point. So if (2) is not true, then $F(\Gamma)=St(v)\cup St(w)$ for distinct vertices $v,w\in\Gamma$. Since the orthogonal complement $v^{\perp}$ satisfies $v^{\perp}\nsubseteq St(w)$, there exists $u\in v^{\perp}$ such that $d(u,w)\ge 2$. Pick any edge $e$ such that $u\in e$, then $e\nsubseteq St(w)$, so $e\subset St(v)$. This implies $u^{\perp}\subset St(v)$, hence $\out(G(\Gamma))$ is infinite, which yields a contradiction.
\end{proof}

By Lemma \ref{4.7}, Lemma \ref{4.14} and Corollary \ref{4.13}, we have following result, which in particular establishes Theorem \ref{1.1} of the introduction.
\begin{thm}
\label{4.15}
Let $\Gamma_{1}$ and $\Gamma_{2}$ be two finite simplicial graphs such that $\out(G(\Gamma_{i}))$ is finite for $i=1,2$. Then $G(\Gamma_{1})$ and $G(\Gamma_{2})$ are quasi-isometric if and only if they are isomorphic. Moreover, for any $(L,A)$-quasi-isometry $q:X(\Gamma_{1})\to X(\Gamma_{2})$, there exist a bijection $q':G(\Gamma_{1})\to G(\Gamma_{2})$ and a constant $D=D(L,A,\Gamma_{1})$ such that 
\begin{enumerate}
\item $d(q(v),q'(v))<D$ for any $v\in G(\Gamma_{1})$.
\item For any standard flat $F_{1}\subset X(\Gamma_{1})$, there exists a standard flat $F_{2}\subset X(\Gamma_{2})$ such that $q'$ induces a bijection between $F_{1}\cap G(\Gamma_{1})$ and $F_{2}\cap G(\Gamma_{2})$.
\end{enumerate}
If $G(\Gamma_1)\neq \Bbb Z$, then such $q'$ is unique.
\end{thm}

\begin{proof}
It suffices to look at the case $G(\Gamma_1)\neq \Bbb Z$. Then $G(\Gamma_2)\neq \Bbb Z$. In this case, every vertex $v$ in $\Gamma_1$ or $\Gamma_2$ is the intersection of maximal cliques that contain $v$ (otherwise there exist vertex $w$ such that $w\neq v$ and $v^{\perp}\subset St(w)$). It follows that every standard geodesic in $X(\Gamma_1)$ or $X(\Gamma_2)$ is the intersection of finitely many maximal standard flats, so is every standard flat. Let $q_{\ast}:\mathcal{P}(\Gamma_1)\to \mathcal{P}(\Gamma_2)$ be the map in Lemma \ref{4.7}. We apply Lemma \ref{4.10} to $q_{\ast}$ and $q^{-1}_{\ast}$ to obtain $q'$ with the required properties. Note that each vertex of $X(\Gamma)$ is the intersection of maximal standard flats that contain it, thus $q'$ is unique.
\end{proof}

\subsection{The automorphism groups of extension complexes}
\label{auto of extension complexes}
Suppose $\out(G(\Gamma))$ is finite, by Theorem \ref{4.15}, each element in the simplicial automorphism group $\aut(\mathcal{P}(\Gamma))$ of $\mathcal{P}(\Gamma)$ induces a bijection $G(\Gamma)\to G(\Gamma)$. However, this bijection does not extend to an isomorphism from $X(\Gamma)$ to itself in general. We start by looking at the following example which was first pointed out in \cite[Section 11]{MR2421136} in a slightly different form.

\begin{example}
\label{flip order}

Let $l\subset X(\Gamma)$ be a standard geodesic let $\pi_{l}: X(\Gamma)\to l$ be the CAT(0) projection. We identify the vertex set of $l$ with $\Bbb Z$. Let $X^{(0)}(\Gamma)$ be the vertex set of $X(\Gamma)$. Then the above projection induces a map $\pi_{l}: X^{(0)}(\Gamma)\to \Bbb Z$.

Recall that each edge of $X(\Gamma)$ is oriented and labeled, and $G(\Gamma)$ acts on $X(\Gamma)$ by transformations that preserve labels and orientations. There is a unique element $\alpha\in G(\Gamma)$ such that $\alpha$ translates $l$ one unit towards the positive direction.

We want to define a bijection $q: X^{(0)}(\Gamma) \to X^{(0)}(\Gamma)$ which basically flips $\pi_{l}^{-1}(0)$ and $\pi_{l}^{-1}(1)$. More precisely:
\begin{center}
$q(x)$=$\begin{cases}
x & \text{if $\pi_{l}(x)\neq 0,1$}\\
\alpha(x) & \text{if $x\in\pi_{l}^{-1}(0)$}\\
\alpha^{-1}(x) & \text{if $x\in\pi_{l}^{-1}(1)$}
\end{cases}$
\end{center}

One can check the following:
\begin{enumerate}
\item $q$ is a quasi-isometry.
\item $q$ does not respect the word metric.
\item $q$ maps vertices in a standard flat to vertices in another standard flat. Thus $q$ induces an element in $\aut(\mathcal{P}(\Gamma))$.
\end{enumerate}

\end{example}

The above example implies that in general, elements in $\aut(\mathcal{P}(\Gamma))$ do not respect the order along the standard geodesics of $X(\Gamma)$. There is another metric on $G(\Gamma)$ which \textquotedblleft forgets about\textquotedblright\ the ordering. Following \cite{kim2014geometry}, we define the \textit{syllable length} of a word $\omega$ to be minimal $l$ such that $\omega$ can be written as a product of $l$ elements of form $v_{i}^{k_{i}}$, where $v_{i}$ is a standard generator and $k_{i}$ is an integer. 

An alternative definition is the following. Let $\{h_{i}\}_{i=1}^{k}$ be the collection of hyperplanes separating $\omega\in G(\Gamma)$ and the identity element (recall that we have identified $G(\Gamma)$ with the $0$-skeleton of $X(\Gamma)$). For each $i$, pick a standard geodesic $l_{i}$ dual to $h_{i}$. Then the syllable length of $\omega$ is the number of elements in $\{\Delta(l_{i})\}_{i=1}^{k}$. The syllable length induces a left invariant metric on $G(\Gamma)$, which will be denoted by $d_{r}$. Note that the map in Example \ref{flip order} is an isometry with respect to $d_{r}$.

Denote the word metric on $G(\Gamma)$ with respect to the standard generators by $d_{w}$. 

\begin{cor}
\label{4.16}
Let $\Gamma$ be a graph such that $\out(G(\Gamma))$ is finite and denote the simplicial automorphism group of $\mathcal{P}(\Gamma)$ by $\aut(\mathcal{P}(\Gamma))$. Then 
\begin{equation*}
\aut(\mathcal{P}(\Gamma))\cong \isom(G(\Gamma),d_{r})
\end{equation*}
\end{cor}

\begin{proof}
We have a group homomorphism $h_{1}:\aut(\mathcal{P}(\Gamma))\to Perm(G(\Gamma))$ by Lemma \ref{4.10}, here $Perm(G(\Gamma))$ is the permutation group of elements in $G(\Gamma)$. Take $\phi\in \aut(\mathcal{P}(\Gamma))$, by Lemma \ref{4.14}, $\varphi=h_{1}(\phi)$ and $\varphi^{-1}=h_{1}(\phi^{-1})$ satisfy the conclusion of Lemma \ref{4.10}. Since every standard geodesic is the intersection of finitely many maximal standard flats, points in a standard geodesic are mapped to points in a standard geodesic by $\phi$, which implies $d_{r}(\varphi(v_{1}),\varphi(v_{2}))\le d_{r}(v_{1},v_{2})$ if $d_{r}(v_{1},v_{2})\le 1$. By triangle inequality, we have $d_{r}(\varphi(v_{1}),\varphi(v_{2}))\le d_{r}(v_{1},v_{2})$ for any $v_{1},v_{2}\in G(\Gamma)$. Similarly, $d_{r}(\varphi^{-1}(v_{1}),\varphi^{-1}(v_{2}))\le d_{r}(v_{1},v_{2})$. Thus $\varphi\in \isom(G(\Gamma),d_{r})$ and we have a homomorphism $h_{1}:\aut(\mathcal{P}(\Gamma))\to \isom(G(\Gamma),d_{r})$.

Now pick $\varphi\in \isom(G(\Gamma),d_{r})$. Let $v_{1},v_{2},v_{3}\in G(\Gamma)$ such that $d_{r}(v_{1},v_{i})=1$ for $i=2,3$. We claim 
\begin{equation}
\label{4.17}
\angle_{v_{1}}(v_{2},v_{3})=\pi/2\Leftrightarrow\angle_{\varphi(v_{1})}(\varphi(v_{2}),\varphi(v_{3}))=\pi/2.
\end{equation}
If $\angle_{v_{1}}(v_{2},v_{3})=\pi/2$, then we can find $v_{4}\in G(\Gamma)$ such that $\{v_{i}\}_{i=1}^{4}$ are the vertices of a flat rectangle in $X(\Gamma)$. Note that $d_{r}(v_{1},v_{4})=d_{r}(v_{2},v_{3})=2$ and $d_{r}(v_{4},v_{2})=d_{r}(v_{4},v_{3})=1$, so  $d_{r}(\varphi(v_{1}),\varphi(v_{4}))=d_{r}(\varphi(v_{2}),\varphi(v_{3}))=2$ and $d_{r}(\varphi(v_{4}),\varphi(v_{2}))=d_{r}(\varphi(v_{4}),\varphi(v_{3}))=1$. Now we consider the 4-gon formed by $\overline{\varphi(v_{1})\varphi(v_{2})}$, $\overline{\varphi(v_{2})\varphi(v_{4})}$, $\overline{\varphi(v_{4})\varphi(v_{3})}$ and $\overline{\varphi(v_{3})\varphi(v_{1})}$. Then the angles at the four vertices of this 4-gon are bigger or equal to $\pi/2$. It follows from $CAT(0)$ geometry that the angles are exactly $\pi/2$ and the 4-gon actually bounds a flat rectangle. Thus one direction of (\ref{4.17}) is proved, the other direction is similar.

We need another observation as follows. If three points $v_{1},v_{2},v_{3}\in G(\Gamma)$ satisfies $d_{r}(v_{i},v_{j})=1$ for $1\le i\neq j\le 3$, then the angle at each vertex of the triangle $\Delta(v_{1},v_{2},v_{3})$ could only be $0$ or $\pi$, thus $\{v_{i}\}_{i=1}^{3}$ are inside a standard geodesic. It follows from this observation that points in a standard geodesic are mapped by $\varphi$ to points in a standard geodesic. 

We define $\phi:\mathcal{P}(\Gamma)\to\mathcal{P}(\Gamma)$ as follows. For vertex $w\in \mathcal{P}(\Gamma)$, let $l$ be a standard geodesic such that $\Delta(l)=w$. Suppose $l'\subset X(\Gamma)$ is the standard geodesic such that $\phi(v(l))\subset l'$ ($v(l)$ denotes the vertex set of $l$). Suppose $w'=\Delta(l')$. We define $w'=\phi(w)$. (\ref{4.17}) implies $w'$ does not depend on the choice of $l$, and $\phi(w_{1})$ and $\phi(w_{2})$ are adjacent if vertices $w_{1},w_{2}\in\mathcal{P}(\Gamma)$ are adjacent. Thus $\phi$ is a well-defined simplicial map. Note that $\varphi^{-1}$ also induces a simplicial map from $\mathcal{P}(\Gamma)$ to itself in a similar way, so $\phi\in \aut(\mathcal{P}(\Gamma))$. We define $\phi=h_{2}(\varphi)$. One readily verify that $h_{2}:\isom(G(\Gamma),d_{r})\to \aut(\mathcal{P}(\Gamma))$ is a group homomorphism and $h_{2}\circ h_{1}=h_{1}\circ h_{2}=\textmd{Id}$. Thus the corollary follows.
\end{proof}

\begin{remark}
\label{4.18}
If we drop the assumption in the above corollary about $\Gamma$, then there is still a monomorphism $h:\isom(G(\Gamma),d_{r})\to \aut(\mathcal{P}(\Gamma))$, moreover, any $\varphi\in \isom(G(\Gamma),d_{r})$ maps vertices in a standard flat to vertices in a standard flat of the same dimension. The homomorphism $h$ is surjective if $\out(G(\Gamma))$ is finite.
\end{remark}

\begin{remark}
\label{4.19}
For any finite simplicial graphs $\Gamma_{1}$ and $\Gamma_{2}$, $G(\Gamma_{1})\cong G(\Gamma_{2})$ if and only if $(G(\Gamma_{1}),d_{r})$ and $(G(\Gamma_{2}),d_{r})$ are isometric as metric spaces. The if only direction follows from \cite{droms1987isomorphisms,laurence1995generating}. For the other direction, let $\varphi:(G(\Gamma_{1}),d_{r})\to(G(\Gamma_{2}),d_{r})$ be an isometry. Pick $v\in G(\Gamma_{1})$ and let $\{l_{i}\}_{i=1}^{k}$ be the collection of standard geodesics passing through $v$. Pick $v_{i}\in G(\Gamma_{1})$ such that $v_{i}\in l_{i}\setminus\{v\}$. Then $d_{r}(v,v_{i})=1$ for $1\le i\le k$ and $d_{r}(v_{i},v_{j})=2$ for $1\le i\neq j\le k$. So $d_{r}(\varphi(v),\varphi(v_{i}))=1$ for $1\le i\le k$ and $d_{r}(\varphi(v_{i}),\varphi(v_{j}))=2$ for $1\le i\neq j\le k$, and $\angle_{v}(v_{i},v_{j})=\pi/2$ if and only if $\angle_{\varphi(v)}(\varphi(v_{i}),\varphi(v_{j}))=\pi/2$ by (\ref{4.17}). This induces a graph embedding $\Gamma_1\to \Gamma_2$. By considering $\varphi^{-1}$, we obtain another graph embedding $\Gamma_2\to \Gamma_1$. Hence $\Gamma_{1}$ and $\Gamma_{2}$ are isomorphic.
\end{remark}

\begin{cor}
\label{4.20}
If $\out(G(\Gamma))$ is finite, then we have the following commutative diagram, where $i_{1}$, $i_{2}$ and $i_{3}$ are injective homomorphisms:
\begin{center}
\begin{tikzpicture}
\node (k-1) at (0,3) {$\isom(G(\Gamma),d_{w})$};
\node (k0) [right=of k-1] {$\qi(G(\Gamma))$};
\node (k1) [right=of k0] {$\isom(G(\Gamma),d_{r})$};
\draw[->]
(k-1) edge node[auto] {$i_{1}$} (k0)
(k0) edge node[auto] {$i_{2}$} (k1);
\draw[->]
(k-1) edge [bend right] node[below] {$i_{3}$} (k1);
\end{tikzpicture}
\end{center}
\end{cor}

$\qi(G(\Gamma))$ is the quasi-isometry group of $G(\Gamma)$.
\begin{proof}
The homomorphism $i_{1}$ and $i_{3}$ are obvious and $i_{2}$ is given by Lemma \ref{4.7} and Corollary \ref{4.16}. It is clear that $i_{2}$ is a group homomorphism and $i_{3}=i_{2}\circ i_{1}$. Note that $i_{3}$ is injective, so $i_{1}$ is injective. Pick $\alpha\in \qi(G(\Gamma))$, by Corollary \ref{4.16}, we know $i_{2}(\alpha)=\textmd{Id}$ implies the image of every standard flat under $\alpha$ is uniformly Hausdorff close to itself, thus $\alpha$ is of bounded distance from the identity map.
\end{proof}

\section{Quasi-isometries and special subgroups}
\label{sec_qi and special subgroups}
Let $G(\Gamma)$ be a RAAG with finite outer automorphism group. In this section we characterize all other RAAG's quasi-isometric to $G(\Gamma)$.
 
\subsection{Preservation of extension complex}
\label{subsec_preservation of extesion complex}
\begin{lem}
\label{5.1}
Let $\Gamma$ be a finite simplicial graph. Pick a vertex $w\in\Gamma$ and let $\Gamma_{w}$ be the minimal stable subgraph containing $w$. Denote $\Gamma_{1}=lk(w)$ and $\Gamma_{2}=lk(\Gamma_{1})$ $($see Section \ref{subsec_notation} for definition of links$)$. Then either of the following is true:
\begin{enumerate}
\item $\Gamma_{w}$ is a clique. In this case $St(w)$ is a stable subgraph.
\item Both $\Gamma_{1}$ and $\Gamma_{1}\circ\Gamma_{2}$ are stable subgraphs of $\Gamma$. Moreover, $\Gamma_{2}$ is disconnected.
\end{enumerate}
\end{lem}

Recall that we use $(\Gamma')^{\perp}$ to denote the orthogonal complement of the subgraph $\Gamma'\subset\Gamma $ (see Section \ref{subsec_notation}) and we assume $(\emptyset)^{\perp}=\Gamma$.
\begin{proof}
If $\Gamma_{w}\subset St(w)$, then $\Gamma_{w}$ is a clique by Lemma \ref{3.32}. We also deduce from Lemma \ref{3.32} that each vertex of $St(w)\setminus \Gamma_{w}$ is in $\Gamma_{w}^{\perp}$. Moreover, $\Gamma_{w}^{\perp}\subset w^{\perp}$ since $w\in\Gamma_w$. Thus $St(w)$ is the full subgraph spanned by vertices in $\Gamma_w$ and $\Gamma^{\perp}_w$. So $St(w)$ is stable by Lemma \ref{3.16}. 

If $\Gamma_{w}\nsubseteq St(w)$, let $\Gamma_{11}$ be the full subgraph spanned by vertices in $\Gamma_{w}\cap lk(w)$ and let $\Gamma'_{2}$ be the full subgraph spanned by vertices in $\Gamma_{w}\setminus\Gamma_{11}$. By Lemma \ref{3.32}, $\Gamma_{w}=\Gamma_{11}\circ\Gamma'_{2}$ and $\Gamma'_{2}=\Gamma_{2}$. Note that $\Gamma_{2}$ is disconnected with isolated point $w\in\Gamma_{2}$, and $\Gamma_{11}$ may be empty. 

Let $V_{w}=v(\Gamma_{w})$ be the vertex set of $\Gamma_w$ and let $\Gamma_{12}$ be the full subgraph spanned by $V_{w}^{\perp}$. Then $\Gamma_{w}\circ\Gamma_{12}=\Gamma_{11}\circ\Gamma_{2}\circ\Gamma_{12}$ is stable by Lemma \ref{3.16}. Pick vertex $v\in\Gamma_{1}\setminus\Gamma_{11}$, then $v\in w^{\perp}\subset St(u)$ for any vertex $u\in\Gamma_{w}$ by Lemma \ref{3.32}, thus $v\in\Gamma_{12}$ and $\Gamma_{1}\subset\Gamma_{11}\circ\Gamma_{12}$. On the other hand, $w\in\Gamma_{2}$, so $\Gamma_{11}\circ\Gamma_{12}\subset\Gamma_{1}$ and $\Gamma_{1}=\Gamma_{11}\circ\Gamma_{12}$. Since $\Gamma_{2}$ does not contain any clique factor and $\Gamma_{11}\circ\Gamma_{2}\circ\Gamma_{12}=\Gamma_{1}\circ\Gamma_{2}$ is stable, we know $\Gamma_{1}$ is stable in $\Gamma$ by Theorem \ref{2.9}.
\end{proof}

\begin{remark}
\label{5.2}
In the above proof, $\Gamma_{12}$ may be empty. But if $\Gamma_{12}\neq\emptyset$, then it does not contain any clique join factor. Thus $\Gamma_{11}$ is the maximal clique join factor of $\Gamma_{11}\circ\Gamma_{2}\circ\Gamma_{12}$.
\end{remark}

The next result answers the question at the end of Example \ref{3.29}.

\begin{thm}
\label{5.3}
Suppose $\out(G(\Gamma))$ is finite and let $q:X(\Gamma)\to X(\Gamma')$ be a quasi-isometry. Then $q$ induces a simplicial isomorphism $q_{\ast}:\mathcal{P}(\Gamma)\to\mathcal{P}(\Gamma')$, in particular, $\out(G(\Gamma'))$ is transvection free.
\end{thm}

In the following proof, we identify $\Gamma$ with the one-skeleton of $F(\Gamma)$, which is the flag complex of $\Gamma$. Also recall that there are label-preserving projections $\pi:\mathcal{P}(\Gamma)\to F(\Gamma)$ and $\pi:\mathcal{P}(\Gamma')\to F(\Gamma')$.

\begin{proof}
By Lemma \ref{4.7}, there is a simplicial embedding $q_{\ast}:\mathcal{P}(\Gamma)\to\mathcal{P}(\Gamma')$. Note that $q_{\ast}(\mathcal{P}(\Gamma))$ is a full subcomplex in $\mathcal{P}(\Gamma')$. To see this, pick a simplex $\Delta'\subset\mathcal{P}(\Gamma')$ with its vertices in $q_{\ast}(\mathcal{P}(\Gamma))$. Then each vertex of $\Delta'$ comes from a stable standard geodesic line in $X(\Gamma')$. Thus there exists a stable standard flat $F'\subset X(\Gamma')$ such that $\Delta(F')=\Delta'$ by Lemma \ref{3.25}. By considering the quasi-isometry inverse of $q$, we know $F'$ is Hausdorff close to the $q$-image of a stable standard flat in $X(\Gamma)$. Thus $\Delta(F')=\Delta'\subset q_{\ast}(\mathcal{P}(\Gamma))$.

Pick vertex $p\in X(\Gamma)$ and let $\{\Delta_{i}\}_{i=1}^{k}$, $\{F_{i}\}_{i=1}^{k}$, $\{\Delta'_{i}\}_{i=1}^{k}$ and $\{F'_{i}\}_{i=1}^{k}$ be as in the proof of Lemma \ref{4.10}. We claim
\begin{equation}
\label{5.4}
\cap_{i=1}^{k}F'_{i}\neq\emptyset.
\end{equation}
Suppose (\ref{5.4}) is not true. Then there exist $1\le i_{1}\neq i_{2}\le k$ and hyperplane $h'\subset X(\Gamma)$ such that $h'$ separates $F'_{i_{1}}$ and $F'_{i_{2}}$. Let $l'$ be a standard geodesic that intersects $h'$ transversely and let $v'=\Delta(l')$. By the discussion in Lemma \ref{4.8}, we can find vertices $v'_{1}\in\Delta'_{i_{1}}$ and $v'_{2}\in\Delta'_{i_{2}}$ such that $v'_{1}$ and $v'_{2}$ are separated by $St(v')$. If there exists $i_{0}$ such that $F'_{i_{0}}\cap h\neq\emptyset$, then $v'\in q_{\ast}(\mathcal{P}(\Gamma'))$ and we can prove (\ref{5.4}) as in Lemma \ref{4.10}. Now we assume $F'_{i}\cap h'=\emptyset$ for any $i$. Let $w'=\pi(v')\in\Gamma'$ and let $\Gamma_{w'}$ be the minimal stable subgraph of $\Gamma'$ that contains $w'$.

We apply Lemma \ref{5.1} to $w'\in\Gamma'$, if case (1) is true, let $F'$ be the standard flat in $X(\Gamma')$ such that $l'\subset F'$ and $\Gamma_{F'}=\Gamma_{w'}$. Since $\Gamma_{w'}$ is stable, $\Delta(F')\subset q_{\ast}(\mathcal{P}(\Gamma'))$, in particular, $v'\in q_{\ast}(\mathcal{P}(\Gamma'))$ and we can prove (\ref{5.4}) as in Lemma \ref{4.10}.

If case (2) is true, let $\Gamma'_{1}=lk(w')$ and let $\Gamma'_{2}=lk(\Gamma'_{1})$. Take $K'_{1}$ and $K'$ to be the standard subcomplexes in $X(\Gamma')$ such that (1) the defining graphs $\Gamma_{K'_1},\Gamma_{K'}$ of $K'_{1}$ and $K'$ satisfy $\Gamma_{K'_{1}}=\Gamma'_{1}$ and $\Gamma_{K'}=\Gamma'_{1}\circ\Gamma'_{2}$; (2) $l'\subset K'$ and $K'_{1}\subset K'$. Set $M'_{1}=\Delta(K'_{1})$ and $M'=\Delta(K')$. Let $K'_{2}$ be an orthogonal complement of $K'_{1}$ in $K'$, i.e. $K'_{2}$ is a standard subcomplex such that $\Gamma_{K'_{2}}=\Gamma'_{2}$ and $K'=K'_{1}\times K'_{2}$. It follows that $M'=M'_{1}\ast M'_{2}$ for $M'_{2}=\Delta(K'_{2})$. By construction, $v'\in M'$ and $lk(v')= M'_{1}$.

Since $K'$ and $K'_{1}$ are stable, there exist stable standard subcomplexes $K$ and $K_{1}$ in $X(\Gamma)$ such that $q(K)\overset{\infty}{=}K'$ and $q(K_{1})\overset{\infty}{=}K_{1}'$. Moreover, by applying Theorem \ref{2.9} to the quasi-isometry between $K$ and $K'$, there exists a standard subcomplex $K_{2}\subset K$ such that $K=K_{1}\times K_{2}$ and $K_{2}$ is quasi-isometric to $K'_{2}$. Thus $\Gamma_{K_{2}}$ is also disconnected. Let $M_{i}=\Delta(K_{i})\subset\mathcal{P}(\Gamma)$ for $i=1,2$ and $M=M_{1}\ast M_{2}=\Delta(K)$. Then $q_{\ast}(M_{1})\subset M'_{1}$ (at this stage we may not know $q_{\ast}(M_{1})=M'_{1}$) and
\begin{equation}
\label{5.5}
q_{\ast}^{-1}(M'_{1})=M_{1}.
\end{equation}
To see this, pick a simplex $\Delta\subset\mathcal{P}(\Gamma)$ with $q_{\ast}(\Delta)\subset M'_{1}$. Suppose $\Delta=\Delta(F)$ for a stable standard flat $F\subset X(\Gamma)$. Then $q(F)\subset_{\infty}K'_{1}$, hence $F\subset_{\infty}K_{1}$ and $\Delta\subset M_{1}$.

Let $L=\cup_{i=1}^{k}\Delta_{i}$ and $L'=\cup_{i=1}^{k}\Delta'_{i}$. By the proof of Lemma \ref{4.10}, $L'\setminus (St(v')\cap L')$ is disconnected, thus $L\setminus q_{\ast}^{-1}(St(v')\cap L')$ is disconnected. Recall that $lk(v')=M'_1$, and we are assuming $v'\notin L'$. Thus $(St(v')\cap L')\subset M'_{1}$. Then $q_{\ast}^{-1}(St(v')\cap L')\subset q_{\ast}^{-1}(M'_{1})$, hence $q_{\ast}^{-1}(St(v')\cap L')\subset M_{1}$ by (\ref{5.5}).

Let $N=\pi(q_{\ast}^{-1}(St(v')\cap L'))$ and let $N_{i}=\pi(M_{i})$ for $i=1,2$. Then $N$ separates $F(\Gamma)$, $N\subset N_{1}$ and $N_{2}$ is disconnected. Pick vertices $u_{1},u_{2}$ in different connected components of $N_{2}$, then $d(u_{1},u_{2})\ge 2$ (since $N_{2}$ is the full subcomplex spanned by $\Gamma_{K_{2}}$). Since $\pi(M)=N_{1}\ast N_{2}\subset F(\Gamma)$, $N\subset St(u_{i})\setminus\{u_{i}\}$ for $i=1,2$. Let $\{C_{j}\}_{j=1}^{d}$ be the connected components of $F(\Gamma)\setminus N$. Then at most one of $C_{j}$ is contained in $St(u_{1})$. If $d\ge 3$, then $St(u_{1})$ would separate $F(\Gamma)$, which is a contradiction. Now we suppose $d=2$. Note that for $i=1,2$, there must exist $j$ such that $C_{j}\subset St(u_{i})$, otherwise $St(u_{i})$ would separate $F(\Gamma)$. Moreover, if $C_{j}\subset St(u_{i})$, then $u_{i}\in C_{j}$. So we can assume without loss of generality that $C_{1}\subset St(u_{1})$ and $C_{2}\subset St(u_{2})$, which implies $F(\Gamma)=St(u_{1})\cup St(u_{2})$, contradiction again (Lemma \ref{4.14}). Thus case (2) is impossible and (\ref{5.4}) is true.

Let $\{F_{\lambda}\}_{\lambda\in\Lambda}$ be the collection of maximal standard flats in $X(\Gamma)$. Then $X(\Gamma)=\cup_{\lambda\in\Lambda}F_{\lambda}$. For each $\lambda$, let $F'_{\lambda}$ be the unique maximal standard flat in $X(\Gamma')$ such that $q(F_{\lambda})\overset{\infty}{=}F'_{\lambda}$. Then 
\begin{equation}
\label{5.6}
X(\Gamma')\overset{\infty}{=}\cup_{\lambda\in\Lambda}F'_{\lambda}.
\end{equation}
Let $h\subset X(\Gamma')$ be an arbitrary hyperplane. Then $h\cap(\cup_{\lambda\in\Lambda}F'_{\lambda})\neq\emptyset$, otherwise $\cup_{\lambda\in\Lambda}F'_{\lambda}$ would stay on one side of the hyperplane since it is a connected set by (\ref{5.4}), and this contradicts (\ref{5.6}). Pick any standard geodesic $r\subset X(\Gamma')$ and let $h_{r}$ be a hyperplane dual to $r$. Then there exists $\lambda\in\Lambda$ such that $F'_{\lambda}\cap h_{r}\neq\emptyset$. It follows that $r\subset_{\infty} F'_{\lambda}$. So $\Delta(r)\in\Delta(F'_{\lambda})\subset q_{\ast}(\mathcal{P}(\Gamma))$, which implies $q_{\ast}$ is surjective on the vertices. However, $q_{\ast}(\mathcal{P}(\Gamma))$ is a full subcomplex in $\mathcal{P}(\Gamma')$, so $q_{\ast}$ is surjective.
\end{proof}

\subsection{Coherent ordering and coherent labeling} 
\label{subsec_coherent ordering and labelling}
Throughout this section, we assume $\out(G(\Gamma))$ is finite and $G(\Gamma)\ncong\mathbb Z$. If $q:G(\Gamma)\to G(\Gamma')$ is a quasi-isometry, then $G(\Gamma')$ has a quasi-action (cf. \cite[Definition 2.2]{kleiner2001groups}) on $G(\Gamma)$, which induces a group homomorphism:
\begin{equation*}
H:G(\Gamma')\to \qi(G(\Gamma))
\end{equation*}
On the other hand, since $G(\Gamma)$ acts by isometries on $X(\Gamma)$, we can identify $G(\Gamma)$ as a subgroup of $\qi(G(\Gamma))$ (more precisely, we embed $G(\Gamma)$ into $\isom(G(\Gamma),d_{w})$ and embed $\isom(G(\Gamma),d_{w})$ into $\qi(G(\Gamma))$ by Corollary \ref{4.20}). In this subsection, we will understand the following question.
\begin{center}
Does there exist $g\in \qi(G(\Gamma))$ such that $g\cdot H(G(\Gamma'))\cdot g^{-1}\subset G(\Gamma)$?
\end{center}

Recall that we have picked an identification between $G(\Gamma)$ and the 0-skeleton of $X(\Gamma)$. Each circle in the $1$-skeleton of the Salvetti complex of $G(\Gamma)$ is labeled by an element in the standard generating set $S$ of $G(\Gamma)$. Moreover, we have chosen an orientation for each such circle. By pulling back the labeling and orientation of edges to the universal cover $X(\Gamma)$, we obtain a $G(\Gamma)$-invariant directed labeling of edges in $X(\Gamma)$. Moreover, both the labeling and orientation of edges in $X(\Gamma)$ are compatible with parallelism between edges. This also induces an associated $G(\Gamma)$-invariant labeling of vertices in $\mathcal{P}(\Gamma)$.

Let $\{l_{\lambda}\}_{\lambda\in\Lambda}$ be the collection of standard geodesics in $X(\Gamma)$ and let $V_{\lambda}=v(l_{\lambda})$ be the vertex set of $l_{\lambda}$. A \textit{coherent ordering} of $G(\Gamma)$ is obtained by assigning a collection of bijections $f_{\lambda}:V_{\lambda}\to \Bbb Z$ for each $\lambda\in\Lambda$ such that if $l_{\lambda_{1}}$ and $l_{\lambda_{2}}$ are parallel, then the $f_{\lambda_{2}}\circ p\circ f^{-1}_{\lambda_{1}}:\Bbb Z\to\Bbb Z$ is a translation, where $p:V_{\lambda_{1}}\to V_{\lambda_{2}}$ is the map induced by parallelism. The map $f_{\lambda}$ pulls back the total order on $\Bbb Z$ to $V_{\lambda}$, which we denote by $\le_{\lambda}$. Then $p:V_{\lambda_{1}}\to V_{\lambda_{2}}$ is order-preserving. 

Two coherent ordering $\Omega_1$ and $\Omega_2$ are \emph{equivalent}, denoted by $\Omega_1=\Omega_2$, if their collections of bijections agree up to a translation of $\mathbb Z$. Recall that we have a $G(\Gamma)$-invariant orientation of edges in $X(\Gamma)$ which is compatible with parallelism between edges. This induces a unique coherent ordering $\Omega$ of $G(\Gamma)$ up to the equivalence relation defined before. Moreover, for any element $g\in G(\Gamma)$, the pull back $g^{\ast}(\Omega)$ is also a coherent ordering, moreover, $g^{\ast}(\Omega)=\Omega$.

Recall that for any vertex $v\in X(\Gamma)$, there is a label-preserving simplicial embedding $i_{v}:F(\Gamma)\to\mathcal{P}(\Gamma)$ by considering the standard geodesics passing through $v$. A \textit{coherent labeling} of $G(\Gamma)$ is a simplicial map $a:\mathcal{P}(\Gamma)\to F(\Gamma)$ such that $a\circ i_{v}: F(\Gamma)\to F(\Gamma)$ is a simplicial isomorphism for every vertex $v\in X(\Gamma)$. 

The label-preserving projection $L:\mathcal{P}(\Gamma)\to F(\Gamma)$ gives rise to a coherent labeling of $G(\Gamma)$. Recall that $G(\Gamma)$ acts on $\mathcal{P}(\Gamma)$ by simplicial automorphisms, and the labeling of vertices in $\mathcal{P}(\Gamma)$ is $G(\Gamma)$-invariant. Thus for any element $g\in G(\Gamma)$, the pull back $g^{\ast}(L)$ is also a coherent labeling and $g^{\ast}(L)=L$.

We have the following alternative characterization of elements in $\isom(G(\Gamma),d_{r})$.

\begin{lem}
\label{5.7}
There is a 1-1 correspondence which associates each element of $\isom(G(\Gamma),d_{r})$ to a triple consisting of:
\begin{enumerate}
\item A point $v\in G(\Gamma)$.
\item A coherent ordering of $G(\Gamma)$ (up to the equivalence relation defined above).
\item A coherent labeling of $G(\Gamma)$.
\end{enumerate} 
\end{lem}

\begin{proof}
Pick $\phi\in \isom(G(\Gamma),d_{r})$ and let $\varphi=h(\phi):\mathcal{P}(\Gamma)\to\mathcal{P}(\Gamma)$, where $h$ is the monomorphism in Remark \ref{4.18}. Then $\varphi^{\ast}L=L\circ \varphi:\mathcal{P}(\Gamma)\to F(\Gamma)$ is a coherent labeling of $G(\Gamma)$. Pick a standard geodesic $l_{1}\subset X(\Gamma)$. Then the parallel set $P_{l_{1}}$ admits a splitting $P_{l_{1}}=l_{1}\times l_{1}^{\perp}$. Since $\phi$ maps vertices in a standard flat bijectively to vertices in a standard flat, there exists a standard geodesic $l_{2}\subset X(\Gamma)$ such that $\phi(v(l_{1}))=v(l_{2})$ and $\phi(v(P_{l_{1}}))=P_{l_{2}}$, moreover, $\phi$ respects the product structure on $P_{l_{1}}$. Thus the pull-back $\phi^{\ast}\Omega$ is a coherent ordering of $G(\Gamma)$. Now we can set up the correspondence in one direction:
\begin{center}
$\phi$ $\rightsquigarrow$ $\phi(\id)$, $\phi^{\ast}\Omega$ and $\varphi^{\ast}L$
\end{center}
here $\id$ denotes the identity element of $G(\Gamma)$.

Conversely, given a point $v\in G(\Gamma)$, a coherent ordering $\Omega'$ and a coherent labeling $L'$, we can construct a map $\phi$ as follows. Set $\phi(\id)=v$. For $u\in G(\Gamma)$, pick a word $w_{u}=a_{1}a_{2}\cdots a_{n}$ representing $u$. Let $u_{i}$ be the point in $G(\Gamma)$ represented by the word $a_{1}a_{2}\cdots a_{i}$ for $1\le i\le n$ and let $u_{0}=\id$. We define $q_{i}=\phi(a_{1}a_{2}\cdots a_{i})\in G(\Gamma)$ inductively as follows. Set $q_{0}=v$ and suppose $q_{i-1}$ is already defined. Denote the standard geodesic containing $u_{i-1}$ and $u_{i}$ by $l_{i}$. Let $v_{i}=L'(\Delta(l_{i}))$ which is a vertex of $\Gamma$, and let $l'_{i}$ be the standard line that contains $q_{i-1}$ and is labeled by $v_i$. Denote the vertex set of $l_{i}$ with the order from $\Omega'$ by $(v(l_{i}),\le_{\Omega'})$. Suppose $k:(v(l_{i}),\le_{\Omega'})\to(v(l'_{i}),\le_{\Omega})$ is the unique order preserving bijection such that $k(u_{i-1})=q_{i-1}$. Then we define $q_{i}=k(u_{i})$. 

We claim that for any other word $w'_{u}$ representing $u$, $\phi(w_{u})=\phi(w'_{u})$, hence there is a well-defined map $\phi: G(\Gamma)\to G(\Gamma)$. To see this, recall that one can obtain $w_{u}$ from $w'_{u}$ by performing the following two basic moves:
\begin{enumerate}
\item $w_{1}aa^{-1}w_{2}\to w_{1}w_{2}$.
\item $w_{1}abw_{2}\to w_{1}baw_{2}$ when $a$ and $b$ commute.
\end{enumerate}
It is clear that $\phi(w_{1}aa^{-1}w_{2})=\phi(w_{1}w_{2})$. For the second move, let $u_{i-1},u_{i},u'_{i}$, and $u_{i+1}$ be points in $G(\Gamma)$ represented by $w_{1},w_{1}a,w_{1}b$ and $w_{1}ab=w_{1}ba$ respectively. Define $q_{i-1}=\phi(w_{1}),q_{i}=\phi(w_{1}a),q'_{i}=\phi(w_{1}b),q_{i+1}=\phi(w_{1}ab)$ and $q'_{i+1}=\phi(w_{1}ba)$. Since $L'$ is a coherent labeling, $\angle_{q_{i}}(q_{i+1},q_{i-1})=\angle_{q_{i-1}}(q_{i},q'_{i})=\angle_{q'_{i}}(q_{i-1},q'_{i+1})=\pi/2$, moreover, the standard geodesic containing $q_{i}$ and $q_{i+1}$ is parallel to the standard geodesic containing $q_{i-1}$ and $q'_{i}$. Since $\Omega'$ is a coherent ordering, $d(q_{i},q_{i+1})=d(q_{i-1},q'_{i})$, thus $\overline{q_{i}q_{i+1}}$ and $\overline{q_{i-1}q'_{i}}$ are parallel. Similarly, $\overline{q_{i-1}q_{i}}$ and $\overline{q'_{i}q'_{i+1}}$ are parallel, thus $q_{i+1}=q'_{i+1}$.

Now we define another map $\phi':G(\Gamma)\to G(\Gamma)$, which serves as the inverse of $\phi$. Set $\phi'(v)=\id$ and pick word $w=a_{1}a_{2}\cdots a_{n}$. Let $r_{i}$ be the point in $G(\Gamma)$ represented by $va_{1}a_{2}\cdots a_{i}$ for $1\le i\le n$ and $r_{0}=v$. We define $p_{i}=\phi'(va_{1}a_{2}\cdots a_{i})$ inductively as follows. Put $p_{0}=\id$ and suppose $p_{i-1}$ is already defined. Since $L'$ is a coherent labeling, there exists a unique standard geodesic $l_{i}$ containing $p_{i-1}$ such that $L'(\Delta(l_{i}))$ and the edge $\overline{r_{i-1}r_{i}}$ share the same label. Let $l'_{i}$ be the unique standard geodesic containing $r_{i-1}$ and $r_{i}$ and let $k':(v(l'_{i}),\le_{\Omega})\to(v(l_{i}),\le_{\Omega'})$ be the unique order preserving bijection such that $k'(r_{i-1})=p_{i-1}$. Put $p_{i}=k'(r_{i})$. By a similar argument as above, $\phi': G(\Gamma)\to G(\Gamma)$ is well-defined. It is not hard to deduce the following properties from our construction: 
\begin{enumerate}
\item $\phi'\circ\phi=\phi\circ\phi'=\textmd{Id}$.
\item $d_{r}(\phi(v_{1}),\phi(v_{2}))\le d_{r}(v_{1},v_{2})$ and $d_{r}(\phi'(v_{1}),\phi'(v_{2}))\le d_{r}(v_{1},v_{2})$ for any vertices $v_{1},v_{2}\in G(\Gamma)$.
\item If $L'=L$ and $\Omega'=\Omega$, then $\phi$ is a left translation. If in addition $v=\id$, then $\phi=\textmd{Id}$.
\end{enumerate}
It follows from (1) and (2) that $\phi\in \isom(G(\Gamma),d_{r})$. Moreover, $v=\phi(\id)$, $L'=\varphi^{\ast}L$ ($\varphi=h(\phi)$ where $h$ is the monomorphism in Remark \ref{4.18}) and $\Omega'=\phi^{\ast}\Omega$, thus we have established the required 1-1 correspondence.
\end{proof}

Pick finite simplicial graphs $\Gamma$ and $\Gamma'$ such that (1) $\out(G(\Gamma))$ is finite; (2) there exists a simplicial isomorphism $s:\mathcal{P}(\Gamma)\to\mathcal{P}(\Gamma')$. By Lemma \ref{4.10}, $s$ induces a map $\phi:G(\Gamma)\to G(\Gamma')$. For every $g'\in G(\Gamma')$, there is a left translation $\bar{\phi}_{g'}:G(\Gamma')\to G(\Gamma')$, which gives rise to a simplicial isomorphism $\bar{s}_{g'}:\mathcal{P}(\Gamma')\to\mathcal{P}(\Gamma')$. Let $s_{g'}=s^{-1}\circ \bar{s}_{g'}\circ s$. Then $s_{g'}$ gives rise to a map $\phi_{g'}\in \isom(G(\Gamma),d_{r})$ by Corollary \ref{4.16}, moreover, by Lemma \ref{4.10},
\begin{equation}
\label{5.8}
\bar{\phi}_{g'}\circ\phi=\phi\circ\phi_{g'}
\end{equation}
for any $g'\in G(\Gamma')$. So $G(\Gamma')$ acts on $G(\Gamma)$, and we can define a homomorphism $\Phi:G(\Gamma')\to \isom(G(\Gamma),d_{r})$ by sending $g'$ to $\phi_{g'}$. $\Phi$ is injective since each step in defining $\Phi$ is injective.

\begin{lem}
\label{5.9}
In the above setting, there exists an element $\phi_{1}\in \isom(G(\Gamma),d_{r})$ such that it conjugates the image of $\Phi$ to a finite index subgroup of $G(\Gamma)$.
\end{lem}

We identify $G(\Gamma)$ as a subgroup of $\isom(G(\Gamma),d_{r})$ via the left action of $G(\Gamma)$ on itself.

\begin{proof}
Pick a reference point $q\in \textmd{Im}\ \phi$ and let $K_{q}=(F(\Gamma'))_{q}$. Denote the points in $\phi^{-1}(q)$ by $\{p_{\lambda}\}_{\lambda\in\Lambda}$ and let $K_{p_{\lambda}}=(F(\Gamma))_{p_{\lambda}}$. Since $\{\phi(K_{p_{\lambda}})\}_{\lambda\in\Lambda}$ are distinct subcomplexes of $K_{q}$, $\Lambda$ is a finite set.

Let $L:\mathcal{P}(\Gamma)\to F(\Gamma)$ and $\Omega$ be the coherent labeling and coherent ordering induced by the $G(\Gamma)$-invariant labeling of $X(\Gamma)$ and $\mathcal{P}(\Gamma)$. We can obtain a coherent labeling $L':\mathcal{P}(\Gamma')\to F(\Gamma')$ and a coherent ordering $\Omega'$ for $G(\Gamma')$ in a similar fashion which are invariant under the $G(\Gamma')$-action, i.e.
\begin{equation}
\label{5.10}
(\bar{s}_{g'})^{\ast}L'=L'\ \textmd{and}\ (\bar{\phi}_{g'})^{\ast}\Omega'=\Omega'.
\end{equation}
Our goal is to find a coherent labeling $L_{1}$ and a coherent ordering $\Omega_{1}$ of $G(\Gamma)$ such that $(s_{g'})^{\ast}L_{1}=L_{1}$ and $(\phi_{g'})^{\ast}\Omega_{1}=\Omega_{1}$ for any $g'\in G(\Gamma')$.

Let $i_{q}:F(\Gamma')\to \mathcal{P}(\Gamma')$ be the canonical embedding and let
\begin{equation*}
L_{1}=L\circ s^{-1}\circ i_{q}\circ L'\circ s
\end{equation*}
be the simplicial map from $\mathcal{P}(\Gamma)$ to $F(\Gamma)$. Pick arbitrary $p\in G(\Gamma)$ and let $i_{p}:F(\Gamma)\to\mathcal{P}(\Gamma)$ be the canonical embedding. We need to show $L_{1}\circ i_{p}$ is a simplicial isomorphism. Let $K_{p}=i_{p}(F(\Gamma))$ and let $g_{1}'\in G(\Gamma')$ such that $g_{1}'\cdot \phi(p)=q$. Then $i_{q}\circ L'|_{s(K_{p})}=\bar{s}_{g_{1}'}|_{s(K_{p})}$. Thus 
\begin{align*}
L_{1}\circ i_{p}=L\circ s^{-1}\circ i_{q}\circ L'\circ s\circ i_{p}=L\circ s^{-1}\circ \bar{s}_{g'_{1}}\circ s\circ i_{p}=L\circ s_{g'_{1}}\circ i_{p},
\end{align*}
which is a simplicial isomorphism by Lemma \ref{4.10}. It follows that $L_{1}$ is a coherent labeling, moreover,
\begin{align*}
(s_{g'})^{\ast}L_{1} &= (L\circ s^{-1}\circ i_q\circ L'\circ s)\circ(s^{-1}\circ \bar{s}_{g'}\circ s)=L\circ s^{-1}\circ i_q\circ L'\circ \bar{s}_{g'}\circ s \\
&=L\circ s^{-1}\circ i_q\circ L'\circ s=L_{1}
\end{align*}
for any $g'\in G(\Gamma')$, where the third equality follows from (\ref{5.10}). So $L_{1}$ is the required coherent labeling.

To simplify notation, we will write $x<_{\Omega}y$ if $x<y$ under the ordering $\Omega$. We define $\Omega_{1}$ as follows. Let $p_{1},p_{2}\in G(\Gamma)$ be two distinct points in a standard geodesic line. If $\phi(p_{1})\neq\phi(p_{2})$, then we set $p_{1}<_{\Omega_{1}} p_{2}$ if and only if $\phi(p_{1})<_{\Omega'}\phi(p_{2})$. If $\phi(p_{1})=\phi(p_{2})$, then by (\ref{5.8}), there exists a unique $g'\in G(\Gamma')$ such that $\phi_{g'}(p_{i})\in \phi^{-1}(q)$ for $i=1,2$ and we set $p_{1}<_{\Omega_{1}}p_{2}$ if and only if $\phi_{g'}(p_{1})<_{\Omega}\phi_{g'}(p_{2})$. It follows from (\ref{5.10}), (\ref{5.8}) and our construction that $p_{1}<_{\Omega_{1}} p_{2}$ if and only if $\phi_{g'}(p_{1})<_{\Omega_{1}}\phi_{g'}(p_{2})$ for any $p_{1}, p_{2}$ in the same standard geodesic line and any $g'\in G(\Gamma')$, thus $(\phi_{g'})^{\ast}\Omega_{1}=\Omega_{1}$.

To verify $\Omega_{1}$ is coherent, pick parallel standard geodesics $l_{1}$ and $l_{2}$ in $X(\Gamma)$ and pick distinct vertices $p_{11},p_{12}\in l$. Let $p_{21},p_{22}$ be the corresponding vertices in $l_{2}$ via parallelism. We assume $p_{11}<_{\Omega_{1}}p_{12}$, it suffices to prove $p_{21}<_{\Omega_{1}}p_{22}$. 

\textit{Case 1:} We assume $\phi(p_{11})\neq\phi(p_{12})$. Recall that $l_{1}$ can be realized as an intersection of finitely many maximal standard flats, so by Lemma \ref{4.10}, there exists a standard geodesic line $l'_{1}\subset X(\Gamma')$ such that $\phi(v(l_{1}))\subset v(l'_{1})$ and $\phi(v(P_{l_{1}}))\subset v(P_{l'_{1}})$, moreover, $\phi$ respects the product structures of $P_{l_{1}}$ and $P_{l'_{1}}$. Thus $\overline{\phi(p_{11})\phi(p_{21})}$ and $\overline{\phi(p_{21})\phi(p_{22})}$ are the opposite sides of a flat rectangle in $X(\Gamma')$. Now $p_{21}<_{\Omega_{1}}p_{22}$ follows since $\Omega'$ is coherent.

\textit{Case 2:} We assume $\phi(p_{11})=\phi(p_{12})\neq\phi(p_{21})$. In this case, we can assume without loss of generality that $\phi(p_{11})=\phi(p_{12})=q$ (since $(\phi_{g'})^{\ast}\Omega_{1}=\Omega_{1}$) and the points $p_{11}$ and $p_{21}$ stay in the same standard geodesic. For $i=1,2$, let $r_{i}$ be the standard geodesic passing $p_{1i}$ and $p_{2i}$. Take $r'_{i}\subset X(\Gamma')$ and $l'_{i}\subset X(\Gamma')$ to be the standard geodesics such that $\phi(v(r_{i}))\subset v(r'_{i})$ and $\phi(v(l_{i}))\subset v(l'_{i})$ respectively. Denote $q'=\phi(p_{21})$. Since $\phi$ restricted on $v(P_{l_{1}})$ respects the product structure, $\phi(p_{21})=\phi(p_{22})=q'$ and $r'_{1}=r'_{2}$. 

Let $\overline{\phi}_{g'}$ be the left translation such that $\overline{\phi}_{g'}(q')=q$. Since $q'\in r'_{1}$ and $q\in r'_{1}$,  $\overline{\phi}_{g'}$ is a translation along $r'_{1}$ and $\overline{s}_{g'}$ fixes every point in $St(\Delta(r'_{1}))$, hence $s_{g'}$ fixes every point in $s^{-1}(St(\Delta(r'_{1})))=St(\Delta(r_{1}))$ and 
\begin{equation}
\label{5.11}
\phi_{g'}(r_{i})=r_{i}
\end{equation}
for $i=1,2$. Let $l_{3}=\phi_{g'}(l_{2})$. Then $l_{3}$ is parallel to $l_{1}$ (or $l_{2}$). To see this, note that $\Delta(l_{1})\in St(\Delta(r_{1}))$, hence $\Delta(l_{1})$ is fixed by $s_{g'}$. Put $p_{3i}=\phi_{g'}(p_{2i})$ for $i=1,2$. Then $p_{3i}\in r_{i}$ by (\ref{5.11}), hence $\overline{p_{11}p_{12}}$ and $\overline{p_{31}p_{32}}$ are the opposite sides of a flat rectangle. Moreover $p_{3i}\in\phi^{-1}(q)$ for $i=1,2$ by (\ref{5.8}), then $p_{31}<_{\Omega_{1}} p_{32}$ since $\Omega$ is coherent and $\Omega=\Omega_{1}$ while restricted on $\phi^{-1}(q)$. Now the $G(\Gamma')$-invariance of $\Omega_{1}$ implies $p_{21}<_{\Omega_{1}}p_{22}$.

\textit{Case 3:} If $\phi(p_{11})=\phi(p_{12})=\phi(p_{21})$, then we can assume without loss of generality that they all equal to $q$. It follows that $\phi(p_{22})=q$ since $\phi$ respects the product structure while restricted on $v(P_{l_{1}})$. Thus $p_{21}<_{\Omega_{1}}p_{22}$ by definition.

By Lemma \ref{5.7}, there exists $\phi_{1}\in \isom(G(\Gamma),d_{r})$ such that $\phi^{\ast}_{1}\Omega=\Omega_{1}$ and $s^{\ast}_{1}L=L_{1}$ ($s_{1}=h(\phi_{1})$ where $h$ is the monomorphism in Remark \ref{4.18}). Thus 
\begin{align*}
(\phi_{1}\circ\phi_{g'}\circ\phi^{-1}_{1})^{\ast}\Omega =(\phi^{-1}_{1})^{\ast}\circ(\phi_{g'})^{\ast}\circ(\phi_{1}^{\ast}\Omega)= (\phi^{-1}_{1})^{\ast}\circ(\phi_{g'})^{\ast}\Omega_{1}= (\phi^{-1}_{1})^{\ast}\Omega_{1}=\Omega
\end{align*}
for any $g'\in G(\Gamma')$. Similarly, $(s_{1}\circ s_{g'}\circ s^{-1}_{1})^{\ast}L=L$ for any $g'\in G(\Gamma')$. Note that $s_{1}\circ s_{g'}\circ s^{-1}_{1}=h(\phi_{1}\circ\phi_{g'}\circ\phi^{-1}_{1})$, thus by Lemma \ref{5.7}, $G(\Gamma')$ acts on $G(\Gamma)$ by left translations via $g'\to\phi_{1}\circ\phi_{g'}\circ\phi^{-1}_{1}$. This induces a monomorphism $G(\Gamma')\to G(\Gamma)$. Moreover, by (\ref{5.8}) and the fact that $\phi^{-1}(q)$ is finite, this action has finite quotient, thus we can realize $G(\Gamma')$ as a finite index subgroup of $G(\Gamma)$.
\end{proof}

The next result basically says under suitable conditions, if there exists a quasi-isometry $q:G(\Gamma)\to G(\Gamma')$, then there exists a very \textquotedblleft nice\textquotedblright\  quasi-isometry $q':G(\Gamma)\to G(\Gamma')$. However, we do not insist that $q'$ is of bounded distance away from $q$ (compared to Theorem \ref{4.15}). 
\begin{thm}
\label{5.12}
Let $\Gamma$ and $\Gamma'$ be finite simplicial graphs such that $\out(G(\Gamma))$ is finite and $G(\Gamma')$ is quasi-isometric to $G(\Gamma)$. Then there exists a cubical map (cf. Definition \ref{cubical}) $\varphi:X(\Gamma)\to X(\Gamma')$ such that
\begin{enumerate}
	\item The map $\varphi$ is onto, and $\varphi$ maps any standard flat in $X(\Gamma)$ onto a standard flat in $X(\Gamma')$ of the same dimension.
	\item The map $\varphi$ maps combinatorial geodesics in the 1-skeleton of $X(\Gamma)$ to combinatorial geodesics in the 1-skeleton of $X(\Gamma')$.
	\item The map $\varphi$ is a quasi-isometry.
\end{enumerate}
\end{thm}

\begin{proof}
Let $f: G(\Gamma)\to G(\Gamma')$ be a quasi-isometry. By Theorem \ref{5.3}, $f$ induces a simplicial isomorphism $s:\mathcal{P}(\Gamma)\to\mathcal{P}(\Gamma')$. By Lemma \ref{4.10}, $s$ induces a map $\phi:G(\Gamma)\to G(\Gamma')$ such that $d_{w}(f(x),\phi(x))<D$ for any $x\in G(\Gamma)$. Let $\phi_{1}$ be the map in Lemma \ref{5.9} and let $\varphi=\phi\circ\phi_{1}^{-1}$. We will use the same notation as in the proof of Lemma \ref{5.9}. 

We claim that if $F=\cap_{i=1}^{h}F_{i}$ where each $F_{i}$ is a maximal standard flat, then there exists a unique standard flat $F'\subset G(\Gamma')$ such that $\phi(v(F))=v(F')$. To see this, let $F'_{i}$ be the maximal standard flat in $X(\Gamma')$ such that $\Delta(F'_{i})=s(\Delta(F_{i}))$ for $1\le i\le h$ and let $F'=\cap_{i=1}^{h}F'_{i}$. Then it follows from Lemma \ref{4.10} that $\phi(v(F))\subset v(F')$. Recall that $G(\Gamma')$ acts on $G(\Gamma')$, $\mathcal{P}(\Gamma')$, $G(\Gamma)$ and $\mathcal{P}(\Gamma)$. The stabilizer $Stab(v(F'))$ fixes $\Delta(F'_{i})$ for all $i$, hence it fixes $\Delta_{i}$ for all $i$ and $Stab(v(F'))\subset Stab(v(F))$. Since $Stab(v(F'))$ acts on $v(F')$ transitively, (\ref{5.8}) implies $\phi(v(F))=v(F')$ and $|\phi^{-1}(y)\cap F|=|\phi^{-1}(y')\cap F|$ for any $y,y'\in v(F')$. It also follows that $Stab(v(F))\subset Stab(v(F'))$, thus $Stab(v(F'))= Stab(v(F))$.

Note that the above claim is also true for $\varphi$, and any standard geodesic satisfies the assumption of the claim. Moreover, $\varphi$ is surjective since $\phi_{1}$ is surjective by (\ref{5.8}). Pick standard geodesic $l\subset X(\Gamma)$ and $l'\subset X(\Gamma')$ such that $v(l')=\varphi(v(l))$, and we identify $v(l)$ and $v(l')$ with $\Bbb Z$ in an order-preserving way. Then the above claim and the construction of $\phi_{1}$ imply that $\varphi|_{v(l)}$ is of form 
\begin{equation}
\label{5.13}
\varphi(a)=\lfloor a/d \rfloor+r
\end{equation}
for some integers $r$ and $d$ ($d\ge 1$). In particular, $\varphi$ can be extended to a simplicial map from the Cayley graph $C(\Gamma)$ of $G(\Gamma)$ to $C(\Gamma')$.

Pick a combinatorial geodesic $\omega\subset C(\Gamma)$ connecting vertices $x$ and $y$, we claim that $\omega'=\phi(\omega)$ is also a geodesic in $C(\Gamma')$ (it could be a point). Let $\{v_{i}\}_{i=0}^{n}$ be vertices in $\omega$ such that for $0\le i\le n-1$, $[v_{i},v_{i+1}]$ is a maximal sub-segment of $\omega$ that is contained in a standard geodesic ($v_{0}=x$ and $v_{n}=y$). Denote the corresponding standard geodesic by $l_{i}$. For $0\le i\le n-1$, let $l'_{i}\subset X(\Gamma')$ be the standard geodesic such that $v(l'_{i})=\varphi(v(l_{i}))$ and $\omega'_{i}=\phi([v_{i},v_{i+1}])$. Then $\omega'_{i}$ is a (possibly degenerate) segment in $l'_{i}$ by (\ref{5.13}). Since $\omega$ is a geodesic, none of two geodesics in $\{l_{i}\}_{i=0}^{n-1}$ are parallel. Note that $\varphi$ is induced by a simplicial isomorphism between $\mathcal{P}(\Gamma)$ and $\mathcal{P}(\Gamma')$, thus the same property is true for the collection of geodesics $\{l'_{i}\}_{i=0}^{n-1}$. It follows that no hyperplane in $X(\Gamma')$ could intersect $\omega'$ at more than one point, hence $\omega'$ is a combinatorial geodesic.

Let $u_{i}=\varphi(v_{i})$. Then $d_{w}(u_{i},u_{i+1})\le d_{w}(v_{i},v_{i+1})$ by (\ref{5.13}) (recall that $d_w$ denotes the word metric on the corresponding group). Thus
\begin{equation}
d_{w}(\varphi(x),\varphi(y))=\sum_{i=0}^{n-1}d_{w}(u_{i},u_{i+1})\le\sum_{i=0}^{n-1}d_{w}(v_{i},v_{i+1})=d_{w}(x,y)
\end{equation}
for any $x,y\in G(\Gamma)$.

Pick $p\in G(\Gamma')$ and let $k=|\varphi^{-1}(p)|$. Then $k$ does not depend on $p$ by (\ref{5.8}). It follows that $d_{w}(\varphi(x),\varphi(y))\ge 1$ whenever $d_{w}(x,y)\ge k+1$. Now we can cut $\omega$ into pieces of length $k+1$. Since $\varphi(\omega)$ is a combinatorial geodesic,
\begin{equation*}
d_{w}(\varphi(x),\varphi(y))\ge \frac{d_{w}(x,y)}{k+1}-1.
\end{equation*}
Note that $\varphi$ naturally extends to a cubical map from $X(\Gamma)$ to $X(\Gamma')$, which satisfies all the required properties.
\end{proof}

\begin{thm}
\label{5.15}
If $\Gamma$ and $\Gamma'$ are finite simplicial graphs such that $\out(G(\Gamma))$ is finite, then the following are equivalent:
\begin{enumerate}
\item $G(\Gamma')$ is quasi-isometric to $G(\Gamma)$.
\item $\mathcal{P}(\Gamma')$ is isomorphic to $\mathcal{P}(\Gamma)$ as simplicial complexes.
\item $G(\Gamma')$ is isomorphic to a subgroup of finite index in $G(\Gamma)$.
\end{enumerate}
\end{thm}

\begin{proof}
$(1)\Rightarrow(2)$ follows from Theorem \ref{5.3}. $(2)\Rightarrow(3)$ follows from Lemma \ref{5.9}. $(3)\Rightarrow(1)$ is trivial.
\end{proof}

This establishes Theorem \ref{1.3} in the introduction.

\section{The geometry of finite index RAAG subgroups}
\label{sec_geometry of f.i. raag subgroups}
Throughout this section, we assume $G(\Gamma)\ncong \Bbb Z$, since the main results of this section (Theorem \ref{6.15} and Theorem \ref{6.20}) are trivial when $G(\Gamma)\cong\Bbb Z$.

\subsection{Constructing finite index RAAG subgroups} 
\label{subsec_construct finite index}
A \textit{right-angled Artin subgroup} is a subgroup which is also a right-angled Artin group. In this section, we introduce a process to obtain finite index RAAG subgroups of an arbitrary RAAG.

\begin{lem}
\label{6.1}
Let $X$ be a $CAT(0)$ cube complex, let $l\subset X$ be a geodesic in the 1-skeleton and let $\{h_{i}\}_{i\in \Bbb Z}$ be consecutive hyperplanes dual to $l$. Let $\pi_l:X\to l$ be the $CAT(0)$ projection. Then
\begin{enumerate}
\item For every edge $e\subset X$, if $e\cap h_{i}=\emptyset$ for all $i$, then $\pi_{l}(e)$ is a vertex in $l$, if $e\cap h_{i}\neq\emptyset$ for some $i$, then $\pi_{l}(e)$ is an edge in $l$.
\item If $K$ is any connected subcomplex such that $e\cap h_{i}=\emptyset$ for all $i$, then $\pi_{l}(K)$ is a vertex in $l$, moreover, if $K$ stays between $h_{i}$ and $h_{i+1}$, then $\pi_{l}(K)$ is the vertex in $l$ that stays between $h_{i}$ and $h_{i+1}$.
\item For every interval $[a,b]\subset l$, $\pi^{-1}_{l}([a,b])$ is a convex set in $X$. In particular, if $x\in l$ is a vertex, then $\pi^{-1}_{l}(x)$ is a convex subcomplex of $X$.
\item If $K$ is a convex subcomplex such that $K\cap l\neq\emptyset$, then $\pi_{l}(K)=K\cap l$.
\end{enumerate}
\end{lem}

\begin{proof}
Here (1) and (3) follow from the fact the every hyperplane has a carrier, and (2) follows from (1). To see (4), it suffices to show for $i$ such that $h_{i}\cap l\neq\emptyset$ and $h_{i}\cap K\neq\emptyset$, we have $e_{i}\subset K$ ($e_{i}$ is the edge in $l$ dual to $h_{i}$). Let $N_{h_{i}}$ be the carrier of $h_{i}$. By Lemma \ref{2.2}, $d(x,N_{h_{i}}\cap K)\equiv c$ for any $x\in e_{i}$. Moreover, $d(x,N_{h_{i}}\cap K)=d(x,K)$ for $x$ in the interior of $e_{i}$, so we must have $c=0$, otherwise the convexity of $d(\cdot,K)$ would imply $K\cap l=\emptyset$.
\end{proof}

\begin{lem}
\label{6.2}
Let $l\subset X(\Gamma)$ be a standard geodesic. Then there is a map $\pi_{\Delta(l)}:v(\mathcal{P}(\Gamma)\setminus St(\Delta(l)))\to v(l)$ (recall that $v(\mathcal{P}(\Gamma)\setminus St(\Delta(l))$ is the collection of vertices in $\mathcal{P}(\Gamma)\setminus St(\Delta(l)$) such that if $v_{1}$ and $v_{2}$ are in the same connected component of $\mathcal{P}(\Gamma)\setminus St(\Delta(l))$, then $\pi_{\Delta(l)}(v_{1})=\pi_{\Delta(l)}(v_{2})$.
\end{lem}

\begin{proof}
Let $\pi_{l}: X(\Gamma)\to l$ be the CAT(0) projection and let $l_{1}\subset X(\Gamma)$ be a standard geodesic such that $d(\Delta(l_{1}),\Delta(l))\ge 2$. Then $\pi_{l}(l_{1})$ is a vertex in $l$ by Lemma \ref{3.1} and Corollary \ref{3.2}. Moreover, we claim $\pi_{l}(l_{1})=\pi_{l}(l_{2})$ if $l_{2}$ is a standard geodesic parallel to $l_{1}$. It suffices to prove the case when there is a unique hyperplane $h$ separating $l_{1}$ from $l_{2}$. Note that $d(\Delta(l_{1}),\Delta(l))\ge 2$ yields $h\cap l=\emptyset$, so $l_{1}$ and $l_{2}$ are pinched by two hyperplanes dual to $l$, then the claim follows from Lemma \ref{6.1}. Thus $\pi_{l}$ induces a well-defined map $\pi_{\Delta(l)}:v(\mathcal{P}(\Gamma)\setminus St(\Delta(l)))\to v(l)$. If $\Delta(l_{1})$ and $\Delta(l_{2})$ are connected by an edge, then there exist standard geodesics $l'_{1}$ and $l'_{2}$ such that $l'_{1}\cap l'_{2}\neq\emptyset$ and $l'_{i}$ is parallel to $l_{i}$ for $i=1,2$. Thus $\pi_{l}(l_{1})=\pi_{l}(l'_{1})=\pi_{l}(l'_{2})=\pi_{l}(l_{2})$ and $\pi_{\Delta(l)}(\Delta(l_{1}))=\pi_{\Delta(l)}(\Delta(l_{2}))$.
\end{proof}

Pick a standard generating set $S$ of $G(\Gamma)$ and let $C(\Gamma,S)$ be the Cayley graph. We identify $G(\Gamma)$ as a subset of $C(\Gamma,S)$ and attach higher dimensional cubes to $C(\Gamma,S)$ to obtain a $CAT(0)$ cube complex $X(\Gamma,S)$, which is basically the universal cover of the Salvetti complex. Here we would like to think $G(\Gamma)$ as a fixed set and $C(\Gamma,S),X(\Gamma,S)$ as objects formed by adding edges and cubes to $G(\Gamma)$ in a particular way determined by $S$, so we write $S$ explicitly. We will choose a $G(\Gamma)$-equivariant orientation for edges in $X(\Gamma,S)$ as before.

An \textit{$S$-flat} (or an \textit{$S$-geodesic}) in $G(\Gamma)$ is defined to be the vertex set of a standard flat (or geodesic) in $X(\Gamma,S)$. We define $\mathcal{P}(\Gamma,S)$ as before such that its vertices correspond to coarse equivalence classes of $S$-geodesics.

We define an isometric embedding $I: G(\Gamma)\to \ell^{1}(v(\mathcal{P}(\Gamma,S)))$ which depends on $S$ and the orientation of edges in $X(\Gamma,S)$. Pick standard geodesic $l\subset X(\Gamma,S)$ and let $\pi_{l}: X(\Gamma,S)\to l$ be the $CAT(0)$ projection. We identify $v(l)$ with $\Bbb Z^{\Delta(l)}$ in an orientation preserving way such that $\pi_{l}(\id)=0$ ($\id$ is the identity element in $G(\Gamma)$). Then $\pi_{l}$ induces a coordinate function $I_{\Delta(l)}:G(\Gamma)\to \Bbb Z^{\Delta(l)}$. If we change $l$ to a standard geodesic $l_{1}$ parallel to $l$, then $I_{\Delta(l)}$ and $I_{\Delta(l_{1})}$ are identical by Lemma \ref{6.1}. Thus for every vertex $v\in \mathcal{P}(\Gamma)$, there is a well-defined coordinate function $I_{v}:G(\Gamma)\to \Bbb Z^{v}$. These coordinate functions induce a map $I:G(\Gamma)\to \Bbb Z^{(v(\mathcal{P}(\Gamma)))}$. 

$I$ is an embedding since every two points in $G(\Gamma)$ are separated by some hyperplane. $I(G(\Gamma))\subset \ell^{1}(v(\mathcal{P}(\Gamma)))$ since for any $g\in G(\Gamma)$, there are only finitely many hyperplanes separating $\id$ and $g$. $I$ naturally extends to a map $I: X(\Gamma,S)\to \ell^{1}(v(\mathcal{P}(\Gamma)))$ and it maps combinatorial geodesics to geodesics by the argument in Theorem \ref{5.12}. Thus $I$ is an isometric embedding with respect to the $\ell^{1}$ metric on $X(\Gamma,S)$. We say a convex subcomplex $K\subset X(\Gamma,S)$ is \textit{non-negative} if each point in $I(K)$ has non-negative coordinates (this notion depends on the orientation of edges in $X(\Gamma,S)$). Let $CN(\Gamma,S)$ be the collection of compact, convex, non-negative subcomplexes of $X(\Gamma,S)$ that contain the identity. 

For any $K\in CN(\Gamma,S)$, we find a maximal collection of standard geodesics $\{c_{i}\}_{i=1}^{s}$ such that $c_{i}\cap K\neq\emptyset$ for all $i$ and $\Delta(c_{i})\neq \Delta(c_{j})$ for any $i\neq j$. Let $g_{i}\in S$ be the label of edges in $c_{i}$ and let $\alpha_{i}=\pi_{c_{i}}(\id)$. Put $n_{i}=|v(K\cap c_{i})|$ and $v_{i}=\alpha_{i}g^{n_{i}}_{i}\alpha^{-1}_{i}$. Let $G'$ be the subgroup generated by $\{v_{i}\}_{i=1}^{s}$. It follows from the convexity of $K$ that if a standard geodesic $c$ is parallel to $c_{i}$ and $c\cap K\neq\emptyset$, then $|v(K\cap c_{i})|=|v(K\cap c)|$. Thus $\{v_{i}\}_{i=1}^{s}$ and $G'$ do not depend on the choice of $c_{i}$'s. 

\begin{lem}
$G'$ is a finite index subgroup of $G(\Gamma)$. 
\end{lem}

\begin{proof}
We prove this by showing $G'\cdot v(K)=G(\Gamma)$. Let $d_r$ be the syllable metric on $G(\Gamma)$ defined in Section \ref{auto of extension complexes}. Pick word $\alpha\in G(\Gamma)$ and assume $\alpha\in G'\cdot v(K)$ when $d_{r}(\alpha,\id)\le k-1$. If $d_{r}(\alpha,\id)=k$, then there exists $\beta\in G(\Gamma)$ such that $d_{r}(\id,\beta)=k-1$ and $d_{r}(\beta,\alpha)=1$. Let $\beta=\beta_{1}\beta_{2}$ for $\beta_{1}\in G'$ and $\beta_{2}\in v(K)$. Then $d_{r}(\beta_{2},\beta^{-1}_{1}\alpha)=1$. Suppose $c$ is the standard geodesic containing $\beta_{2}$ and $\beta^{-1}_{1}\alpha$. Then there exists $i$ such that $c_{i}$ and $c$ are parallel. Note that $P_{c}\cap K$ is a convex set in the parallel set $P_{c}$, hence respects the natural splitting $P_{c}=c\times c^{\perp}$, moreover, the left action of $v_{i}$ translates the $c$ factor by $n_{i}$ units and fixes the other factor. Thus there exists $d\in \Bbb Z$ and $\beta'_{2}\in K\cap c$ such that $v^{d}_{i}\beta'_{2}=\beta^{-1}_{1}\alpha$, which implies $\alpha=\beta_{1}v^{d}_{i}\beta'_{2}\in G'\cdot v(K)$.
\end{proof}

Let $\Gamma'$ be the full subgraph of $\mathcal{P}(\Gamma)$ spanned by points $\{\Delta(c_{i})\}_{i=1}^{s}$. Then there is a natural homomorphism $G(\Gamma')\to G'$.

\begin{lem}
	\label{raag subgroup}
The homomorphism $G(\Gamma')\to G'$ is actually an isomorphism. Hence $G'$ is a finite index RAAG subgroup of $G(\Gamma)$.
\end{lem}

We will follow the strategy in \cite{koberda2012ping}, where the following version of ping-pong lemma for right-angled Artin groups was used.

\begin{thm}[Theorem 4.1 of \cite{koberda2012ping}]
\label{6.9}
Let $G=G(\Gamma)$ and let $X$ be a set with a $G$-action. Suppose the following hold.
\begin{enumerate}
\item For each vertex $v_i$ of $\Gamma$, there exists subset $X_i\subset X$, such that the union of all $X_i$'s is properly contained in $X$.
\item For each nonzero $k\in \Bbb Z$ and vertices $v_{i},v_{j}$ joined by en edge, $v^{k}_{i}(X_{j})\subset X_{j}$.
\item For each nonzero $k\in \Bbb Z$ and vertices $v_{i},v_{j}$ not joined by en edge, $v^{k}_{i}(X_{j})\subset X_{i}$.
\item There exists $x_{0}\in X\setminus\cup_{i\in V}X_{i}$ ($V$ is the vertex set of $\Gamma$) such that for each nonzero $k\in\Bbb Z$, $v^{k}_{i}(x_{0})\in X_{i}$.
\end{enumerate}
Then the $G$-action is faithful.
\end{thm}

\begin{proof}[Proof of Lemma \ref{raag subgroup}]
We will apply Theorem \ref{6.9} with $X=X(\Gamma,S)$ and $G=G(\Gamma')$. For $1\le i\le s$, we identify $c_{i}$ and $\Bbb R$ in an orientation preserving way such that $\pi_{c_{i}}(\id)$ corresponds to $0\in \Bbb R$. Define $X^{+}_{i}=\pi^{-1}_{c_{i}}([n_{i}-1/2,\infty))$, $X^{-}_{i}=\pi^{-1}_{c_{i}}((-\infty,-1/2])$ and $X_{i}=X^{+}_{i}\cup X^{-}_{i}$. It clear that the identity element $\id\notin X_{i}$ for all $i$, so (1) of Theorem \ref{6.9} is true. Each $v_{i}=\alpha_{i}g^{n_{i}}_{i}\alpha^{-1}_{i}$ translates $c_{i}$ by $n_{i}$ units, so (4) is also true with $x_{0}=\id$.

If $\Delta(c_{i})$ and $\Delta(c_{j})$ are connected by an edge in $\mathcal{P}(\Gamma)$, then $v_{i}$ stabilizes every hyperplane dual to $v_{j}$, thus $v^{k}_{i}(X_{j})=X_{j}$ and (2) is true. If 
\begin{equation}
\label{6.10}
d(\Delta(c_{i}),\Delta(c_{j}))\ge 2,
\end{equation}
then $\pi_{c_{j}}(c_{i})$ is a point. Lemma \ref{6.1} and $c_{i}\cap K\neq\emptyset$ yield that $\pi_{c_{j}}(c_{i})\subset\pi_{c_{j}}(K)=c_{j}\cap K=[0,n_{j}-1]$, thus 
\begin{equation}
\label{6.11}
c_{i}\cap X_{j}=\emptyset.
\end{equation}
Similarly $c_{i}\cap X_{j}=\emptyset$. Let $h=\pi^{-1}_{c_{j}}(-1/2)$ be the boundary of $X^{-}_{j}$ and let $N_{h}$ be the carrier of $h$. Then (\ref{6.10}) implies that $h$ has empty intersection with any hyperplane dual to $c_{i}$, so is $N_{h}$. It follows from Lemma \ref{6.1} that $\pi_{c_{i}}(h)=\pi_{c_{i}}(N_{h})=p$ is a vertex in $c_{i}$. If $h_{1}=\pi^{-1}_{c_{i}}(p-1/2)$ and $h_{2}=\pi^{-1}_{c_{i}}(p+1/2)$ are two hyperplanes that pinch $p$, then $h\cap h_{k}=\emptyset$ for $k=1,2$. This and (\ref{6.11}) yield $X^{-}_{j}\cap h_{k}=\emptyset$, hence $\pi_{c_{i}}(X^{-}_{j})=p$ by Lemma \ref{6.1}. Similarly $\pi_{c_{i}}(X^{+}_{j})=p$, so $p=\pi_{c_{i}}(X_{j})=\pi_{c_{i}}(c_{j})\subset \pi_{c_{i}}(K)=c_{i}\cap K=[0,n_{i}-1]$. Note that $(\pi_{c_{i}}\circ v^{k}_{i})(X_{j})=(v^{k}_{i}\circ \pi_{c_{i}})(X_{j})=v^{k}_{i}(p)=p+kn_{i}$, so $v^{k}_{i}(X_{j})\subset X_{i}$ for $k\neq 0$. 
\end{proof}


The discussion in this subsection yields a well-defined map 
\begin{equation*}
\Theta_{S}:CN(\Gamma,S)\to \{\textmd{Finite\ index\ RAAG\ subgroups\ of\ }G(\Gamma)\}.
\end{equation*}

The images of $\Theta_{S}$ are called \textit{$S$-special subgroups of $G(\Gamma)$}. A subgroup of $G(\Gamma)$ is \textit{special} if it is $S$-special for some standard generating set $S$ of $G(\Gamma)$.

\subsection{Rigidity of RAAG subgroups}
In this subsection, we will assume $G(\Gamma')$ is finite index RAAG subgroup in $G(\Gamma)$ and $\out(G(\Gamma))$ is finite. We will show that under such condition, $G(\Gamma')$ must arise from the process described in the previous subsection. We will prove this in 3 steps. First we produce a convex subcomplex of $X(\Gamma,S)$ from $G(\Gamma')$. Then we will modify this convex subcomplex such that it is an element in $CN(\Gamma,S)$. Thus we have defined a map from finite index RAAG subgroups of $G(\Gamma)$ to elements in $CN(\Gamma,S)$. In the last step, we show the map defined in step 2 is an inverse to the map $\Theta_{S}$ defined in Section \ref{subsec_construct finite index}.

Also near the end of this subsection, we will leave several relatively long remarks which discuss relevant material in the literature. The reader can skip these remarks at first reading.

Recall that $\out(G(\Gamma))$ is finite and $\out(G(\Gamma'))$ is transvection free (Theorem \ref{5.3}), so any two standard generating sets of $G(\Gamma)$ (or $G(\Gamma')$) differ by a sequence of conjugations or partial conjugations. Then given any two standard generating sets $S$ and $S_{1}$ for $G(\Gamma)$, there is a canonical way to identify $\mathcal{P}(\Gamma,S)$ and $\mathcal{P}(\Gamma,S_{1})$ (every $S$-geodesic is Hausdorff close to an $S_{1}$-geodesic). Thus we will write $\mathcal{P}(\Gamma)$ and $\mathcal{P}(\Gamma')$ and omit the generating set.

\begin{lem}
\label{6.3}
Let $\phi,s$ be as in the discussion before Lemma \ref{5.9}. Let $l\subset X(\Gamma)$ and $l'\subset X(\Gamma')$ be standard geodesics such that $\phi(v(l))=v(l')$. Then $\phi\circ\pi_{\Delta(l)}=\pi_{\Delta(l')}\circ s$. 
\end{lem}

\begin{proof}
Pick standard geodesics $r\subset X(\Gamma)$ and $r'\subset X(\Gamma')$ such that $\phi(v(r))=v(r')$, then $s(\Delta(r))=\Delta(r')$ by Lemma \ref{4.10} (recall that $r$ is the intersection of maximal standard flats). Therefore, by the definition of $\pi_{\Delta(l)}$, it suffices to show $\phi\circ\pi_{l}(x)=\pi_{l'}\circ\phi(x)$ for any vertex $x\in X(\Gamma)$. Let $y$ be a vertex such that $y\notin l$ and let $x=\pi_{l}(y)$. By Lemma \ref{6.1}, we can approximate $\overline{xy}$ by a combinatorial geodesic $\omega$ in the 1-skeleton of $\pi_{l}^{-1}(y)$, then no hyperplane could intersect both $l$ and $\omega$. Let $\{v_{i}\}_{i=0}^{n}$ be vertices in $\omega$ such that for $0\le i\le n-1$, $[v_{i},v_{i+1}]$ is a maximal sub-segment of $\omega$ that is contained in a standard geodesic ($v_{0}=x$ and $v_{n}=y$). Denote the corresponding standard geodesic by $l_{i}$. Then $\Delta(l)\neq\Delta(l_{i})$ for all $i$. Let $u_{i}=\phi(v_{i})$ and let $l'_{i}$ be the standard geodesic such that $\phi(v(l_{i}))=v(l'_{i})$. Then $\overline{u_{i}u_{i+1}}\subset l'_{i}$ and $\Delta(l')\neq\Delta(l'_{i})$ for all $i$, thus $\pi_{l'}(l'_{i})$ is a point by Corollary \ref{3.2} and $\pi_{l'}(u_{i})=\pi_{l'}(u_{j})$ for all $1\le i,j\le n$.
\end{proof}

\emph{Step 1:} We produce a convex subcomplex of $X(\Gamma,S)$ from $G(\Gamma')$.

The left action $G(\Gamma)\curvearrowright G(\Gamma)$ induces $G(\Gamma')\curvearrowright G(\Gamma)$ and $G(\Gamma')\curvearrowright X(\Gamma,S)$. By choosing a standard generating set $S'$ of $G(\Gamma')$, we have left action $G(\Gamma')\curvearrowright X(\Gamma',S')$. For $h\in G(\Gamma')$, we use $\phi_{h}$, $\bar{\phi}_{h}$, $s_{h}$ and $\bar{s}_{h}$ to denote the action of $h$ on $G(\Gamma)$, $G(\Gamma')$, $\mathcal{P}(\Gamma)$ and $\mathcal{P}(\Gamma')$ respectively. Pick a $G(\Gamma')$-equivariant quasi-isometry $q:X(\Gamma,S)\to X(\Gamma',S')$ such that $q|_{G(\Gamma')}=\textmd{Id}$. By Theorem \ref{5.3} and Lemma \ref{4.10}, $q$ induces surjective $G(\Gamma')$-equivariant maps $\phi:G(\Gamma)\to G(\Gamma')$ and $s:\mathcal{P}(\Gamma)\to\mathcal{P}(\Gamma')$. Note that $\phi$ depends on the choice of generating set $S$ and $S'$, and this flexibility comes from the automorphism groups of $G(\Gamma)$ and $G(\Gamma')$. 

The key of step 1 is to choose a \textquotedblleft nice\textquotedblright\ standard generating set $S'$ of $G(\Gamma')$ such that $\phi$ behaves like $\varphi$ in Theorem \ref{5.12}.

\begin{lem}
	\label{change base point}
By choosing a possibly different standard generating set $S'$ for $G(\Gamma')$, we can assume the map $\phi$ satisfies $\phi(\id)=\id$, where $\id$ denotes the identity element in the corresponding group.
\end{lem}

\begin{proof}
Assume $\phi(\id)=a\neq \id$, we claim if we change the generating set from $S'$ to $aS'a^{-1}$, then the resulting $\phi$ will satisfy our requirement. By the construction of $\phi$, it suffices to show for any maximal $S'$-flat $F'_{1}$ such that $a\in F'_{1}$, there exists a maximal $aS'a^{-1}$-flat $F'_{2}$ such that $\id\in F'_{2}$ and $d_{H}(F'_{1},F'_{2})<\infty$. Let us assume $F'_{1}=\{ag^{k}\}_{k\in \Bbb Z}$ for some $g\in S'$. Then $F'_{2}=\{(aga^{-1})^{k}\}_{k\in \Bbb Z}$ would satisfy the required condition. We can prove the general case in a similar way.
\end{proof}

Pick a standard geodesic $l\subset X(\Gamma,S)$, we want to flip the order of points of $l$ in a $G(\Gamma')$-equivariant way such that (\ref{5.13}) is true. We choose an order preserving identification of $v(l)$ and $\Bbb Z$. Let $d=|\phi^{-1}(\phi(p))\cap v(l)|$ where $p$ is a vertex in $v(l)$. Let $Stab(v(l))$ be the stabilizer of $v(l)$ under the action $G(\Gamma')\curvearrowright G(\Gamma)$. By the second paragraph of the proof of Theorem \ref{5.12}, $d$ does not depend on the choice of $p$ in $v(l)$, and $Stab(v(l))$ acts on $v(l)$ in the same way as $d\Bbb Z$ acts on $\Bbb Z$ (recall that $\phi$ is $G(\Gamma')$-equivariant and the action of $G(\Gamma')$ on $G(\Gamma)$ is induced from the left action of $G(\Gamma)$ on itself). 

We will write $\chi(l)=d$. If $\bar{l}$ and $l$ are parallel, then $\chi(l)=\chi(\bar{l})$. Thus $\chi:\mathcal{P}(\Gamma)\to \Bbb Z$ is well-defined. Since $\chi(l)$ only depends on how $Stab(v(l))$ acts on $v(l)$, $\chi$ does not depend on the standard generating set $S'$. However, for any choice of $S'$, $\chi$ descends to $\chi: S'\to \Bbb Z$ by the $G(\Gamma')$-equivariance of $\phi$. 

Let $\phi(0)=a$. Then $Stab(v(l))$ is generated by $aha^{-1}$ for some $h\in S'$. By the same reasoning as Lemma \ref{change base point}, we can assume $a=\id$. Let $S'=\{h_{\lambda}\}_{\lambda\in\Lambda}$. For each $h_{\lambda}\in S'$, we associated an integer $n_{\lambda}$ as follows. If $h_{\lambda}h=hh_{\lambda}$, we set $n_{\lambda}=0$. Now we consider the case $h_{\lambda}h\neq hh_{\lambda}$. Let $l'_{\lambda}\subset X(\Gamma',S')$ be the standard geodesic that contains all powers of $h_{\lambda}$, and let $b_{\lambda}=\pi_{\Delta(l)}\circ s^{-1} (\Delta(l'_{\lambda}))$ ($\pi_{\Delta(l)}$ is the map in Lemma \ref{6.2}). Then $n_{\lambda}$ is defined to be the unique integer such that $b_{\lambda}+n_{\lambda}d\in[0,d-1]$ (recall that $d=\chi(l)$). Define $f: S'\to G(\Gamma')$ by sending $h_{\lambda}$ to $h^{n_{\lambda}}h_{\lambda}h^{-n_{\lambda}}$, then $f$ extends to an automorphism of $G(\Gamma')$ and $S''=\{f(h_{\lambda})\}_{\lambda\in \Lambda}$ is also a standard generating set. Indeed, if $\Delta(l'_{\lambda_{1}})$ and $\Delta(l'_{\lambda_{2}})$ stay in the same connected component of $\mathcal{P}(\Gamma')\setminus St(\Delta(l'))$, then $b_{\lambda_{1}}=b_{\lambda_{2}}$ by Lemma \ref{6.2}, hence $n_{\lambda_{1}}=n_{\lambda_{2}}$. It follows that $f$ can be realized as a composition of partial conjugations.

\begin{lem}
	\label{order-preserving}
We replace $S'$ by $S''$ in the definition of $\phi$ and denote the new map by $\phi_{1}$. Then $\phi_{1}|_{v(l_{1})}$ satisfies (\ref{5.13}) for any standard geodesic $l_{1}\subset X(\Gamma,S)$ with $\Delta(l_{1})\in\{s_{h}(\Delta(l))\}_{h\in G(\Gamma')}$.
\end{lem}

Recall that for any $h\in G(\Gamma')$, we use $s_h$ to denote the action of $h$ on $\mathcal{P}(\Gamma)$.

\begin{proof}
It suffices to show $\phi_{1}|_{v(l)}$ satisfies (\ref{5.13}), and the rest follows from the $G(\Gamma')$-equivariance of $\phi_1$. To show this, we only need to prove $\phi_1(i)=\id$ for any $i\in [0,d-1]$. Let $\Lambda$, $b_{\lambda}$ and $n_{\lambda}$ be as above.

We pick $i\in [0,d-1]$. Then there exists $\lambda\in \Lambda$ such that $b_{\lambda}+n_{\lambda}d=i$. By Lemma \ref{6.3}, $\phi(b_{\lambda})=\id$, hence $\phi(i)=h^{n_{\lambda}}$. Let $l_{i}$ be a standard geodesic such that $b_{\lambda}\in l_{i}$ and $d(\Delta(l_{i}),\Delta(l))\ge 2$. Then there exists $h_{\lambda'}\in S'$ with $b_{\lambda'}=b_{\lambda}$ such that $\phi(v(l_{i}))=\{h_{\lambda'}^{k}\}_{k\in \Bbb Z}$. Then $(\phi_{h})^{n_{\lambda}}(v(l_{i}))$ is an $S$-geodesic passing through $i$, and $(\phi\circ(\phi_{h})^{n_{\lambda}})(v(l_{i}))=((\bar{\phi}_{h})^{n_{\lambda}}\circ\phi)(v(l_{i}))=\{h^{n_{\lambda}}h_{\lambda'}^{k}\}_{k\in \Bbb Z}$. Note that
\begin{equation}
\label{6.4}
d_{H}(\{h^{n_{\lambda}}h_{\lambda'}^{k}\}_{k\in \Bbb Z},\{(f(h_{\lambda'}))^{k}\}_{k\in \Bbb Z})<\infty
\end{equation}
Now we look at the new map $\phi_{1}$. Note that $\phi_{1}(0)=\id$ is still true. Moreover, (\ref{6.4}) and Lemma \ref{6.3} imply $\phi_{1}(i)=\id$. Thus the lemma follows.
\end{proof} 

The next lemma basically says the above change of basis process does not affect other geodesics in an essential way.
\begin{lem}
	\label{independent}
Let $r$ be a standard geodesic in $X(\Gamma,S)$ which satisfies the condition that $\Delta(r)\notin\{s_{h}(\Delta(l))\}_{h\in G(\Gamma')}$. Pick two different vertices $x_{1},x_{2}\in r$. If $\phi(x)=\phi(y)$, then $\phi_{1}(x)=\phi_{1}(y)$.
\end{lem}

\begin{proof}
For $i=1,2$, let $r_{i}\subset X(\Gamma,S)$ be a standard geodesic containing $x_{i}$ such that $d(\Delta(r_{i}),\Delta(r))\ge 2$ for $i=1,2$. Let $r'$ (resp. $r''$) be an $S'$-geodesic (resp. $S''$-geodesic) such that $\phi(v(r))=v(r')$ (resp. $\phi_{1}(v(r))=v(r'')$). Let $\alpha=\phi(x)=\phi(y)$. Then there exist elements $h_{\lambda},h_{\lambda_{1}}$ and $h_{\lambda_{2}}$ in $S'$ such that $\phi(v(r_{i}))=\{\alpha h^{k}_{\lambda_{i}}\}_{k\in \Bbb Z}$ for $i=1,2$, and $r'=\{\alpha h^{k}_{\lambda}\}_{k\in \Bbb Z}$. Note that
\begin{equation}
\label{6.5}
h\neq h_{\lambda}, h_{\lambda_{1}}\neq h_{\lambda}, h_{\lambda_{2}}\neq h_{\lambda}
\end{equation}
Recall that $h$ is the generator of $Stab(v(l))$. The first inequality of (\ref{6.5}) follows from $\Delta(r)\notin\{s_{h}(\Delta(l))\}_{h\in G(\Gamma')}$.


It suffices to show there exist $S''$-geodesics $r''_{1}$ and $r''_{2}$ such that 
\begin{equation}
\label{6.6}
d_{H}(\phi(v(r_{i})),r''_{i})<\infty
\end{equation}
for $i=1,2$ and 
\begin{equation}
\label{6.7}
\pi_{\Delta(r'')}(\Delta(r''_{1}))=\pi_{\Delta(r'')}(\Delta(r''_{2})),
\end{equation}
then $\phi_{1}(x)=\phi_{1}(y)$ follows from Lemma \ref{6.3}. Define $r''_{i}=\{\alpha h^{-n_{\lambda_{i}}}(f(h_{\lambda_{i}}))^{k}\}_{k\in \Bbb Z}$, then (\ref{6.6}) is immediate. Note that for any $a\in r'_{1}$ and $b\in r'_{2}$, we have 
\begin{equation*}
b=a\cdot(f(h_{\lambda_{1}}))^{k_{1}}\cdot h^{n_{\lambda_{1}}-n_{\lambda_{2}}}\cdot (f(h_{\lambda_{2}}))^{k_{2}}
\end{equation*}
for some $k_{1},k_{2}\in \Bbb Z$, then (\ref{6.7}) follows from (\ref{6.5}) and the definition of $\pi_{\Delta(r'')}$. 
\end{proof}

Similarly, we can prove that if we change $\phi$ with respect to the conjugation $S'\to aS'a^{-1}$, then Lemma \ref{independent} is still true with $r$ being an arbitrary standard geodesic.

By Lemma \ref{order-preserving} and Lemma \ref{independent}, we can apply the above change-of-basis procedure for finitely many times to find appropriate standard generating set $S'$ of $G(\Gamma')$ such that the corresponding map $\phi$ satisfies (\ref{5.13}) when restricted to any standard geodesic in $X(\Gamma,S)$. By the proof of Theorem \ref{5.12}, we can extend $\phi$ to a cubical map $\phi:X(\Gamma,S)\to X(\Gamma',S')$ such that combinatorial geodesics in $C(\Gamma,S)$ are mapped to combinatorial geodesics in $C(\Gamma',S')$. Thus $\phi^{-1}(\id)$ is a combinatorially convex subcomplex. The subcomplex $\phi^{-1}(\id)$ is also compact since $\phi^{-1}(\id)$ contains finitely many vertices. Recall that combinatorial convexity in $\ell^{1}$-metric and convexity in $CAT(0)$ metric are the same for subcomplexes of $CAT(0)$ cube complexes (\cite{haglund2007isometries}), so we have constructed a compact convex subcomplex $\phi^{-1}(\id)\subset X(\Gamma,S)$ from a given finite index RAAG subgroup $G(\Gamma')\le G(\Gamma)$. 

\emph{Step 2:} We show $\phi^{-1}(\id)$ can be assumed to be an element in $CN(\Gamma,S)$.

For $K\subset G(\Gamma)$, denote the union of all standard geodesics in $X(\Gamma,S)$ that have non-trivial intersection with $K$ by $K^{\ast}$. $K$ is \textit{S-convex} if and only if $K$ is the vertex set of some convex subcomplex in $X(\Gamma,S)$. Now we return to $\phi$. By step 1, we can assume $\phi(\id)=\id$, and $\phi^{-1}(y)$ is $S$-convex for any $y\in G(\Gamma')$.

\emph{Step 2.1.} Let $\{l_{i}\}_{i=1}^{q}$ be the collection of standard geodesics passing through $\id$ and $\Lambda_{1}=\{\id\}$. Let $I:G(\Gamma)\to \ell^{1}(v(\mathcal{P}(\Gamma,S)))$ and $I_{\Delta(l)}:G(\Gamma)\to \mathbb Z^{\Delta(l)}$ be the map defined in Section \ref{subsec_construct finite index}. Since $v(l_{i})$ and $v(l_{j})$ are in different $G(\Gamma')$-orbits for $i\neq j$, by Lemma \ref{order-preserving} and Lemma \ref{independent}, we can apply the change-of-basis procedure in step 1 to find a standard generating set $S'$ for $G(\Gamma')$ such that for each $1\le i\le q$,
\begin{equation}
\label{6.8}
I^{-1}_{\Delta(l_{i})}([0,\chi(l_{i})-1])\cap v(l_{i})\subset\phi^{-1}(\id).
\end{equation}

\emph{Step 2.2.} Let $\Lambda_{2}=\Lambda^{\ast}_{1}\cap\phi^{-1}(\id)$. Pick a vertex $x\in \Lambda_{2}\setminus\Lambda_{1}$ (if such $x$ does not exist, then our process terminate). Let $l$ be a standard geodesic such that $x\in l$. If $l$ is parallel to some $l_{i}$ in step 2.1, then (\ref{6.8}) with $l_{i}$ replaced by $l$ is automatically true without any modification on $S'$, because both $I$ and $\phi$ respect the product structure of $P_{l_{i}}$. If $l$ is not parallel to any $l_{i}$, then $I_{\Delta(l)}(x)=0$. Moreover, $\Delta(l)$ is not in the $G(\Gamma')$-orbits of $\Delta(l_{i})$'s, so we can modify $S'$ as before such that both (\ref{6.8}) and $I^{-1}_{\Delta(l)}([0,\chi(l)-1])\cap v(l)\subset\phi^{-1}(\id)$ are true. We deal with other standard geodesics passing through $x$ and other points in $\Lambda_{2}\setminus\Lambda_{1}$ in a similar way. 

\emph{Step 2.3}. Let $\Lambda_{3}=\Lambda^{\ast}_{2}\cap\phi^{-1}(\id)$. For each vertex in $\Lambda_3\setminus\Lambda_2$, we repeat the procedure in step 2.2. Then we can define $\Lambda_4,\Lambda_5,\cdots$. Since $|\phi^{-1}(\id)|$ is finite and this number does not change after adjusting $S'$, our procedure must terminate after finitely many steps. Since $\phi^{-1}(\id)$ remains connected in each step, once the procedure terminates, we must have already dealt with each point in $\phi^{-1}(\id)$ and each standard geodesic passing through each point in $\phi^{-1}(\id)$. By construction, the resulting $\phi$ satisfies $\id\in\phi^{-1}(\id)$ and $I^{-1}_{\Delta(l)}([0,\chi(l)-1])\cap v(l)\subset\phi^{-1}(\id)$ for each standard geodesic $l$ which intersects $\phi^{-1}(\id)$. Thus $\phi^{-1}(\id)$ is non-negative. 

Note that the sets $\Lambda_{i}$'s actually do not depend on the map $\phi$ from step $i-1$. They only depend on the map $\chi: v(\mathcal{P}(\Gamma))\to\Bbb Z$. Thus non-negative subset $\phi^{-1}(\id)\subset G(\Gamma)$ produced above depends only on $S$ and the subgroup $G(\Gamma')\le G(\Gamma)$. Then we have a well-defined map
\begin{equation*}
\Xi_{S}:\{\textmd{Finite\ index\ RAAG subgroups\ of\ } G(\Gamma)\}\to CN(\Gamma,S)
\end{equation*}

\emph{Step 3:} We show $\Xi_{S}$ is an inverse to the map $\Theta_S$ defined in Section \ref{subsec_coherent ordering and labelling}.

First we prove $\Theta_{S}\circ\Xi_{S}=\textmd{Id}$. Let $K=\Xi_{S}(G(\Gamma'))$. Let $S'$ be the corresponding standard generating set for $G(\Gamma')$ and let $\phi: G(\Gamma)\to G(\Gamma')$ be the corresponding map. We find a maximal collection of standard geodesics $\{c_{i}\}_{i=1}^{s}$ such that $c_{i}\cap K\neq\emptyset$ for all $i$ and $\Delta(c_{i})\neq \Delta(c_{j})$ for any $i\neq j$. Let $n_{i}=\chi(c_{i})$ and let $g_{i}\in S$ be the label of edges in $c_{i}$. Suppose $\alpha_{i}=\pi_{c_{i}}(\id)$ where $\pi_{c_{i}}:X(\Gamma,S)\to c_{i}$ is the $CAT(0)$ projection. Then it suffices to prove the following lemma.

\begin{lem} $S'=\{\alpha_{i}g^{n_{i}}_{i}\alpha^{-1}_{i}\}_{i=1}^{s}$. 
\end{lem}

\begin{proof}
Pick $h\in S'$ and let $c_{h}\subset X(\Gamma',S')$ be the standard geodesic containing $\id$ and $h$. Then there exists a unique $i$ such that $\phi(v(c_{i}))=c_{h}$. To see this, let $c$ be a standard geodesic in $X(\Gamma,S)$ such that $s(\Delta(c))=\Delta(c_{h})$. Then $\phi(v(c))$ and $c_{h}$ are parallel and there exists $u\in G(\Gamma')$ which sends $\phi(v(c))$ to $v(c_{h})$. Thus $\phi\circ\phi_{u}(v(c))=v(c_{h})$ by (\ref{5.8}), where $\phi_u$ is defined in the beginning of step 1. Note that $\phi_{u}(v(c))$ has nontrivial intersection with $K$. We choose $c_{i}$ to be the geodesic parallel to $\phi_{u}(v(c))$. Then $\phi(v(c_i))=v(c_h)$.

For any standard geodesic $c'_{i}$ parallel to $c_{i}$, $\phi(c'_{i})$ is parallel to $c_{h}$, so $h\in Stab(v(\phi(c'_{i})))=Stab(v(c'_{i}))$. It follows that $\phi_{h}$ stabilizes the parallel set $P_{c_{i}}$ and acts by translation along the $c_{i}$-direction. Note that $(I_{\Delta(c_{i})}\circ\phi_{h})(x)=I_{\Delta(c_{i})}(x)+\chi(c_{i})$ for any $x\in v(P_{c_{i}})$, so $h=\phi_{h}(\id)=\alpha_{i}g^{n_{i}}_{i}\alpha^{-1}_{i}$ and the claim follows.
\end{proof}

It remains to show $\Xi_{S}\circ\Theta_{S}=\textmd{Id}$. This follows from the following result.

\begin{lem}
\label{6.13}
Let $\Gamma$ be an arbitrary finite simplicial graph. Pick a standard generating set $S$ for $G(\Gamma)$ and $K\in CN(\Gamma,S)$. Let $G(\Gamma')=\Theta_{S}(K)$ and let $S'$ be the corresponding generating set. Suppose $q:G(\Gamma)\to G(\Gamma')$ is a $G(\Gamma')$-equivariant quasi-isometry such that $q|_{G(\Gamma')}$ is the identity map. Then
\begin{enumerate}
\item $q$ induces a simplicial isomorphism $q_{\ast}:\mathcal{P}(\Gamma,S)\to\mathcal{P}(\Gamma',S')$.
\item $q_{\ast}$ induces a $G(\Gamma')$-equivariant retraction $r:G(\Gamma)\to G(\Gamma')$ such that $r$ sends every $S$-flat to an $S'$-flat.
\item $r$ extends to a surjective cubical map $r:X(\Gamma,S)\to X(\Gamma',S')$ such that $r^{-1}(\id)=K$. In particular, the vertex set of $K$ is the strict fundamental domain for the left action $G(\Gamma')\curvearrowright G(\Gamma)$.
\end{enumerate}
\end{lem}

\begin{proof}
It suffices to prove the case when $\Gamma$ does not admit a nontrivial join decomposition and $\Gamma$ is not a point.

By the construction of $\Theta_{S}$, we know the $q$-image of any $S$-flat which intersects $K$ is Hausdorff close to an $S'$-flat which contains the identity. Moreover, if the $S$-flat is maximal, then the corresponding $S'$-flat is unique. Since $G(\Gamma')\cdot v(K)=G(\Gamma)$, so the equivariance of $q$ implies the $q$-image of every $S$-flat is Hausdorff close to an $S'$-flat. Since $q$ is a quasi-isometry, so images of parallel $S$-geodesics are Hausdorff closed to each other. This induces $q_{\ast}:\mathcal{P}(\Gamma,S)\to\mathcal{P}(\Gamma',S')$. $q_{\ast}$ is injective since $q$ is a quasi-isometry and $q_{\ast}$ is surjective by the $G(\Gamma')$-equivariance.

Pick $x\in G(\Gamma)$, let $\{F_{i}\}_{i\in I}$ be the collection of maximal $S$-flats containing $x$. For each $i$, let $F'_{i}$ be the unique maximal $S'$-flat such that $d_{H}(q(F_{i}),F'_{i})<\infty$. Note that $\cap_{i\in I}F_{i}=x$ by our assumption on $\Gamma$. So $\cap_{i\in I}F'_{i}$ is either empty or one point. Note that if $x\in K$, then $\cap_{i\in I}F'_{i}=\id$. The equivariance of $q_{\ast}$ implies for every $x$, $\cap_{i\in I}F'_{i}$ is a point, which is defined to be $r(x)$. It is clear that $v(K)\subset r^{-1}(\id)$, but $|G(\Gamma):G(\Gamma')|\le |v(K)|$, so $v(K)=r^{-1}(\id)$. It follows that $v(K)$ is the strict fundamental domain for the left action of $G(\Gamma')$, and $r$ is a $G(\Gamma')$-equivariant map which maps $v(K)$ to $\id$. 

Note that $r(\id)=\id$. Then the $G(\Gamma')$-equivariance of $r$ implies $r(g)=g$ for any $g\in G(\Gamma')\subset G(\Gamma)$. Thus $r$ is a retraction. Similarly, by using the $G(\Gamma')$-equivariance of $r$, we deduce that $r$ sends every $S$-flat that intersects $K$ to an $S'$-flat passing through the identity element of $G(\Gamma')$. Thus $r$ sends every $S$-flat to an $S'$-flat by the equivariance of $r$. It is easy to see $r$ extends to a cubical map $r:X(\Gamma,S)\to X(\Gamma',S')$ such that $r^{-1}(\id)=K$.
\end{proof}

\begin{remark}
\label{6.14}
We can generalize some of the results in Lemma \ref{6.13} to infinite convex subcomplexes of $X(\Gamma,S)$. A convex subcomplex $K\subset X(\Gamma,S)$ is \textit{admissible} if for any standard geodesic $l$, the $CAT(0)$ projection $\pi_{l}(K)$ is either a finite interval or the whole $l$ (a ray is not allowed). Let $\{l_{\lambda}\}_{\lambda\in\Lambda}$ be a maximal collection of standard geodesics such that (1) $l_{\lambda}\cap K\neq\emptyset$; (2) $l_{\lambda}$ and $l_{\lambda'}$ are not parallel for $\lambda\neq\lambda'$; (3) $\pi_{l_{\lambda}}(K)$ is a finite interval. For each $l_{\lambda}$, let $\alpha_{\lambda}\in G(\Gamma)$ be an element which translates along $l_{\lambda}$ with translation length $= 1+length (\pi_{l_{\lambda}}(K))$. Let $G_{K}$ be the subgroup generated by $S'=\{\alpha_{\lambda}\}_{\lambda\in\Lambda}$. If $K$ is admissible, we can prove $G_{K}\cdot v(K)=G(\Gamma)$ as before. Moreover, for any finite subset $S'_{1}\subset S'$, the subgroup $G_{1}$ generated by $S'_{1}$ is a right-angled Artin group, and $G_{1}\hookrightarrow G_{K}$ is an isometric embedding with respect to the word metric. We can define $S'$-flat as before and view each vertex of $G_{K}$ as a $0$-dimensional $S'$-flat.

Now we show $v(K)$ is a strict fundamental domain for the action $G_{K}\curvearrowright G(\Gamma)$. It suffices to show $\alpha(K)\cap K=\emptyset$ for each nontrivial $\alpha\in G_{K}$. We can assume there is a right-angled Artin group $G_{1}$ such that $\alpha\in G_{1}\subset G_{K}$. Let $\alpha=w_{1}w_{2}\cdots w_{n}$ be a canonical form of $\alpha$ (see \cite[Section 2.3]{charney2007introduction}). Then
\begin{enumerate}
\item Each $w_{i}$ belongs to an abelian standard subgroup of $G_{1}$.
\item For each $i$, let $w_{i}=r^{k_{i,1}}_{i,1}r^{k_{i,2}}_{i,2}\cdots r^{k_{i,n_{i}}}_{i,n_{i}}$ ($r_{i,j}\in S'$). Then for each $r_{i+1,j}$ ($1\le j\le n_{i+1}$), there exists $r_{i,j'}$ which does not commute with $r_{i+1,j}$.
\end{enumerate}
We associate each generator $r_{i,j}$ with a subset $X_{i,j}\subset X(\Gamma,S)$ as in the proof of Lemma \ref{raag subgroup}, and claim there exists $j$ with $1\le j\le n_{1}$ such that $\alpha(K)\subset X_{1,j}$, then $\alpha(K)\cap K=\emptyset$ follows. We prove by induction on $n$ and assume $w_{2}w_{3}\cdots w_{n}(K)\subset X_{2,j'}$. By (2), there is $r_{1,j}$ such that $r_{1,j}$ and $r_{2,j'}$ does not commute, so $r^{k_{1,j}}_{1,j}(X_{2,j'})\subset X_{1,j}$. Moreover, by (1), $r^{k_{1,h}}_{1,h}(X_{1,j})=X_{1,j}$ for $h\neq j$, so $\alpha(K)\subset w_{1}(X_{2,j'})\subset X_{1,j}$.

Now we can define a $G_{K}$-equivariant map $r:G(\Gamma)\to G_{K}$ by sending $v(K)$ to the identity of $G_{K}$. We prove as before that $r$ maps $S$-flats to (possibly lower dimensional or 0-dimensional) $S'$-flats, thus $r$ is 1-Lipschitz with respect to the word metric. Let $i:G_{K}\hookrightarrow G(\Gamma)$ be the inclusion. Then by the equivariance of $r$, $r\circ i$ is a left translation of $G_K$. In particular, if $K$ contains the identity, then $r$ is a retraction. It follows that if $S'$ is finite, then $i$ is a quasi-isometric embedding.

Note that a related construction in the case of right-angled Coxeter groups has been discussed in \cite{haglund2008finite}. By taking larger and larger convex compact subcomplexes of $X(\Gamma,S)$, we know $G(\Gamma)$ is residually finite. Moreover, pick $\beta\in Stab(K)\subset G(\Gamma)$, by definition of $S'$, $S'=\beta S' \beta^{-1}$, so $Stab(K)$ normalize $G_{K}$. Now we have obtained a direct proof of the fact that every word-quasi-convex subgroup of a finite generated right-angled Artin group is separable (Theorem F of \cite{haglund2008finite}) by using the above discussion together with the outline in Section 1.5 of \cite{haglund2008finite}.
\end{remark}

The following result follows readily from the above discussion.

\begin{thm}
\label{6.15}
Let $G(\Gamma)$ be a RAAG with $\out(G(\Gamma))$ finite. We pick a standard generating set $S$ for $G(\Gamma)$. Then there is a 1-1 correspondence between non-negative convex compact subcomplexes of $X(\Gamma,S)$ that contain the identity and finite index RAAG subgroups of $G(\Gamma)$. In particular, these subgroups are generated by conjugates of powers of elements in $S$.
\end{thm}

In particular, Theorem \ref{1.4} in the introduction follows from Theorem \ref{6.15}.
\begin{remark}
\label{6.16}

If we drop the finite automorphism group assumption in the above theorem, then there exist a RAAG $G(\Gamma_{1})$ and its finite index RAAG subgroup $G(\Gamma_{2})$ such that $G(\Gamma_{2})$ is not isomorphic to any special subgroup of $G(\Gamma_{1})$. To see this, let $G(\Gamma_{1})$ be a right-angled Artin group such that $\out(G(\Gamma_{1}))$ is transvection free. Then Lemma \ref{6.13} and Theorem \ref{3.28} imply each special subgroup of $G(\Gamma_{1})$ does not admit non-trivial transvection in its outer automorphism group. Let $\Gamma_{1}$ and $\Gamma_{2}$ be the graphs in Example \ref{3.29}. Then $G(\Gamma_{2})$ is a right-angled Artin subgroup of $G(\Gamma_{1})$ and there are non-trivial transvections in $\out(G(\Gamma_{2}))$. Thus $G(\Gamma_{2})$ is not isometric to any special subgroup of $G(\Gamma_{1})$.
\end{remark}

\begin{remark}
\label{6.17}
Pick $G(\Gamma)$ such that $\out(G(\Gamma))$ is finite, then Theorem \ref{6.15} can be used to show certain subgroup of $G(\Gamma)$ is not a RAAG. For example, let $\{v_{i}\}_{i=1}^{k}$ be a subset of some standard generating set for $G(\Gamma)$. We define homomorphism $h:G(\Gamma)\to\Bbb Z/2$ by sending each $v_{i}$ to the non-trivial element in $\mathbb Z/2$ and killing all other generators. Then $\ker(h)$ is a RAAG if and only if $k=1$. One can compare this example to Example \ref{3.29}.
\end{remark}

\begin{remark}
\label{6.18}
It is shown in \cite[Theorem 2]{kim2013embedability} that if $F(\Gamma')$ embeds into $\mathcal{P}(\Gamma)$ as a full subcomplex, then there exists a monomorphism $G(\Gamma')\hookrightarrow G(\Gamma)$. This result can be recovered by the our previous discussion as follows. Let $\Gamma$ be an arbitrary finite simplicial graph. Let $S$ be a standard generating set for $G(\Gamma)$. For any vertex $w\in\mathcal{P}(\Gamma)$, let $\alpha_{w}\in G(\Gamma)$ be a conjugate of some element in $S$ such that $\alpha_{w}(l)=l$ for every standard geodesic $l\subset X(\Gamma,S)$ with $\Delta(l)=w$. 

Suppose $M\subset\mathcal{P}(\Gamma,S)$ is a compact full subcomplex and $\Gamma'$ is the 1-skeleton of $M$. Denote the vertex set of $M$ by $\{w_{i}\}_{i=1}^{n}$ and let $l_{i}$ be a standard geodesic with $\Delta(l_{i})=w_{i}$. We identify each $l_i$ in an orientation-preserving way with $\mathbb R$ such that $0\in \mathbb R$ is identified with $\pi_{l_i}(\id)\subset l_i$, where$\pi_{l_i}$ is the $CAT(0)$ projection to $l_i$ and $\id$ is the identity element of $G(\Gamma)$.

For $1\le i\le n$, define $\Lambda_{i}=\{1\le j\le n\mid d(w_{i},w_{j})\ge 2\}$. For each $i$, we define a pair of integers $a_i$ and $k_i$ as follows. If $\Lambda_{i}\neq\emptyset$, then let $[a_{i},a_{i}+k_{i}]\subset \mathbb R$ be the minimal interval such that $\cup_{j\in\Lambda_{i}}\pi_{l_{i}}(l_{j})\subset [a_{i},a_{i}+k_{i}]$ (recall that $l_i$ is identified with $\mathbb R$). If $\Lambda_{i}=\emptyset$, then we pick an arbitrary $a_{i}$ and set $k_{i}=0$. Define $X_{i}=\pi^{-1}_{c_{i}}((-\infty,a_{i}-1/2])\cup\pi^{-1}_{c_{i}}([a_{i}+k_{i}+1/2,\infty))$. Then by construction, $X_{i}\cap X_{j}=\emptyset$ for $i,j$ satisfying $d(w_{i},w_{j})\ge 2$. Using the argument in Section \ref{subsec_construct finite index}, we can show the subgroup generated by $S'=\{\alpha^{k_{i}+1}_{w_{i}}\}_{i=1}^{n}$ is a RAAG with defining graph $\Gamma'$.

\end{remark}

At this point it is natural to ask the following question.
\begin{que}
\label{6.19}
Let $S$ be a standard generating set of $G(\Gamma)$ and let $S'$ be a finite collection of elements of form $\alpha r^k\alpha^{-1}$, where $r\in S$, $k\in\Bbb Z$ and $\alpha\in G(\Gamma)$. Suppose $G$ is the subgroup generated by $S'$. Is $G$ a right-angled Artin group?
\end{que}

\subsection{Generalized star extension}
\label{subsec_GSE}
Our goal in this subsection is to find an algorithm to determine whether $G(\Gamma)$ and $G(\Gamma')$ are quasi-isometric or not, given $\out(G(\Gamma))$ is finite.

For convex subcomplex $E\subset X(\Gamma)$, we denote the full subcomplex in $\mathcal{P}(\Gamma,S)$ spanned by $\{\Delta(l_{\lambda})\}_{\lambda\in\Lambda}$ by $\hat{E}$, where $\{l_{\lambda}\}_{\lambda\in\Lambda}$ is the collection of standard geodesics in $X(\Gamma)$ with $l_{\lambda}\cap E\neq\emptyset$.

Now we describe a process to construct a graph $\Gamma'$ from $\Gamma$ such that $G(\Gamma')$ is isomorphic to a special subgroup of $G(\Gamma)$. Let $\Gamma_{1}=\Gamma$ and let $K_{1}$ be one point. We will construct a pair $(\Gamma_{i},K_{i})$ inductively such that
\begin{enumerate}
\item $K_{i}$ is a compact $CAT(0)$ cube complex and there is a cubical embedding $f:K_{i}\to X(\Gamma)$ such that $f(K_{i})$ is convex in $X(\Gamma)$.
\item $\Gamma_{i}$ is a finite simplicial graph and there is a simplicial isomorphism $g:F(\Gamma_{i})\to\widehat{f(K_i)}$.
\end{enumerate}
Note that these assumptions are true for $i=1$. 

We associate each edge $e\subset K_{i}$ with a vertex in $\Gamma_{i}$, denoted by $v_{e}$, as follows. Let $l_{e}$ be the standard geodesic in $X(\Gamma)$ that contains $f(e)$. We define $v_{e}:=g^{-1}(\Delta(l_{e}))$. Each vertex $x\in K_{i}$ can be associated with a full subcomplex $\Phi(x)\subset F(\Gamma_{i})$ defined by $\Phi(x)=g^{-1}(\hat{x})$.

To define $(\Gamma_{i+1},K_{i+1})$, pick a vertex $v\in\Gamma_{i}$ and let $\{x_{j}\}_{j=1}^{m}$ be the collection of vertices in $K_{i}$ such that $v\in\Phi(x_{j})$. Then $\{f(x_{j})\}_{j=1}^{k}$ are exactly the vertices in $P_{l}\cap f(K_{i})$, where $l$ is a standard geodesic such that $\Delta(l)=g(v)$. Let $L$ be the convex hull of $\{x_{j}\}_{j=1}^{m}$ in $K_{i}$. Then $e\subset L$ for any edge $e\subset K_{i}$ with $v_{e}=v$. 

Since $f(L)=P_l\cap f(K_i)$, the natural product decomposition $P_l\cong l\times l^{\perp}$ induces a product decomposition of $L=h\times [0,a]$. Note that it is possible that $a=0$, and $a>0$ if and only if there exists an edge $e\subset K_{i}$ with $v_{e}=v$. When $a>0$, $h$ is isomorphic to the hyperplane dual to $e$, and for any edge $e'\in K_{i}$ with $v_{e'}=v$, the projection of $e'$ to the interval factor $[0,a]$ is an edge. 

Let $L_{i}=h\times \{a\}\subset L$ and let $M_{i}=\cup_{x\in L_{i}}\Phi(x)$ ($x$ is a vertex). We define $F(\Gamma_{i+1})$ to be the simplicial complex obtained by gluing $F(\Gamma_{i})$ and $M_{i}$ along $St(v,M_{i})$ (see Section \ref{subsec_notation} for the notation), and define $K_{i+1}$ to be the $CAT(0)$ cube complex obtained by gluing $K_{i}$ and $L_{i}\times[0,1]$ along $L_{i}$. One readily verifies that one can extend $f$ to a cubical embedding $f':K_{i+1}\to X(\Gamma)$ such that $f'(K_{i+1})$ is convex. This also induces an isomorphism $g':F(\Gamma_{i+1})\to\hat{K}_{i+1}$ which is an extension of $g$. 

By construction, each $G(\Gamma_{i})$ is isomorphic to a special subgroup of $G(\Gamma)$, moreover, the associated convex subcomplex of this special subgroup is $K_i$. Also note that the above induction process actually does not depend on knowing what $X(\Gamma)$ is. Thus it also provides a way to construct convex subcomplexes of $X(\Gamma)$ by hand.

The above process of obtaining $(\Gamma_{i+1},K_{i+1})$ from $(\Gamma_{i},K_{i})$ is called a \textit{generalized star extension} (GSE) at $v$. Note that the following are equivalent.
\begin{enumerate}
	\item $\Gamma_{i}\subsetneq\Gamma_{i+1}$.
	\item $P_l\subsetneq X(\Gamma)$, where $l$ is the standard geodesic in $X(\Gamma)$ such that $\Delta(l)=g(v)$.
	\item $St(\pi(g(v)))\subsetneq F(\Gamma)$, where $\pi:\mathcal{P}(\Gamma)\to F(\Gamma)$ is the natural label-preserving projection defined in (\ref{projection}).
\end{enumerate}
A GSE is \textit{nontrivial} if $\Gamma_{i}\subsetneq\Gamma_{i+1}$. If $\Gamma$ is not a clique, then at each stage, there exists a vertex $v\in\Gamma_{i}$ such that the GSE at $v$ is nontrivial.

\begin{lem}
Suppose $G(\Gamma')$ is isomorphic to a special subgroup of $G(\Gamma)$. Then we can construct $\Gamma'$ from $\Gamma$ by using finitely many GSE's.	
\end{lem}

\begin{proof}
Let $\Theta_S$ and $CN(\Gamma,S)$ be the objects defined in Section \ref{subsec_construct finite index}. Suppose $G(\Gamma')$ is isomorphic to $\Theta_{S}(K)$ for $K\in CN(\Gamma,S)$. We define a sequence of convex subcomplexes in $K$ by induction. Let $K_{1}$ be the identity element in $G(\Gamma)$. Suppose $K_{i}$ is already defined. If $K_{i}=K$, then the induction terminates. If $K_{i}\subsetneq K$, pick an edge $e_{i}\subset K$ such that $e_{i}\cap K_{i}$ is a vertex and let $K_{i+1}$ be the convex hull of $K_{i}\cup e_{i}$. Let $\{K_{i}\}_{i=1}^{s}$ be the resulting collection of convex subcomplexes. An alternative way of describing $K_{i+1}$ is the following. If $h_{i}$ is the hyperplane in $K$ dual to $e_{i}$ and $N_{i}$ is the carrier of $h_{i}$ in $K$, then $h_{i}\cap K_{i}=\emptyset$ by the convexity of $K_{i}$. Thus $K_i\cap N_i$ is disjoint from $h_i$. Hence there is a copy of $(K_{i}\cap N_{i})\times[0,1]$ inside $N_{i}$, which is denoted by $M_i$. Then $K_{i+1}=K_{i}\cup M_{i}$. Now one readily verifies that one can obtain $(\hat{K}_{i+1},K_{i+1})$ from $(\hat{K}_{i},K_{i})$ by a GSE.
\end{proof}

The above construction gives rise to an algorithm to detect whether $G(\Gamma')$ is isomorphic to a special subgroup of $G(\Gamma)$. If there are $n$ vertices in $\Gamma'$, then $\Gamma'$ can be obtained from $\Gamma$ by at most $n$ nontrivial GSE's. So we can start with $\Gamma$, enumerate all possible $n$-step nontrivial GSE's from $\Gamma$, and compare each resulting graph with $\Gamma'$. By Theorem \ref{5.15} and Theorem \ref{6.15}, we have the following result.

\begin{thm}
\label{6.20}
If $\out(G(\Gamma))$ is finite, then $G(\Gamma')$ is quasi-isometric to $G(\Gamma)$ if and only if $\Gamma'$ can be obtained from $\Gamma$ by finitely many GSE's. In particular, there is an algorithm to determine whether $G(\Gamma')$ and $G(\Gamma)$ are quasi-isometric.
\end{thm}

Note that a GSE gives rise to a pair $(\Gamma_{i},K_{i})$. If one does not care about the associated convex subcomplex $K_{i}$, then there is a simpler description of GSE when $\out(G(\Gamma))$ is finite. Suppose we have already obtained $F(\Gamma_{i})$ together with a finite collection of full subcomplexes $\{G_{\lambda}\}_{\lambda\in\Lambda_{i}}$ such that
\begin{enumerate}
\item $\{G_{\lambda}\}_{\lambda\in\Lambda_{i}}$ is a covering of $F(\Gamma_{i})$.
\item Each $G_{\lambda}$ is isomorphic to $F(\Gamma)$.
\end{enumerate}
When $i=1$, we pick the trivial cover of $F(\Gamma)$ by itself. To construct $\Gamma_{i+1}$, pick vertex $v\in F(\Gamma_{i})$, let $\Lambda_{v}=\{\lambda\in\Lambda_{i}\mid v\in G_{\lambda}\}$ and let $\Gamma_{v}=\cup_{\lambda\in\Lambda_{v}}G_{\lambda}$. Suppose $\{C_{j}\}_{j=1}^{m}$ is the collection of connected components of $\Gamma_{v}\setminus St(v,\Gamma_{v})$ and suppose $C'_{j}=C_{j}\cup St(v,\Gamma_{v})$. Then $F(\Gamma_{i+1})$ is defined by gluing $C'_{1}$ and $F(\Gamma_{i})$ along $St(v,\Gamma_{v})$, and $\Gamma_{i+1}$ is the 1-skeleton of $F(\Gamma_{i+1})$.

\begin{lem}
Suppose $\out(G(\Gamma))$ is finite. Then the above simplified process is consistent with GSE.
\end{lem}

\begin{proof}
We assume inductively that there is a $CAT(0)$ cube complex $K_{i}$ such that the two induction assumptions for GSE are satisfied, moreover, $\{G_{\lambda}\}_{\lambda\in\Lambda_{i}}$ coincides with $\{\Phi(x)\}_{x\in K_{i}}$ ($x$ is a vertex). Let $L=h\times[0,a]$ be as before and let $L_{j}=h\times\{j\}\subset L$ for each integer $j\in [0,a]$. It suffices to show there is a 1-1 correspondence between $\{L_{j}\}_{j=0}^{a}$ and $\{C'_{j}\}_{j=1}^{m}$ such that for each $j$, there exists a unique $j'$ with $\widehat{f(L_{j})}=g(C'_{j'})$. Pick adjacent vertices $x_{1},x_{2}\in f(L_{j})$ and let $\bar{w}\in\Gamma$ be the label of edge $\overline{x_{1}x_{2}}$. Suppose $\bar{v}=\pi(g(v))$. Then $d(\bar{w},\bar{v})=1$. Since $\out(G(\Gamma))$ is finite, the orthogonal complement of $\bar{w}$ satisfies $\bar{w}^{\perp}\nsubseteq St(\bar{v})$. Then there is a vertex $\bar{u}\in\bar{w}^{\perp}$ such that $d(\bar{u},\bar{v})=2$. The lifts of $\bar{u}$ in $\hat{x}_{1}$ and $\hat{x}_{2}$ are the same point, so $(\hat{x}_{1}\cap\hat{x}_{2})\setminus St(g(v))$ contains a vertex. Since $F(\Gamma)$ does not have separating closed stars, $\hat{x}_{i}\setminus St(g(v))$ is connected for $i=1,2$. Thus $(\hat{x}_{1}\cap\hat{x}_{2})\setminus St(g(v))$ is connected. It follows that $\widehat{f(L_{j})}\setminus St(g(v))$ is connected. Moreover, Lemma \ref{4.8} implies that $\widehat{f(L_{j_{1}})}\setminus St(g(v))$ and $\widehat{f(L_{j_{2}})}\setminus St(g(v))$ are in different components of $\mathcal{P}(\Gamma)\setminus St(g(v))$ when $j_{1}\neq j_{2}$, so there exists a unique $j'$ such that $\widehat{f(L_{j})}=g(C'_{j'})$.
\end{proof}

\bibliographystyle{alpha}
\bibliography{1}

\end{document}